\documentclass{article}
\usepackage[english]{babel}

\usepackage{subfiles}
\usepackage{appendix}

\usepackage[backend=biber,style=numeric]{biblatex}
\usepackage{amsmath}
\usepackage{amssymb}
\usepackage{bm}
\usepackage{stmaryrd}
\usepackage{enumitem}

\usepackage{tikz-cd}
\usepackage{tikz}
\usepackage{ytableau}
\usepackage{youngtab}

\usepackage{physics}
\usepackage{faktor}

\numberwithin{equation}{section}

\usepackage{amsthm}
\theoremstyle{plain}
\newtheorem{lemma}{Lemma}
\newtheorem{proposition}[lemma]{Proposition}
\newtheorem{theorem}{Theorem}
\newtheorem{corollary}[lemma]{Corollary}
\theoremstyle{definition}
\newtheorem{definition}[lemma]{Definition}
\newtheorem{example}[lemma]{Example}
\newtheorem{remark}[lemma]{Remark}

\usepackage[backend=biber,style=numeric]{biblatex}
\renewbibmacro{in:}{}
\title{Signed Real Hurwitz numbers}
\author{Thomas Guidoni\thanks{Supported by the ERC Consolidator Grant ROGW-864919.}}
\date{}

\begin{document}
\sloppy

\maketitle

\begin{abstract}
In this paper, we enumerate with signs the holomorphic maps between Real Riemann surfaces. We relate the signs and the numbers obtained to the Real Gromov-Witten theory of the target. For a finite group $G$ with a non-trivial morphism to $\{ \pm 1 \}$, we present a signed count of principal $G$-bundles yielding a Klein Topological Quantum Field Theory. 
\end{abstract}

\tableofcontents

\section{Introduction}

Classical Hurwitz theory is the enumerative study of ramified covers of a given Riemann surface with prescribed ramification profile. It is closely related to the representation theory of the symmetric group $\mathfrak{S}_d$ : degree $d$ Hurwitz numbers admit closed formulas in terms of the irreducible characters of $\mathfrak{S}_d$. This connection extends to any finite group $G$ if one counts principal $G$-bundles instead of covers. It yields a Topological Quantum Field Theory, or equivalently the structure of a Frobenius algebra on the center $Z \mathbb{C} G$ of the group algebra of $G$.\\

From the perspective of Gromov-Witten theory, the Hurwitz numbers are the relative Gromov-Witten invariants of the target. In \cite{articleGeorgievaIonel}, Georgieva and Ionel express the relative Real Gromov-Witten invariants of Real Riemann surfaces in terms of the irreducible characters of the symmetric groups. These numbers arise as \textit{signed} counts of ramified Real covers of a given Real Riemann surface with prescribed ramification profile. Real means that both the source and the target carry an anti-holomorphic involution with which the covering map commutes. The signs come from orientation of the moduli spaces of stable Real maps and there is no explicit desciption of the sign of a given Real cover.\\

In this paper, we first define explicitly a sign of a Real ramified cover in terms of topological data of the target. In Theorem \ref{theoremComparSign}, we prove that it coincides with the sign defined abstractly in \cite{articleGeorgievaIonel}. Thus, it can be understood as the $1$-dimensional counterpart of Welschinger's signs \cite{articleWelschinger}.

The sign is defined using the monodromy representation and the sign of a permutation as follows. Let $(X,\sigma)$ be a connected Real Riemann surface such that $\sigma$ has no fixed point and $f : (X',\sigma') \rightarrow (X,\sigma)$ be a Real holomorphic map. A topological cover is obtained by taking the quotients by the involutions and removing the branch locus $B$. This yields a monodromy representation $\rho_f$ from the fundamental group of $(X \setminus B) / \sigma$ to a symmetric group. The sign of $f$ is the sign of the image by $\rho_f$ of a representative of the Poincaré dual of the first Stiefel-Whitney class of $(X \setminus B) / \sigma$. The \textit{signed Real Hurwitz numbers} are defined as the signed automorphism-weighted count of Real ramified covers of $(X,\sigma)$ with prescribed ramification profile.\\

More generally, let $G$ be a finite group and \begin{equation}
\varepsilon : G \rightarrow \{ \pm 1 \}
\end{equation} be a non-trivial group morphism. We define similarly a sign of a principal $G$-bundle over $(X \setminus B) / \sigma$ using the monodromy representation morphism of principal $G$-bundles and the morphism $\varepsilon$. The $(G,\varepsilon)$-Real Hurwitz number is obtained as the signed automorphism-weighted count of principal $G$-bundles with prescribed restriction around the punctures $B$. For the symmetric group $\mathfrak{S}_d$ with the only non-trivial morphism $\varepsilon$, the $(\mathfrak{S}_d,\varepsilon)$-Real Hurwitz numbers and the signed Real Hurwitz numbers coincide. 

When the irreducible characters of $G$ are real-valued, the $(G,\varepsilon)$-Real Hurwitz numbers do not depend on the representative of the first Stiefel-Whitney class chosen. We derive a formula in terms of the irreducible characters of $G$. This expression involves the \textit{signed Frobenius-Schur indicator} $SFS_{G,\varepsilon}$ studied in Section \ref{sectionRepresentationPairs} and reads \begin{equation}
\mathbb{R} H_{g,G,\varepsilon}^\bullet(\bm{c}) = \sum_{\rho^T = \rho} \left( SFS_{G,\varepsilon}(\rho) \frac{dim(\rho)}{\# G} \right)^{1-g} \prod_{i=1}^n f_{\bm{c}_i}(\rho)
\end{equation} where $g$ is the genus of $X$, $\rho^T$ is the representation obtained from $\rho$ by tensoring with $\varepsilon$, $\bm{c}$ is a sequence of conjugacy classes of $G$ and $f_c(\rho)$ is defined in (\ref{equationFcRho}). For the symmetric group $\mathfrak{S}_d$ with the sign morphism, the irreducible representations are labelled by the set of integer partitions $\mu$ of $d$, and we prove that the signed Frobenius-Schur indicator is given by \begin{equation}
	SFS_{\mathfrak{S}_d,\varepsilon}(\rho_\mu) = \left\{ \begin{array}{ll}
		(-1)^{\frac{d-r(\mu)}{2}} & \text{ if } \rho_\mu^T = \rho_\mu, \\
		0 & \text{ otherwise}.
	\end{array} \right.
\end{equation} where $r(\mu)$ is the length of diagonal of the Young diagram of $\mu$, see Subsection \ref{subsectionYoungDiagram}. Thus, the signed Real Hurwitz numbers are equal to the relative Real Gromov-Witten invariants studied in \cite{articleGeorgievaIonel}.

These relative Gromov-Witten invariants arise as integrals over a moduli space of pseudo-holomorphic maps whose domain has no marked point. Allowing marked points with \textit{stationnary insertions} corresponds on the Hurwitz side to consider the \textit{Hurwitz numbers with completed cycles} \cite{articleOkounkovPandharipande}. In Subsection \ref{subsectionCompletedCyclesReal}, we introduce the \textit{signed Real Hurwitz numbers with completed cycles}. It will be proved in a subsequent publication that they are the relative Real Gromov-Witten invariants with stationnary insertions.\\

The $(G,\varepsilon)$-signed Real Hurwitz numbers satisfy degeneration formulas that arise as one degenerates simulaneously pairs of conjugated circles in the target, see Proposition \ref{propositionDegene}. They are equivalent to the fact that $(G,\varepsilon)$-Real Hurwitz numbers can be assembled to form a \textit{Klein Topological Quantum Field Theory}. We describe this from the equivalent perspective of the structure of \textit{extended Frobenius algebras} \cite{articleTuraevTurner} that they define on $Z \mathbb{C} G$, which extends the standard Frobenius algebra structure on $Z \mathbb{C} G$. Thus, we obtain a map \begin{equation}
\left\{ \begin{array}{c}
\text{Non-trivial} \\
\varepsilon : G \rightarrow \{ \pm 1 \}
\end{array}  \right\} \rightarrow \left\{ \begin{array}{c}
\text{ Extensions of the} \\
\text{standard Frobenius algebra} \\
\text{structure on } Z \mathbb{C} G
\end{array}  \right\}.
\end{equation} for groups $G$ whose characters are real-valued. These extensions differ from the one obtained in \cite{articleLoktevNatanzon}. The latter would correspond in our setting to the trivial group morphism $\varepsilon$.

The degeneration of a genus $1$ Riemann surface into a pair of genus $0$ Riemann surfaces exchanged by the involution leads to a surprising combinatorial formula : \begin{equation}
\# \{ \rho \ | \ \rho^T = \rho \} = \sum_{c} \varepsilon(c)
\end{equation} where the sum is over the conjugacy classes of $G$. For the symmetric groups, the lef-hand side is the number of symmetric Young diagrams of a given size.\\

In Section \ref{sectionExtensionNonEmpty}, we consider targets $(X,\sigma)$ for which the fixed locus of $\sigma$ is not empty, and Real holomorphic maps without Real branch points. We define signs and the corresponding signed Real Hurwitz numbers, and prove that the signed Real Hurwitz numbers obtained do not depend on the involution $\sigma$ of the target. \\

The paper is organized as follows. Section \ref{sectionPremliminaries} contains preliminary results. In Section \ref{sectionRepresentationPairs}, we study the signed Frobenius-Schur indicator which is needed to express the $(G,\varepsilon)$-Real Hurwitz numbers and signed Real Hurwitz numbers. The signs and the corresponding Real Hurwitz numbers are introduced in Section \ref{sectionConstruction} for targets without fixed locus and we the formulas in terms of the irreducible characters are obtained. In Subsection \ref{subsectionDoubletContribution}, we give an expression of the \textit{doublet contributions} to the signed Real Hurwitz numbers in terms of Hurwitz numbers. Section \ref{sectionDegenerationFormulas} studies the degeneration formulas and its consequences. In Subsection \ref{subsectionDegenerationSigns}, we study the degeneration process at the more accurate level of the signs. It yields Theorem \ref{theoremComparSign} that relates the signs of Section \ref{sectionConstruction} to those defined abstractly in \cite{articleGeorgievaIonel}. The case of targets with fixed points is discussed in Section \ref{sectionExtensionNonEmpty}.

\section{Preliminaries}
\label{sectionPremliminaries}

	\subsection{Representation theory of finite groups}
	\label{subsectionRepresentationTheory}

Let $G$ be a finite group. We denote by $C(G)$ the set of conjugacy classes of $G$ and by $Irr(G)$ the set of (isomorphism classes of) irreducible representations of $G$ on a finite-dimensional complex vector space. They are finite sets of the same cardinality. The vector space \begin{equation}
Z\mathbb{C}G = \bigoplus_{c \in C(G)} \mathbb{C} \cdot c  
\end{equation} admits an algebra structure coming from the product in $G$, considering $c$ as the sum of its elements. It is commutative, associative and unital, with unit the conjugacy class $\{ e \}$ of the identity element $e$ of $G$. There exists moreover an idempotent basis $(v_\rho)$ labelled  by the set $Irr(G)$. Given $\rho \in Irr(G)$, its character $\chi_\rho : C(G) \rightarrow \mathbb{C}$ is defined by \begin{equation}
\chi_\rho (c) = \mathrm{Tr}(\rho(g)) \text{ for any } g \in c.
\end{equation} The idempotent basis consists of the elements \begin{equation}
v_\rho = \sum_{c \in C(G)} \frac{dim(\rho)}{\# G} \overline{\chi_\rho(c)} \cdot c.
\label{equationBasis1}
\end{equation} The elements of the standard basis read  \begin{equation}
c = \sum_{\rho \in Irr(G)} f_c(\rho) \cdot v_\rho.
\label{equationBasis2}
\end{equation} where \begin{equation}
f_c(\rho) = \frac{\# c}{dim(\rho)} \chi_\rho(c).
\label{equationFcRho}
\end{equation} The vector \begin{equation}
\mathfrak{K} = \sum_{c \in C(G)} z_c c \overline{c} = \sum_{\rho \in Irr(G)} \left( \frac{\# G}{dim(\rho)} \right)^2 \cdot v_\rho,
\end{equation} where $z_c = \tfrac{\# G}{ \# c}$ and $\overline{c} = \{ g^{-1}, \ g \in c \}$, plays a key role in the classical Hurwitz theory.

	\subsection{Young diagrams}
	\label{subsectionYoungDiagram}

Fix a non-negative integer $d$. In the case of the symmetric group $\mathfrak{S}_d$ of permutations of the set $\{ 1,\ldots,d \}$, the sets $C(\mathfrak{S}_d)$ and $Irr(\mathfrak{S}_d)$ are both canonically labeled by the set of \textit{integer partitions} of $d$. Such a partition is a sequence $\lambda = (\lambda_1,\ldots,\lambda_l)$ of positive integers such that $\lambda_1 + \ldots + \lambda_l = d$ and $\lambda_1 \geq \ldots \geq \lambda_l \geq 1$. The \textit{size} of the partition $\lambda$ is $| \lambda | = \lambda_1 + \ldots + \lambda_l = d$ and its \textit{length} is $l(\lambda) = l$. 

The identification between the of partitions of $d$ and $C(\mathfrak{S}_d)$ is obtained by sending $\lambda$ to the conjugacy class $c_\lambda$ of permutations whose cycle type is $\lambda$. The identification of a partition $\mu$ and an irreducible representation $\rho_\mu$ is more complicated ; see \cite[Lecture 4]{bookFultonHarris}. If the context is clear enough, we might denote by $\lambda$ and $\mu$ respectively $c_\lambda$ and $\rho_\mu$. For instance, $\chi_\mu(\lambda)$ refers to $\chi_{\rho_\mu}(c_\lambda)$ and $f_\lambda(\mu)$ to $f_{c_\lambda}(\rho_\mu)$. 

The case $d=0$ refers to the trivial group $\mathfrak{S}_0$. There is only one partition, the empty partition, which corresponds to the unit of the group and to the trivial representation.

The symmetric group $\mathfrak{S}_d$ admits a unique non-trivial morphism \begin{equation}
\varepsilon : \mathfrak{S}_d \rightarrow \{ \pm 1 \}.
\end{equation} It associates a permutation to its sign. The notation $\varepsilon(\lambda)$ refers to the sign of any permutation in the conjugacy class $c_\lambda$.\\

Partitions of $d$ can be depicted as \textit{Young diagrams}. A Young diagram is a sequence of rows of boxes, left-justified, such that the lengths are non-increasing, see Figure \ref{figureYoung}. To a partition $\lambda$, we associate the Young diagram whose $i$-th row contains $\lambda_i$ boxes. The \textit{diagonal length} $r(\lambda)$ of a partition $\lambda$ is the number $\{ i \ | \lambda_i \geq i \}$ of diagonal boxes in the associated Young diagram. The pictorial description of partitions allows for a natural involution $\mu \mapsto \mu^T$ on the set of partitions of $d$. On the level of Young diagrams, it is obtained by sending a diagram to its symmetric with respect to the diagonal. A partition is said to be \textit{symmetric} if $\mu = \mu^T$.

\begin{figure}[h!]
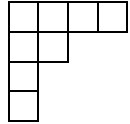

\centering
$\yng(4,2,1,1)$
\caption{Young diagram corresponding to the partition $(4,2,1,1)$ of $8$.}
\label{figureYoung}
\end{figure}

	\subsection{Completed cycles}
	\label{subsectionCompletedCycles}

In the context of Hurwitz theory, completed cycles were introduced by Okounkov and Pandharipande \cite{articleOkounkovPandharipande} to express the stationnary Gromov-Witten with descendents in terms of Hurwitz numbers.

Following \cite{articleKerovOlshanki,articleOkounkovPandharipande}, we extend the functions $f_\lambda(\mu)$ by \begin{equation}
f_\lambda(\mu ) = \binom{m_1(\lambda) + |\mu| - | \lambda |}{m_1(\lambda)} f_{(\lambda,1,\ldots,1)}(\mu)
\label{equationfExtended}
\end{equation} where $\lambda,\mu$ might not have the same size. The number $m_1(\lambda)$ is the number of parts $\lambda_i$ that equal $1$. The extended function $f_\lambda(\mu )$ vanishes for $|\mu| < | \lambda |$. We introduce the functions \begin{equation}
p_k^*(\mu) = k!c_{k+1}+ \sum_{i=1}^\infty \left( \mu_i - i + \frac{1}{2} \right)^k - \left( - i + \frac{1}{2} \right)^k
\end{equation} for positive integers $k$, where the coeffcients $(c_k)$ appear in the expansion \begin{equation}
\frac{1}{\mathcal{S}(z)} = \sum_{k=0}^\infty c_k z^k \text{ with } \mathcal{S}(z) = \frac{\sinh(z/2)}{z/2}
\end{equation} and $(k)$ is the partition of $k$ with one part. By \cite{articleKerovOlshanki}, there exists coefficients $(q_{k,\nu})$ for $|\nu | < k$ such that \begin{equation}
\frac{1}{k} p_k^* = f_{(k)} + \sum_{|\nu|} q_{k,\nu} f_{\nu}.
\label{equationCompletedCycles}
\end{equation} For instance, $q_{k,\emptyset} = (k-1)! c_{k+1}$ and we interpret the formula as $q_{k,(k)} = 1$.

\begin{definition}[\cite{articleOkounkovPandharipande}]
The \textit{$k$-th completed cycle} is the formal sum of integer partitions \begin{equation}
	\overline{(k)} = (k) + \sum_{|\nu| < k} q_{k,\nu} \nu.
\end{equation}
\end{definition}

\begin{example}[\cite{articleOkounkovPandharipande}]
The first few completed cycles are \begin{equation}
\begin{array}{l}
	\overline{(1)} = (1) - \frac{1}{24} (\emptyset) \\
	\overline{(2)} = (2) \\
	\overline{(3)} = (3) + (1,1) + \frac{1}{12} (1) + \frac{7}{2880} (\emptyset) \\
	\overline{(4)} = (4) + 2 (2,1) + \frac{5}{4} (2).
\end{array}
\end{equation}
\end{example}

By \cite{articleOkounkovPandharipande}, the functions $p_k^*$ satisfy $p_k^*(\mu^T) = (-1)^{k-1} p_k(\mu)$, while the functions $f_\nu$ satisfy $f_\nu(\mu^T) = \varepsilon(\nu) f_\nu(\mu)$. The functions $(f_\nu)$ are linearly independant, therefore $q_{k,\nu}$ vanishes unless $\varepsilon(\nu) = (-1)^{k-1} $.

	\subsection{Complex Hurwitz numbers}
	\label{subsectionClassiclaHurwitz}	

In this subsection, we review the classical Hurwitz numbers. The terminology \textit{Complex} is non-standard - we use it to stress the difference with the Real case defined in Section \ref{sectionConstruction}.

\begin{definition}[(Complex) Hurwitz numbers]
Let $g,d,n$ be non-negative integers, $X$ a closed connected genus $g$ Riemann surface and $B = \{ b_1,\ldots,b_n \}$ a set of distinct points of $X$. Given $n$ integer partitions  $\bm{\lambda} = (\bm{\lambda}_1,\ldots,\bm{\lambda}_n)$ of $d$, the \textit{connected (Complex) Hurwitz number} $H_{g,d}(\bm{\lambda})$ is the weighted-number of isomorphism classes of degree $d$ holomorphic maps $f : X' \rightarrow X$ such that \begin{enumerate}[label = (\arabic*)]
	\item $X'$ is a connected Riemann surface,
	\item The ramification profile over $b_i$ is given by the partition $\bm{\lambda}_i$ and
	\item $f$ is unramified over $X \setminus B$.
\end{enumerate} The weight of the isomorphism class $[f]$ is $1 / \# \mathrm{Aut}(f)$. The \textit{disconnected (Complex) Hurwitz number} $H_{g,d}^\bullet(\bm{\lambda})$ is defined similarly by removing the connectedness assumption in (1).
\label{definitionHurwitz}
\end{definition}

The connected and disconnected (Complex) Hurwitz numbers\footnote{In the litterature, the subscript $g$ often denote the genus of the source when the target is $\mathbb{P}^1$. This is not the convention we use in the present paper. The subscript $g$ refers to the genus of the target, and the genus of the source can be computed using the Riemann-Hurwitz formula.} are related by an exponential formula on their generating series ; see \cite{bookCavalieriMiles}. In particular, the knowledge of all the disconnected Hurwitz numbers with a given value of $g,n$ determines every connected Hurwitz number with the same $g,n$. 

Using monodromy representation, the disconnected Hurwitz numbers of degree $d$ can be expressed as the number of solutions to an equation in $\mathfrak{S}_d$. Indeed, in the basis $(c_\lambda)$, the coefficient of the unit in the product \begin{equation}
\frac{1}{d!} \mathfrak{K}^g c_{\bm{\lambda}_1} \ldots c_{\bm{\lambda}_n}
\end{equation} is $H_{g,d}^\bullet(\bm{\lambda}).$ Computing this product in the idempotent basis, one recovers the well-known formula \begin{equation}
H_{g,d}^\bullet(\bm{\lambda}) = \sum_{\mu} \left( \frac{dim(\mu)}{d!} \right)^{2 - 2g} \prod_{i=1}^n f_{\bm{\lambda}_i}(\mu)
\label{equationClosedFormulaComplex}
\end{equation} where \begin{equation}
f_\lambda( \mu ) = \frac{\# c_{\lambda}}{dim(\mu)} \chi_\mu ( \lambda ).
\end{equation} The disconnected (Complex) Hurwitz numbers can be generalized in two different ways. \begin{itemize}
	\item We can replace the group $\mathfrak{S}_d$ by any finite group $G$ and the partitions $\bm{\lambda}$ by a sequence $\bm{c}$ of conjugacy classes of $G$. It leads to numbers \begin{equation}
	H_{g,G}^\bullet(\bm{c}) = \sum_{\rho \in Irr(G)} \left( \frac{dim(\rho)}{\# G} \right)^{2-2g} \prod_{i=1}^n f_{\bm{c}_i}(\rho)
	\label{equationClosedFormulaComplexG}
	\end{equation} corresponding to counts of isomorphism classes $[P]$ of principal $G$-bundles over $X \setminus B$ whose monodromy around the $i$-th puncture belongs to the conjugacy class $\bm{c}_i$, weighted by $1 / \# Aut(P)$. We call them \textit{$G$-(Complex) Hurwitz numbers}.
	\item In the case of the symmetric groups, we can introduce completed cycles \cite{articleOkounkovPandharipande}. This leads to numbers \begin{equation}
H_{g,d}^\bullet(\bm{\lambda};\bm{k}) = \sum_{| \mu | = d} \left( \frac{dim(\mu)}{d!} \right)^{2-2g} \prod_{i=1}^n f_{\bm{\lambda}_i}(\mu) \prod_{j=1}^m \frac{p^*_{\bm{k}_j}(\mu)}{\bm{k}_j}
\label{equationHurwitzCompletedCycles}
\end{equation} in the disconnected case. These can be expressed as a weighted sum of isomorphism classes of ramified covers together with the data of \textit{distinguished subsets} in the fiber of $m$ additionnal points ; see Subsection \ref{subsectionCompletedCyclesReal} for a definition in the Real case.
\end{itemize}

	\subsection{Frobenius algebras}
	\label{subsectionFrobenius}

\begin{definition}
A \textit{Frobenius algebra} (over $\mathbb{C}$) is an associative commutative $\mathbb{C}$-algebra $(A,\star)$, with unit $e \in A$, together with a symmetric non-degenerate bilinear form $\eta$ which satisfies \begin{equation}
\eta(x \star y,z) = \eta(x, y \star z) \text{ for all } x,y,z \in A.
\end{equation}
\end{definition}

The axioms of a Frobenius algebra reflect the relations between the morphisms in the category of oriented $(1+1)$-cobordisms, see \cite{bookKock}.

\begin{example}
Let $G$ be a finite group. The center $Z \mathbb{C} G$ of the algebra of $G$ has the structure of an associative commutative and unital $\mathbb{C}$-algebra, see Subsection \ref{subsectionRepresentationTheory}. Let us denote it $(Z \mathbb{C} G,\star,e)$. The bilinear form $\eta$ sending $x,y$ to the coefficient along the unit in the standard basis of the product $x \star y$ endows $(Z \mathbb{C} G,\star,e)$ with the structure of a Frobenius algebra.

Using the relation between Frobenius algebras and (1+1) topological field theories, this Frobenius algebra induces a linear map \begin{equation}
Z_X : Z \mathbb{C} G^{\otimes m} \rightarrow Z \mathbb{C} G^{\otimes n}
\end{equation} for any Riemann surface $X$ with $m+n$ marked points, $m$ of them being understood as inputs and $n$ as outputs. The linear maps $Z_X$ where $X$ is a sphere with one input, a sphere with two inputs or a sphere with two inputs and one output is defined to be $e$, $\eta$ and $\star$ respectively. The other linear maps are obtained by degenerations of $X$, see \cite{bookKock}. If $X$ is connected of genus $g$, the linear map $Z_X$ can be expressed in the standard basis by sending $\bm{c}_1 \otimes \ldots \otimes \bm{c}_m$ to \begin{equation}
\sum_{\bm{c'}} z_{\bm{c'}_1} \ldots z_{\bm{c'}_n} H^\bullet_{g,G}(\bm{\overline{c}},\bm{c'}) \bm{c'}_1 \otimes \ldots \otimes \bm{c'}_n
\end{equation} where $z_c = \# G / \# c$ and $\overline{c} = \{h^{-1}, \ h \in c\} $ for a conjugacy class $c$ of $G$. It is straightforward to check that $\mathfrak{K}$ is obtained as $Z_X$ for a torus with one output. Another vector of interest is $\Delta = Z_X$ for a sphere with two outputs. Explicitly, \begin{equation}
\Delta = \sum_{c} z_c c \otimes \overline{c}.
\end{equation}
\end{example}

\section{Representation theory of a pair $(G,\varepsilon)$}
\label{sectionRepresentationPairs}

In this section, we study the \textit{signed Frobenius-Schur indicator} of a finite group $G$ with a non-trivial group morphism $\varepsilon : G \rightarrow \{ \pm 1 \}$, see Definition \ref{definitionSFS}. Lemma \ref{lemmaSFS} expresses the signed Frobenius-Schur indicator of $(G,\varepsilon)$ in terms of the (non-signed) Frobenius-Schur indicators of $G$ and $\ker(\varepsilon)$. This is used in Subsection \ref{subsectionSFSsymmetric} to obtain an explicit formula for the signed Frobenius-Schur indicator of the symmetric groups $\mathfrak{S}_d$, which recovers the formula obtained in \cite{articleGeorgievaIonel}.

	\subsection{Signed Frobenius-Schur indicator}
	\label{subsectionSFS}

Let $G$ be a finite group and $\varepsilon : G \rightarrow \{ \pm 1 \}$ be a non-trivial group morphism. Denote by $H$ the kernel of $\varepsilon$. An element $\tau \in G$, and by extension its conjugacy class, is said to be \textit{even} or \textit{odd} if it satisfies respectively $\varepsilon(\tau)=1$ or $\varepsilon(\tau)=-1$.

\begin{definition}[\cite{articleGeorgievaIonel}]
The \textit{signed Frobenius-Schur indicator} of the pair $(G,\varepsilon)$ is the function $SFS_{G,\varepsilon} : \{  C(G) \rightarrow \mathbb{C} \}  \rightarrow  \mathbb{C}$ defined by \begin{equation}
SFS_{G,\varepsilon} (\chi) = \frac{1}{\# G} \sum_{\tau \in G} \varepsilon (\tau) \chi(\tau^2).
\end{equation}
\label{definitionSFS}
\end{definition} In \cite{articleGeorgievaIonel}, the signed Frobenius-Schur indicator has been computed explicitly for the symmetric groups $\mathfrak{S}_d$ with the sign morphism $\varepsilon$ using the Weyl formula. The purpose of the present subsection is to provide a similar result for any pair $(G,\varepsilon)$. Since the pair $(G,\varepsilon)$ is fixed, we omit the subscript $(G,\varepsilon)$ for the rest of the subsection. We will often denote by $SFS(\rho)$ the number $SFS(\chi_\rho)$ for $\rho \in Irr(G)$.\\

In order to compute $SFS(\chi)$, let us recall the \textit{Frobenius-Schur indicator} \cite[\S 3.5]{bookFultonHarris} of $G$. It is defined on any central function $\chi$ by \begin{equation}
FS_G(\chi) = \frac{1}{\# G} \sum_{\tau \in G} \chi(\tau^2).
\end{equation} A direct computation shows that \begin{equation}
SFS(\chi) + FS_G(\chi) = FS_H(\chi). 
\label{equationSFSvsFS}
\end{equation} Thus, the knowledge of $FS_G$ and $FS_H$ is enough to determine $SFS$. It turns out that the values of the Frobenius-Schur indicator on the characters of irreducible representations are known \cite[\S 3.5]{bookFultonHarris}. \begin{enumerate}[label = (\alph*)]
	\item Suppose that the irreducible representation $\rho$ can be defined over $\mathbb{R}$. Then, $FS(\chi_\rho) = 1$. We call $\rho$ a \textit{real irreducible representation}.
	\item Suppose that the character $\chi_\rho$ of the irreducible representation $\rho$ is not real-valued. Then, $FS(\chi_\rho) = 0$. We call $\rho$ a \textit{complex irreducible representation}.
	\item Suppose that the irreducible representation $\rho$ does not satisfy (a) or (b). Then, $FS(\chi_\rho) = -1$. We call $\rho$ a \textit{quaternionic irreducible representation}.
\end{enumerate} 

In order to compute $SFS(\chi_\rho)$ for $\rho \in Irr(G)$ using (\ref{equationSFSvsFS}) and this knowledge on the Frobenius-Schur indicator, we describe how the restriction of $\rho$ to $H$ splits as a sum of irreducible representations of $H$. The morphism $\varepsilon : G \rightarrow \{ \pm 1 \} \subseteq GL(\mathbb{C})$ defines an irreducible representation on $\mathbb{C}$, that we also denote by $\varepsilon$. Its character is \begin{equation}
\chi_\varepsilon = \varepsilon.
\end{equation} Taking the tensor product with $\varepsilon$ defines an involution $\rho \mapsto \rho^T$ of $Irr(G)$. An irreducible representation is said to be \textit{symmetric} if $\rho^T$ and $\rho$ are isomorphic. Since \begin{equation}
\chi_{\rho^T}(\tau) = \varepsilon(\tau) \chi_\rho(\tau),
\end{equation} the irreducible representation $\rho$ is symmetric if and only if $\chi_\rho(\tau) = 0$ any time $\varepsilon(\tau) = -1$. By \cite[\S 5.1]{bookFultonHarris} : \begin{itemize}
	\item In the case $\rho$ is symmetric, its restriction to $H$ splits as the sum of two irreducible representations which are both real, both complex or both quaternionic.
	\item Otherwise, the restriction of $\rho$ to $H$ is irreducible.
\end{itemize}

\begin{lemma}
Let $G$ be a finite group and $\varepsilon : G \rightarrow \{ \pm 1 \}$ a non-trivial morphism. The signed Frobenius-Schur indicator $SFS(\chi_\rho)$ for $\rho \in Irr(G)$ is computed using exactly one of the following cases. \begin{enumerate}[label = (\arabic*)]
	\item $\rho$ is a real representation. \begin{enumerate}[label = (1\alph*)]
		\item If $\rho$ is symmetric and its restriction to $H$ splits as a sum of two real representations, then $SFS( \chi_\rho ) = 1$. 
		\item If $\rho$ is symmetric and its restriction to $H$ splits as a sum of two complex representations, then $SFS( \chi_\rho ) = -1$.
		\item If $\rho$ is not symmetric, then $SFS( \chi_\rho ) = 0$.
\end{enumerate}
	\item $\rho$ is a complex representation. \begin{enumerate}[label = (2\alph*)]
		\item If $\rho$ is symmetric, then $SFS( \chi_\rho ) = 0$.
		\item If $\rho$ is not symmetric, then $SFS( \chi_\rho )$ equals the Frobenius-Schur indicator of $\chi_{\rho \vert_H}$, which can take any value in $\{1,0,-1 \}$.
		\end{enumerate}
	\item $\rho$ is a quaternionic representation. \begin{enumerate}[label = (3\alph*)]
		\item If $\rho$ is symmetric and its restriction to $H$ splits as a sum of two quaternionic representations, then $SFS( \chi_\rho ) = -1$.
		\item If $\rho$ is symmetric and its restriction to $H$ splits as a sum of two complex representations, then $SFS( \chi_\rho ) = 1$.
		\item If $\rho$ is not symmetric, then $SFS( \chi_\rho ) = 0$.
\end{enumerate}
\end{enumerate}
\label{lemmaSFS}
\end{lemma}

\begin{proof}
First, we prove that $SFS(\chi_\rho)$ belongs to $\{ 1,0,-1 \}$. Denote by $S$ and $\Lambda$ the symmetric and anti-symmetric squares of $\rho$, such that $S \oplus \Lambda$ is isomorphic to $\rho \otimes \rho$. Denote by $a,b,c$ the number of copies of the representation $\varepsilon$ in the decomposition of $S,\Lambda,\rho\otimes \rho$ as a sum of irreducible representations. By \cite[Proposition 2.1]{bookFultonHarris}, we have  \begin{equation}
\chi_\rho(\tau^2) = \chi_S(\tau) - \chi_\Lambda(\tau) \text{ and }\chi_\rho(\tau)^2 =  \chi_S(\tau) + \chi_\Lambda(\tau)
\end{equation} for all $\tau$. Multiplying those identities with $\varepsilon(\tau)$ and summing over $\tau$, we obtain \begin{equation}
SFS(\chi_\rho) = a - b \text{ and } c = a + b.
\end{equation} Since $c$ can also be expressed as \begin{equation}
c = \frac{1}{\# G} \sum_{\tau \in G} \chi_{\rho}(\tau) \chi_{\rho^T}(\tau) = \delta_{\rho^*,\rho^T} \in \{ 0,1 \}, 
\end{equation} we have $a,b \in \{ 1,0 \}$ and $SFS(\chi_\rho) \in \{ 1,0,-1 \}$.

The cases (1),(2),(3) are handled similarly. We only provide the proof of the case (1). Formula (\ref{equationSFSvsFS}) then reads $SFS(\chi_\rho) = FS_H(\chi_\rho) - 1$. 

Suppose first that $\rho$ is symmetric. Then, $FS_H(\chi_\rho) \in \{ 2,0,-2 \}$. It can not take the value $-2$, otherwise we would obtain $SFS(\chi_\rho) = -3$. The two other values are the cases (1a),(1b). 

Suppose that $\rho$ is not symmetric. Its restriction to $H$ is also defined over $\mathbb{R}$, and we know that it is irreducible. Thus, $FS_H(\chi_\rho) = 1$ and $SFS(\chi_\rho) = 0$. This is the case (1c).
\end{proof}

Since the characters of irreducible representations form a basis of the space of functions $\{ C(G) \rightarrow \mathbb{C} \}$, Lemma \ref{lemmaSFS} describes entirely the signed Frobenius-Schur indicator.\\

A key role for signed Real Hurwitz numbers is played by the element \begin{equation}
\mathfrak{L} = \sum_{\tau \in G} \varepsilon(\tau) \tau^2 \in Z \mathbb{C} G.
\label{equationVectorL}
\end{equation} It is of similar importance as $\mathfrak{K}$ for Hurwitz numbers. When all the irreducible characters of $G$ are real-valued, the element $\mathfrak{L}$ posseses crucial properties described in Corollary \ref{corollaryVectorL}.
	
\begin{corollary}
Let $G$ be a finite group whose characters are real-valued and $\varepsilon : G \rightarrow \{ \pm 1 \}$ a non-trivial morphism. Then, \begin{equation}
\mathfrak{L} c = 0 \text{ in } Z \mathbb{C} G
\end{equation} for every odd conjugacy class $c$.
\label{corollaryVectorL}
\end{corollary}

\begin{proof}
In the idempotent basis, \begin{equation}
\mathfrak{L} = \sum_{\rho \in Irr(G)} \frac{\# G}{dim(\rho)} SFS_{G,\varepsilon}(\rho) \cdot v_\rho
\end{equation} so that \begin{equation}
\mathfrak{L} c = \sum_{\rho \in Irr(G)} \frac{\# G \# c}{dim(\rho)^2} SFS_{G,\varepsilon}(\rho) \chi_\rho(c) \cdot v_\rho.
\end{equation} The product $SFS_{G,\varepsilon}(\rho) \chi_\rho(c)$ vanishes whenever $c$ is odd. Indeed, the representation $\rho$ is real or quaternionic since its character is real-valued. Thus, according to Lemma \ref{lemmaSFS}, either $SFS_{G,\varepsilon}(\rho) = 0$ or $\rho$ is symmetric. In the later case, $\chi_{\rho}(c) = \varepsilon(c) \chi_{\rho}(c) = - \chi_\rho(c)$. Thus, $\mathfrak{L} c = 0$.
\end{proof} 



	\subsection{The case of the symmetric group}
	\label{subsectionSFSsymmetric}

In this subsection, we consider the symmetric group $G = \mathfrak{S}_d$ with the sign morphism $\varepsilon : \mathfrak{S}_d \rightarrow \{ \pm 1 \}$. Its kernel is the alternating group $\mathfrak{A}_d$. By \cite[\S 5.1]{bookFultonHarris}, the involutions $\rho_\mu \mapsto \left( \rho_\mu \right)^T$ and $\mu \mapsto \mu^T$ are related by \begin{equation}
\rho_{\mu^T} = \left( \rho_\mu \right)^T.
\end{equation} In particular, the representation $\rho_\mu$ is symmetric if and only if the partition $\mu$ is symmetric. 

\begin{proposition}
The signed Frobenius-Schur indicator $SFS_{\mathfrak{S}_d,\varepsilon}$ satisfies \begin{equation}
	SFS_{\mathfrak{S}_d,\varepsilon}(\mu) = \left\{ \begin{array}{ll}
		(-1)^{\frac{d-r(\mu)}{2}} & \text{ if } \mu^T = \mu, \\
		0 & \text{ otherwise}.
	\end{array} \right.
\end{equation}
\label{propositionSFSsymmetric}
\end{proposition}

\begin{proof}
All irreducible representations of $\mathfrak{S}_d$ can be defined over $\mathbb{R}$ (even over $\mathbb{Q}$). Therefore, by Lemma \ref{lemmaSFS}, $SFS_{\mathfrak{S}_d,\varepsilon}(\mu) = \pm 1$ if $\mu$ is symmetric and $SFS_{\mathfrak{S}_d,\varepsilon}(\mu) = 0$ otherwise. 

Suppose that $\mu$ is symmetric. The restriction of $\rho_\mu$ to $\mathfrak{A}_d$ splits as the sum of two irreducible representations $\rho_+,\rho_-$. The values of the characters $\chi_{\rho_+},\chi_{\rho_-}$ can be described as follows. Following \cite[\S 5.1]{bookFultonHarris}, we associate to the symmetric partition $\mu$ another partition $\lambda$ of $d$, with distinct odd parts, by requiring that $\lambda_i$ is the number of boxes below and to the right of the $i$-th diagonal box in the Young diagram of $\mu$. Explicitly, $\lambda_i = 2 (\mu_i -i ) + 1 = 2\mu_i - 2i + 1$. The conjugacy class $c_\lambda$ splits in $\mathfrak{A}_d$ as the union of two different conjugacy classes $c_+,c_-$.  

By \cite[\S 5.1]{bookFultonHarris}, the values of the characters $\chi_{\rho_+},\chi_{\rho_-}$ are the following. \begin{itemize}
	\item On a conjugacy class $c \neq c_+,c_-$, $\chi_{\rho_+}(c) = \chi_\mu(c) / 2 = \chi_{\rho_-}(c)$. In particular, both are real numbers.
	\item On the conjugacy classes $c_+,c_-$, we have \begin{equation}
	\chi_{\rho_+} (c_+) = (-1)^{\frac{d-r(\mu)}{2}} + \left( (-1)^{\frac{d-r(\mu)}{2}} \lambda_1,\ldots \lambda_l \right)^{1/2} =  \chi_{\rho_-} (c_-)
	\end{equation} and \begin{equation}
	\chi_{\rho_+} (c_-) = (-1)^{\frac{d-r(\mu)}{2}} - \left( (-1)^{\frac{d-r(\mu)}{2}} \lambda_1,\ldots \lambda_l \right)^{1/2} =  \chi_{\rho_-} (c_+).
	\end{equation}
\end{itemize} The number $d-r(\mu)$ is even for $\mu$ symmetric so that $(-1)^{\frac{d-r(\mu)}{2}} = \pm 1$. Thus, the characters $\chi_{\rho_\pm}$ are real-valued if and only if $\frac{d-r(\mu)}{2}$ is even. Lemma \ref{lemmaSFS} implies the statement.
\end{proof}

An independant proof of Proposition \ref{propositionSFSsymmetric} is provided in \cite[Lemma 9.6]{articleGeorgievaIonel}. It uses the Weyl formula for $B$-type Lie algebras. 

\section{Construction of the signed Real Hurwitz numbers}
\label{sectionConstruction}

	\subsection{Real Riemann surfaces}

\begin{definition}
A \textit{Real Riemann surface} $(X,\sigma)$ is a closed Riemann surface $X$ together with an anti-holomorphic involution $\sigma$. The pair is \textit{connected} if $X$ is connected. It is a \textit{doublet} if it is isomorphic to a pair $(Y \sqcup \overline{Y}, id)$, where $Y$ is a connected Riemann surface, $\overline{Y}$ is the same Riemann surface with the opposite complex structure, and the involution switches the two connected components. 
\end{definition}

Any Real Riemann surface is the disjoint union of connected Real Riemann surfaces and doublets. The \textit{real locus} $X^\sigma = \{x \in X \ | \  \sigma(x) = x \}$ of a Real Riemann surface is a closed $1$-dimensional smooth submanifold, \textit{ie.} a disjoint union of embedded circles. Except in Section \ref{sectionExtensionNonEmpty}, we assume \begin{equation}
X^\sigma = \emptyset.\end{equation}

Let $(X,\sigma)$ be a Riemann surface and $B \subseteq X$ a finite $\sigma$-invariant subset. It can be written as $B = \{ b_1^\pm,\ldots,b_n^\pm \}$ where $\sigma(b_i^\pm) = b_i^\mp$. The triple $(X,B,\sigma)$, together with a specified ordere $b_i^+,b_i^-$ for each $i$, will be called a \textit{marked} Real Riemann surface. Our counts will be independant of the order if the target is connected and it will differ from the Complex Hurwitz numbers by a global sign if it is a doublet, see Definition \ref{definitionRealHurwitz} and Definition \ref{definitionDoubletHurwitz}. We denote respectively by $X_\sigma$, $X^o_\sigma$ and $B_\sigma$ the quotient space of $X$, $X \setminus B$ and $B$ by $\sigma$. They are smooth manifolds. \\

In this paper, homology and cohomology are taken with coefficients in $\mathbb{Z} / 2 \mathbb{Z}$. Let $(X,B,\sigma)$ be a connected marked Riemann surface. The spaces $X_\sigma$ and $X_\sigma^o$ are both unorientable. We can describe the canonical map \begin{equation}
H_1(X_\sigma^o) \rightarrow H_1^{BM}(X_\sigma^o)
\end{equation} from homology to Borel-Moore homology as follows. Its kernel is generated by small loops around the missing points $B_\sigma$. Its cokernel consists of open paths joining two points of $B_\sigma$, \textit{modulo} the classes of closed loops. Moreover, $H_1(X_\sigma^o)$ can be described explictly as the abelianization of the fundamental group of $X_\sigma^o$. The latter admits the presentation (see Figure \ref{figurePresentation}) \begin{equation}
\pi_1 (X_\sigma^o,p) \simeq \left< \left. \begin{array}{c}
\alpha_1,\beta_1,\ldots,\alpha_h,\beta_h \\
\gamma_0,\ldots,\gamma_k \\
\delta_1 ,\ldots,\delta_n 
\end{array} \right| \begin{array}{cc}
[\alpha_1,\beta_1] \ldots [\alpha_h,\beta_h] \\
= \gamma_0^2 \ldots \gamma_k^2 \delta_1 \ldots \delta_n 
\end{array} \right>
\label{equationPresentation}
\end{equation} arising from a standard polygonal description of $X_\sigma^o$. The non-negative integers $h,k$ are related to the genus $g$ of $X$ by $g = 2h + k$. Thus, $H_1(X_\sigma^o)$ can be described as the quotient of the free $\mathbb{Z} / 2 \mathbb{Z}$-vector space generated by the classes $\alpha,\beta,\gamma,\delta$ by the relation $\delta_1 + \ldots + \delta_n = 0$, and $H_1^{BM}(X_\sigma^o)$ admits a subspace freely generated by the classes $\alpha,\beta,\gamma$.

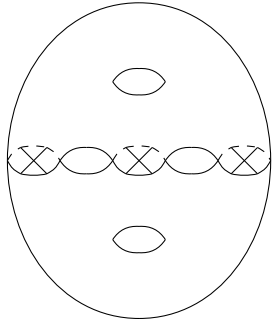
\begin{figure}[h!]
\centering
\begin{tikzpicture}[scale = 0.7]
		\node (0) at (-2.5, 0) {};
		\node (1) at (2.5, 0) {};
		\node (2) at (0, 3) {};
		\node (3) at (0, -3) {};
		\node (14) at (-1, 0.25) {};
		\node (15) at (-1, -0.25) {};
		\node (16) at (-0.5, 0) {};
		\node (17) at (-1.5, 0) {};
		\node (18) at (1, 0.25) {};
		\node (19) at (1, -0.25) {};
		\node (20) at (1.5, 0) {};
		\node (21) at (0.5, 0) {};
		\node (22) at (0, 1.75) {};
		\node (23) at (0, 1.25) {};
		\node (24) at (0.5, 1.5) {};
		\node (25) at (-0.5, 1.5) {};
		\node (26) at (0, -1.25) {};
		\node (27) at (0, -1.75) {};
		\node (28) at (0.5, -1.5) {};
		\node (29) at (-0.5, -1.5) {};
		\node (30) at (-2.5, 0) {};
		\node (31) at (-1.5, 0) {};
		\node (32) at (-2.25, 0.25) {};
		\node (33) at (-1.75, 0.25) {};
		\node (34) at (-2.25, -0.25) {};
		\node (35) at (-1.75, -0.25) {};
		\node (36) at (-0.5, 0) {};
		\node (37) at (0.5, 0) {};
		\node (38) at (-0.25, 0.25) {};
		\node (39) at (0.25, 0.25) {};
		\node (40) at (-0.25, -0.25) {};
		\node (41) at (0.25, -0.25) {};
		\node (42) at (1.5, 0) {};
		\node (43) at (2.5, 0) {};
		\node (44) at (1.75, 0.25) {};
		\node (45) at (2.25, 0.25) {};
		\node (46) at (1.75, -0.25) {};
		\node (47) at (2.25, -0.25) {};
		\draw [in=180, out=90] (0.center) to (2.center);
		\draw [in=90, out=0] (2.center) to (1.center);
		\draw [in=0, out=-90] (1.center) to (3.center);
		\draw [in=-90, out=-180] (3.center) to (0.center);
		\draw [bend left] (17.center) to (14.center);
		\draw [bend left] (14.center) to (16.center);
		\draw [bend left] (16.center) to (15.center);
		\draw [bend left] (15.center) to (17.center);
		\draw [bend left] (21.center) to (18.center);
		\draw [bend left] (18.center) to (20.center);
		\draw [bend left] (20.center) to (19.center);
		\draw [bend left] (19.center) to (21.center);
		\draw [bend left] (25.center) to (22.center);
		\draw [bend left] (22.center) to (24.center);
		\draw [bend left] (24.center) to (23.center);
		\draw [bend left] (23.center) to (25.center);
		\draw [bend left] (29.center) to (26.center);
		\draw [bend left] (26.center) to (28.center);
		\draw [bend left] (28.center) to (27.center);
		\draw [bend left] (27.center) to (29.center);
		\draw [bend left=15, looseness=1.25, dashed] (30.center) to (32.center);
		\draw [bend left=15, looseness=0.75, dashed] (32.center) to (33.center);
		\draw [bend left=15, looseness=1.25, dashed] (33.center) to (31.center);
		\draw [bend left=15, looseness=1.25] (31.center) to (35.center);
		\draw [bend left=15, looseness=0.75] (35.center) to (34.center);
		\draw [bend left=15, looseness=1.25] (34.center) to (30.center);
		\draw (32.center) to (35.center);
		\draw (34.center) to (33.center);
		\draw [bend left=15, looseness=1.25, dashed] (36.center) to (38.center);
		\draw [bend left=15, looseness=0.75, dashed] (38.center) to (39.center);
		\draw [bend left=15, looseness=1.25, dashed] (39.center) to (37.center);
		\draw [bend left=15, looseness=1.25] (37.center) to (41.center);
		\draw [bend left=15, looseness=0.75] (41.center) to (40.center);
		\draw [bend left=15, looseness=1.25] (40.center) to (36.center);
		\draw (38.center) to (41.center);
		\draw (40.center) to (39.center);
		\draw [bend left=15, looseness=1.25, dashed] (42.center) to (44.center);
		\draw [bend left=15, looseness=0.75, dashed] (44.center) to (45.center);
		\draw [bend left=15, looseness=1.25, dashed] (45.center) to (43.center);
		\draw [bend left=15, looseness=1.25] (43.center) to (47.center);
		\draw [bend left=15, looseness=0.75] (47.center) to (46.center);
		\draw [bend left=15, looseness=1.25] (46.center) to (42.center);
		\draw (44.center) to (47.center);
		\draw (46.center) to (45.center);
\end{tikzpicture}
\caption{Connected Real Riemann surface of genus $g=4$. The involution exchanges the upper and lower part and acts as the antipodal involution on the three crossed circles. The classes $\gamma_0,\gamma_1,\gamma_2$ are represented by the image in the quotient of a path along half of the crossed circles. The presentation of the fundamental group corresponding to this picture is (\ref{equationPresentation}) with $h = 1, k=2, n= 0$.}
\label{figurePresentation}
\end{figure}

\begin{lemma}
The Poincaré dual $\gamma(X_\sigma^o) \in H_1^{BM}(X_\sigma^o)$ of the first Stiefel-Whitney class $w_1(X_\sigma^o) \in H^1 (X_\sigma^o)$ is given in the presentation (\ref{equationPresentation}) by \begin{equation}
\gamma(X_\sigma^o) = \gamma_0 + \ldots + \gamma_k.
\end{equation}
\label{lemmaPoincareDual}
\end{lemma}

\begin{proof}
The cohomology class $w_1(X_\sigma^o)$ takes value $0$ on any orientation-preserving loop and $1$ on any orientation-reversing loop. Thus, under the identification $H^1 (X^o_\sigma) = H_1(X_\sigma^o)^*$, \begin{equation}
\omega_1(X_\sigma^o) = \gamma^0 + \ldots + \gamma^k
\end{equation} where we raise the labels to denote the dual basis. By a standard argument on the polygonal description, the Poincaré dual of $\gamma^i$ is the class $\gamma_i$ in $H_1^{BM}(X_\sigma^o)$. The lemma is proved. 
\end{proof}
	
\begin{definition}[Admissible class]
A homology class $a \in H_1(X_\sigma^o)$ is \textit{admissible} if its image in $H_1^{BM}(X_\sigma^o)$ is $\gamma(X_\sigma^o)$.
\label{definitionAdmissible}
\end{definition}

Having chosen a presentation (\ref{equationPresentation}) of $\pi_1(X_\sigma^o,p)$, any admissible class reads \begin{equation}
a = \gamma_0 + \ldots + \gamma_k + x_1 \delta_1 + \ldots + x_n \delta_n
\end{equation} for some $x_1,\ldots,x_n \in \mathbb{Z} / 2 \mathbb{Z}$.\\

\begin{remark}
In the presentation (\ref{equationPresentation}), the class $\gamma_i$ corresponds to a cross-cap, see Figure \ref{figurePresentation}. We can obtain another presentation by replacing the pair $\gamma_{k-1},\gamma_{k}$ of generators of the presentation by $\zeta = \gamma_{k-1} \gamma_{k}$ and $\xi = \gamma_{k}$. The presentation becomes \begin{equation}
\pi_1 (X_\sigma^o,p) \simeq \left< \left. \begin{array}{c}
\alpha_1,\beta_1,\ldots,\alpha_h,\beta_h \\
\gamma_0,\ldots,\gamma_{k-2} , \zeta,\xi \\
\delta_1 ,\ldots,\delta_n 
\end{array} \right| \begin{array}{cc}
[\alpha_1,\beta_1] \ldots [\alpha_h,\beta_h] = \\
\gamma_0^2 \ldots \gamma_{k-2}^2 \zeta \xi^{-1} \zeta \xi \delta_1 \ldots \delta_n 
\end{array} \right>.
\label{equationPresentationBis}
\end{equation} With these new generators, an admissible class reads \begin{equation}
a = \zeta + \gamma_0 + \ldots + \gamma_{k-2} + x_1 \delta_1 + \ldots x_n \delta_n
\end{equation} with $x_1,\ldots,x_n \in \mathbb{Z} / 2 \mathbb{Z}$. The relation between the classes $\zeta,\xi$ and $\gamma_{k-1},\gamma_k$ can be understood by considering different polygonal descriptions of the quotient $X_\sigma^o$. It is described in the case $k=1$ of the Klein bottle in Figure \ref{figurePolygonalDescription}. When drawing the orientable surface $X^o$ with its involution, we think of two of the middle circles as being exchanged, see Figure \ref{figurePresentationBis} in the case $h=0,k=1$. This remark will be used in Lemma \ref{lemmaFuntorialityAdm}.

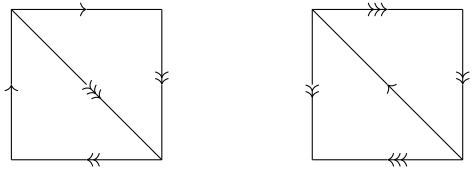
\begin{figure}[h!]
\centering
\begin{tikzpicture}
		\node (0) at (-3, 2) {};
		\node (1) at (-3, 0) {};
		\node (2) at (-3, 1) {};
		\node (3) at (-2, 2) {};
		\node (4) at (-1, 2) {};
		\node (5) at (-1, 1) {};
		\node (6) at (-1, 0) {};
		\node (7) at (-2, 0) {};
		\node (8) at (-2, 1) {};
		\node (9) at (1, 2) {};
		\node (10) at (1, 0) {};
		\node (11) at (1, 1) {};
		\node (12) at (2, 2) {};
		\node (13) at (3, 2) {};
		\node (14) at (3, 1) {};
		\node (15) at (3, 0) {};
		\node (16) at (2, 0) {};
		\node (17) at (2, 1) {};
		\draw[->] (1.center) to (2.center);
		\draw (2.center) to (0.center);
		\draw[->] (0.center) to (3.center);
		\draw (3.center) to (4.center);
		\draw[->>] (4.center) to (5.center);
		\draw (5.center) to (6.center);
		\draw[->>] (6.center) to (7.center);
		\draw (7.center) to (1.center);
		\draw[-<<<] (6.center) to (8.center);
		\draw (8.center) to (0.center);
		\draw[-<<] (10.center) to (11.center);
		\draw (11.center) to (9.center);
		\draw[->>>] (9.center) to (12.center);
		\draw (12.center) to (13.center);
		\draw[->>] (13.center) to (14.center);
		\draw (14.center) to (15.center);
		\draw[->>>] (15.center) to (16.center);
		\draw (16.center) to (10.center);
		\draw[->] (15.center) to (17.center);
		\draw (17.center) to (9.center);
\end{tikzpicture}
\caption{Two different polygonal descriptions of the Klein bottle. The edges with one, two and three arrows correspond respectively the loops $\gamma_0$, $\gamma_1 = \xi$ and $\zeta$.}
\label{figurePolygonalDescription}
\end{figure}

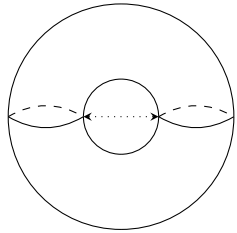
\begin{figure}[h!]
\centering
\begin{tikzpicture}
		\node (6) at (0, 6) {};
		\node (7) at (0, 5) {};
		\node (8) at (0, 3) {};
		\node (9) at (0, 4) {};
		\node (10) at (-0.5, 4.5) {};
		\node (11) at (0.5, 4.5) {};
		\node (12) at (-1.5, 4.5) {};
		\node (13) at (1.5, 4.5) {};
		\draw [bend left=45] (10.center) to (7.center);
		\draw [bend left=45] (7.center) to (11.center);
		\draw [bend left=45] (11.center) to (9.center);
		\draw [bend left=45] (9.center) to (10.center);
		\draw [bend left=45] (12.center) to (6.center);
		\draw [bend left=45] (6.center) to (13.center);
		\draw [bend left=45] (13.center) to (8.center);
		\draw [bend left=45] (8.center) to (12.center);
		\draw [bend left,dashed] (12.center) to (10.center);
		\draw [bend left,dashed] (11.center) to (13.center);
		\draw [bend right] (12.center) to (10.center);
		\draw [bend right] (11.center) to (13.center);
		\draw [stealth-stealth,dotted](10.center) to (11.center);
\end{tikzpicture}
\caption{The involution of the torus exchanges the upper and lower halves, and switches the two middle circles. In the presentation (\ref{equationPresentationBis}) with $h = 0,k=1$, the generator $\zeta$ is represented by one of the two switched circles and $\xi$ by a path from one of the switched circles to the other.}
\label{figurePresentationBis}
\end{figure}
\label{remarkPresentationBis}
\end{remark}

	\subsection{Definition of the signs and closed formulas : connected Real target}
	\label{subsectionDefinitionSign}
		
Let $(X,B,\sigma)$ be a connected marked Real Riemann surface. Given a degree $d$ Real ramified cover $f : (X',\sigma') \rightarrow (X,\sigma)$ unramified over $X \setminus B$, we obtain a degree $d$ unramified cover of $X_\sigma^o$ and therefore a monodromy representation morphism \begin{equation}
\rho_f : \pi_1 (X_\sigma^o,p) \rightarrow \mathfrak{S}_d 
\end{equation} well-defined up to conjugacy in $\mathfrak{S}_d$. Composing with the sign morphism $\varepsilon : \mathfrak{S}_d \rightarrow \{ \pm 1 \}$, it induces a well-defined map \begin{equation}
\varepsilon \rho_f : H_1(X_\sigma^o) \rightarrow \{ \pm 1 \}.
\end{equation}

\begin{definition}[Signed Real Hurwitz numbers]
Let \begin{itemize}
	\item $g,d,n$ be non-negative integers,
	\item $(X,B,\sigma)$ a connected genus $g$ marked Real Riemann surface with $B = \{ b_1^\pm,\ldots,b_n^\pm \}$,
	\item $\bm{\lambda} = (\bm{\lambda}_1,\ldots,\bm{\lambda}_n)$ a sequence of partitions of the integer $d$ and
	\item $a \in H_1(X_\sigma^o)$ an admissible class. 
\end{itemize} The \textit{connected signed Real Hurwitz number} $\mathbb{R} H_{ g,d}(\bm{\lambda})$ is the signed weighted-number of isomorphism classes of degree $d$ Real holomorphic maps $f : (X',\sigma') \rightarrow (X,\sigma)$ such that \begin{enumerate}[label = (\alph*)]
	\item $(X',\sigma')$ is a connected Real Riemann surface or a doublet,
	\item The ramification profile over $b_i^\pm$ is given by the partition $\bm{\lambda}_i$ and
	\item $f$ is unramified over $X \setminus B$.
\end{enumerate} The weight of the isomorphism class $[f]$ is $1 / \# \mathrm{Aut}(f)$ and its sign is $\varepsilon \rho_f(a)$, \textit{ie.} \begin{equation}
\mathbb{R}H_{g, d}(\bm{\lambda}) = \sum_{[f]} \frac{\varepsilon \rho_f(a)}{\# Aut(f)}.
\end{equation} The \textit{disconnected signed Real Hurwitz number} $\mathbb{R} H_{g,d}^\bullet(\bm{\lambda})$ is defined similarly by removing the connectedness assumption in (a).
\label{definitionRealHurwitz}
\end{definition}

The admissible class $a$ does not appear in the notations $\mathbb{R}H_{g, d}(\bm{\lambda})$ and $\mathbb{R}H^\bullet_{g, d }(\bm{\lambda})$. In fact, those numbers do not depend on $a$. It is proved in Lemma \ref{lemmaRealOperatorProduct}. By Definition \ref{definitionAdmissible}, the sign $\varepsilon \rho_f(a)$ does not depend on $a$ if all the partitions are even. The independence result is equivalent to the fact that the numbers $\mathbb{R}H_{g, d}(\bm{\lambda})$ and $\mathbb{R}H^\bullet_{g, d }(\bm{\lambda})$ vanish whenever a partition $\bm{\lambda}_i$ is odd.

\begin{lemma}
The number $\mathbb{R}H^\bullet_{g, d}(\bm{\lambda})$ is the coefficient of the identity in the product \begin{equation}
\frac{1}{d!} \mathfrak{K}^h \mathfrak{L}^{k+1} c_{\bm{\lambda}_1} \ldots c_{\bm{\lambda}_n} \in Z \mathbb{C} \mathfrak{S}_d
\end{equation} where $g = 2h + k$. The numbers $\mathbb{R}H_{g, d}(\bm{\lambda})$ and $\mathbb{R}H^\bullet_{g, d }(\bm{\lambda})$ do not depend on the admissible class chosen.
\label{lemmaRealOperatorProduct}
\end{lemma}
	
\begin{proof}
The standard argument in Hurwitz theory using monodromy representation and a presentation (\ref{equationPresentation}) of the fundamental group of $X_\sigma^o$ shows that $\mathbb{R}H^\bullet_{g, d}(\bm{\lambda})$ is the signed weighted number of tuples of permutations $\bm{\alpha},\bm{\beta},\bm{\gamma}\bm{\delta}$ satisfying \begin{equation}
[\alpha_1,\beta_1] \ldots [\alpha_h,\beta_h] = \gamma_0^2 \ldots \gamma_k^2 \delta_1 \ldots, \delta_n
\end{equation} and such that $\delta_i$ has cycle type $\bm{\lambda}_i$. The weight is constant equal to $1 / d!$ and the sign is \begin{equation}
\varepsilon(\gamma_0) \ldots \varepsilon(\gamma_k ) \varepsilon(\bm{\lambda}_1)^{x_1} \ldots \varepsilon(\bm{\lambda}_n)^{x_n}
\end{equation} for some coefficients $x_1,\ldots,x_n \in \{0,1\}$ defined by the admissible class $a$. This signed weighted sum is the coefficient of the identity in \begin{equation}
\frac{\varepsilon(\bm{\lambda}_1)^{x_1} \ldots \varepsilon(\bm{\lambda}_n)^{x_n}}{d!} \mathfrak{K}^h \mathfrak{L}^{k+1} c_{\bm{\lambda}_1} \ldots c_{\bm{\lambda}_n} \in Z \mathbb{C} \mathfrak{S}_d.
\end{equation} If every partition $\bm{\lambda}_i$ is even, then it is the required formula. Otherwise, Corrolary \ref{corollaryVectorL} implies that this product vanishes so that the sign $\varepsilon(\bm{\lambda}_1)^{x_1} \ldots \varepsilon(\bm{\lambda}_n)^{x_n}$ does not contribute. Thus, $\mathbb{R}H^\bullet_{g, d }(\bm{\lambda})$ does not depend on the choice of the admissible class $a$. 

In order to prove it for the connected numbers $\mathbb{R}H_{g, d}(\bm{\lambda})$, we use the relation between connected and disconnected Hurwitz numbers, see \cite[Chapter 10]{bookCavalieriMiles} in the Complex case. Choose an admissible class $a$ and introduce the generating series $\mathcal{H}$ of the connected signed Real Hurwitz of genus $g$ with $n$ pairs of marked points. It depends on formal variables $p_{i,j}$ for $i = 1,\ldots,n$ and $j \geq 1$. We define the formal series $\mathcal{H}$ as follows. Given partitions $\bm{\lambda} = (\bm{\lambda}_1,\ldots,\bm{\lambda}_n)$ of an integer $d$, $\mathbb{R}H_{g, d}(\bm{\lambda})$ is the coefficient of the monomial \begin{equation}
\prod_{i=1}^n \prod_{k=1}^{l(\bm{\lambda}_i)} p_{i,\bm{\lambda}_{i,k}}
\end{equation} in $\mathcal{H}$. We obtain similarly $\mathcal{H}^\bullet$ by considering the disconnected numbers $\mathbb{R}H^\bullet_{g, d}(\bm{\lambda})$. The two generating series are easily seen to satisfy \begin{equation}
\mathcal{H}^\bullet = \exp( \mathcal{H} ).
\label{equationConnectedDisconnected}
\end{equation} This relation can be inverted to obtain $\mathcal{H} = \log( \mathcal{H}^\bullet )$. Since the right-hand side does not depend on $a$, the left-hand side does not either.
\end{proof}

\begin{remark}
The term \textit{connected} qualifying the numbers $\mathbb{R}H_{g, d}(\bm{\lambda})$ in Definition \ref{definitionRealHurwitz} refers to the fact that it corresponds to the count of holomorphic maps for which the quotient of the domain by the involution is connected (orientable for doublet domains, non-orientable for connected domains). It is expressed by the usual relation (\ref{equationConnectedDisconnected}) between connected and disconnected theories. 

The connected signed Real Hurwitz number $\mathbb{R}H_{g, d}(\bm{\lambda})$ can be written as the sum of the \textit{connected contributions} and the \textit{doublet contributions}. In Corollary \ref{corollaryContributionDoublets}, we give a closed formula of the doublet contributions in terms of the Complex Hurwitz numbers. Since this formula does not depend on the choice of the admissible class, both the connected and doublet contributions are independent of this choice. 
\label{remarkConventionConnected}
\end{remark}


\begin{theorem}
The disconnected signed Real Hurwitz numbers are given explicitly by \begin{equation}
\mathbb{R} H_{g,d}^\bullet(\bm{\lambda}) = \sum_{\mu^T = \mu} \left( (-1)^\frac{d - r(\mu)}{2} \frac{dim(\mu)}{d!} \right)^{1-g} \prod_{i=1}^n f_{\bm{\lambda}_i}(\mu).
\end{equation} 
\label{theoremRealExplicit}
\end{theorem}

\begin{proof}
It is a computation done by expressing $\mathfrak{K},\mathfrak{L},c_\lambda$ in the idempotent basis, see (\ref{equationBasis1}) and (\ref{equationBasis2}), and using Lemma \ref{lemmaRealOperatorProduct}. We have \begin{align}
\frac{1}{d!} \mathfrak{K}^h \mathfrak{L}^{k+1} c_{\bm{\lambda}_1} \ldots c_{\bm{\lambda}_n} &= \frac{1}{d!}\sum_{| \mu | = d} \left( \frac{dim(\mu)}{d!} \right)^{- 2h - k-1} SFS(\mu)^{k+1} f_{\bm{\lambda}_1}(\mu) \ldots f_{\bm{\lambda}_n}(\mu) v_\mu \\
	&= \frac{1}{d!} \sum_{\mu^T = \mu  } \left(SFS(\mu) \frac{dim(\mu)}{d!} \right)^{-1-g } f_{\bm{\lambda}_1}(\mu) \ldots f_{\bm{\lambda}_n}(\mu) v_\mu.
\end{align} The coefficient of the identity in this expression is \begin{equation}
 \sum_{\mu^T = \mu} \left(SFS(\mu) \frac{dim(\mu)}{d!} \right)^{1-g} f_{\bm{\lambda}_1}(\mu) \ldots f_{\bm{\lambda}_n}(\mu).
\end{equation} Finally, we use Proposition \ref{propositionSFSsymmetric} to express $SFS(\mu)$.
\end{proof}

Theorem \ref{theoremRealExplicit} recovers from a different perspective the closed formulas obtained in \cite{articleGeorgievaIonel} using Real Gromov-Witten invariants.

	\subsection{Doublet contribution}
	\label{subsectionDoubletContribution}

The connected signed Real Hurwitz number $\mathbb{R}H_{g,d}(\bm{\lambda})$ is the sum of the doublet contributions and the connected contributions, defined by replacing the condition (a) of Definition \ref{definitionRealHurwitz} by respectively \begin{enumerate}[label = (\alph*')]
	\item $(X',\sigma')$ is a doublet
\end{enumerate} and \begin{enumerate}[label = (\alph*'')]
	\item $(X',\sigma')$ is a connected Real Riemann surface.
\end{enumerate} In the present subsection, we study the doublet contribution. It vanishes if the degree is odd. If $d = 2d'$, Corollary \ref{corollaryContributionDoublets} expresses it in terms of a signed sum of Complex Hurwitz numbers $H_{g,d'}(\bm{\eta})$ with $2n$ marked points. In light of Corollary \ref{corollaryContributionDoublets}, the doublet contribution turns out to be independent of the admissible class chosen, thus so is the connected contribution. 

Since modifying the admissible class by a loop around a puncture modifies the count by the sign $\varepsilon (\bm{\lambda}_i)$ of the partition corresponding to the ramification profile around the pucture, this independance property is equivalent to the vanishing of the doublet contribution and of the connected contribution whenever one of the partitions $\bm{\lambda}_i$ is odd.\\

Let $(X,B,\sigma)$ be a marked connected Real Riemann surface of genus $g$. Let \begin{equation}
f : (D,\sigma') \rightarrow (X,\sigma)
\end{equation} be a degree $d = 2d'$ ramified Real cover by a doublet, corresponding to the ramification profile $\bm{\lambda} = (\bm{\lambda}_1,\ldots,\bm{\lambda}_n)$. Choosing a connected component $X'$ of $D$, one obtains a degree $d'$ ramified cover \begin{equation}
f' = f \vert_{X'} : X' \rightarrow X.
\end{equation} Consider a cross-cap circle in $X^o$. Denote by $\mu$ the partition of $d'$ describing the monodromy of $f'$ around this cross-cap circle in $X^o$ and by $\gamma$ the loop in $X^o_\sigma$ represented in $X^o$ by a path from a point of the cross-cap circle to its opposite. 

\begin{lemma}
Let $\gamma$ and $\mu$ be as above. They satisfy \begin{equation}
\varepsilon \rho_f(\gamma) = (-1)^{l(\mu)}.
\end{equation}
\label{lemmaSignCrosscapDoublet}
\end{lemma}

\begin{proof}
By definition, \begin{equation}
\varepsilon \rho_f(\gamma) = (-1)^{d-l(\nu)}
\end{equation} where $\nu$ is the partition of $d$ corresponding to the monodromy of $f$ around the loop $\gamma$ in $X_\sigma^o$. Since $d = 2d'$ is even, we prove the Lemma by showing that $l(\nu) = l(\mu)$. 

Denote by $p$ and $q$ the starting point and endpoint of the path in $X^o$ representing $\gamma$. Choose a part $\mu_i$ of $\mu$, that is a lift through $f'$ of the cross-cap circle in $X^o$. The point $p$ has $\mu_i$ pre-images, and between any two consecutive pre-images of $p$ there is exactly one pre-image of $q$. 

In the quotient $X_\sigma^o$, $p$ and $q$ are identified. To the lift of the cross-cap circle by $f'$ corresponds a lift of the loop $\gamma$ to $D / \sigma' \simeq X'$. The point $p=q$ has $2 \mu_i$ pre-images on the this lift. Thus, the partition $\nu$ of $d = 2d'$ can be obtained from the partition $\mu$ of $d'$ as \begin{equation}
\nu = (2 \mu_1,2 \mu_2,\ldots).
\end{equation} In particular, $l(\nu) = l(\mu)$ and the Lemma is proved. 
\end{proof}

Choose a presentation of $(X,\sigma)$ with $g+1$ cross-caps, denote by $\gamma_0,\ldots,\gamma_g$ the corresponding loops in $X_\sigma^o$ and choose the admissible class $a = \gamma_0 + \ldots + \gamma_g$. The monodromy of $f' : X' \rightarrow X$ is described by $2n$ partitions $\bm{\lambda}_1^+,\bm{\lambda}_1^-,\ldots,\bm{\lambda}_n^+,\bm{\lambda}_n^-$ of $d'$ such that $\bm{\lambda}_i$ is the union of $\bm{\lambda}_i^+$ and $\bm{\lambda}_i^-$. 

The $g+1$ cross-cap circles separate $X$ in two halves. The choice of a half leads to an ordering function $s : \{ 1,\ldots,n \} \rightarrow \{ \pm \}$ such that $b_i^{s(i)}$ belongs to the chosen half for all $i$. 

\begin{proposition}
Let $f : (D,\sigma') \rightarrow (X,\sigma)$, $a = \gamma_0 + \ldots + \gamma_g$ and $\bm{\lambda}_1^+,\bm{\lambda}_1^-,\ldots,\bm{\lambda}_n^+,\bm{\lambda}_n^-$ be as above. Choose a half of $X$. Then \begin{equation}
\varepsilon \rho_f(a) = (-1)^{d'(g-1)} \prod_{i=1}^n \varepsilon \left(\bm{\lambda}_i^{s(i)} \right)
\end{equation} where $b_i^{s(i)}$ belongs to the chosen half for all $i$.
\label{propositionSignDoubletCover}
\end{proposition}

\begin{proof}
Denote by $\mu_0,\ldots,\mu_g$ the partitions of $d'$ describing the momodromy of $f'$ around the $g+1$ cross-cap circles in $X^o$. By Lemma \ref{lemmaSignCrosscapDoublet}, \begin{equation}
\varepsilon  \rho_f(a) = (-1)^{l(\mu_0) + \ldots + l(\mu_g)} = (-1)^{d'(g+1)} \varepsilon(\mu_0) \ldots \varepsilon(\mu_g).
\end{equation} In the fundamental group of $X^o$, the product of the cross-cap circles equals the product of the loops around the punctures of the chosen half. Applying the monodromy representation morphism of $f$ to this relation leads to \begin{equation}
\varepsilon  \rho_f(a) = (-1)^{l(\mu_0) + \ldots + l(\mu_g)} = (-1)^{d'(g+1)} \prod_{i=1}^n \varepsilon \left(\bm{\lambda}_i^{s(i)} \right).
\end{equation} Since $(-1)^{d'(g+1)} = (-1)^{d'(g-1)}$, the proposition is proved. 
\end{proof}

As expected, the sign expressed in Proposition \ref{propositionSignDoubletCover} does not depend on the choices made when the partitions $\bm{\lambda}$ of $d =2d'$ are even. 

\begin{corollary}
The contribution of doublet covers to $\mathbb{R}H_{g, 2d'}(\bm{\lambda})$ is \begin{equation}
\frac{(-1)^{d'(g-1)}}{2} \sum_{\bm{\lambda}_i = \bm{\lambda^+}_i \sqcup \bm{\lambda^-}_i} \prod_{i=1}^n \varepsilon \left( \bm{\lambda}_i^{s(i)} \right)  H_{g,d}(\bm{\lambda^+}, \bm{\lambda^-})
\end{equation} for any function \begin{equation}
s : \{ 1,\ldots,n\} \rightarrow \{ \pm \}.
\end{equation} It vanishes if at least one of the partitions $\bm{\lambda}_i$ is odd.
\label{corollaryContributionDoublets}
\end{corollary}

\begin{proof}
Choose a description of $(X,\sigma)$ with $g+1$ cross-cap circles as above. It is possible to require that the points $b_i^{s(i)}$ belong to the same half. Start with a ramified Real cover \begin{equation}
f : (D,\sigma') \rightarrow (X,\sigma)
\end{equation} corresponding to the ramification profile $\bm{\lambda}$. Choosing a connected component $X'$ of $D$ defines a ramified cover \begin{equation}
f : X' \rightarrow X
\end{equation} corresponding to the ramification profile $\bm{\lambda^+}$ and $\bm{\lambda^-}$ around $B^+$ and $B^-$ respectively. The choice of the other connected component leads to a ramified cover isomorphic to \begin{equation}
\sigma \circ f' : \overline{X'} \rightarrow X. 
\end{equation} In particular, its ramification profile around $B^+$ and $B^-$ is $\bm{\lambda^-}$ and $\bm{\lambda^+}$ respectively. The initial ramified Real cover can be reconstructed from either $f'$ or $\sigma \circ f'$. Thus the doublet contribution can be written as a sum over the ramified covers $f'$. 

The overall factor $1/2$ in the statement of Corollary \ref{corollaryContributionDoublets} is obtained as follows. If the ramified covers $f'$ and $\sigma \circ f'$ are not isomorphic, then \begin{equation}
\# Aut(f) = \# Aut(f') = \# Aut( \sigma \circ f')
\end{equation} and we divide by $2$ to compensate the fact that the two non-isomorphic covers $f'$ and $\sigma \circ f'$ come from a single ramified Real cover $f$. Otherwise, $f'$ and $\sigma \circ f'$ are isomorphic, and \begin{equation}
\# Aut(f) = 2 \# Aut(f').
\end{equation} The formula is proved. Notice that at this point of the proof, it depends on the presentation of $X$ chosen.

Suppose now that $\bm{\lambda}_i$ is odd. Consider a summand, which corresponds to a splitting $\bm{\lambda}_j = \bm{\lambda^+}_j \sqcup \bm{\lambda^-}_j$ for all $j$. Exactly one of $\bm{\lambda^+}_i$ and $\bm{\lambda^-}_i$ is odd and the other is even. Thus, the summand considered is the opposite of the one obtained by exchanging $\bm{\lambda^+}_i$ and $\bm{\lambda^-}_i$. Therefore, all the summands come by pairs whose contribution vanishes, and the doublet contribution vanish. 

As a consequence, the doublet contribution is well-defined - it does not depend on the admissible class chosen.
\end{proof}

	\subsection{Definition of the signs and closed formulas : doublet target}

Let $(X,B,\sigma)$ be a marked \textit{doublet}. The degree $d$ Real holomorphic maps $(X',\sigma') \rightarrow (X,\sigma)$ unramified over $X \setminus B$ correspond to the degree $d$ holomorphic maps\footnote{The quotient $X_\sigma$ is not canonically oriented - we make an arbitrary choice of orientation.} $X'_{\sigma'} \rightarrow X_\sigma$ unramified over $X_\sigma \setminus B_\sigma$. Thus, the connected and disconnected Hurwitz numbers $H_{g,d}(\bm{\lambda})$ and $H_{g,d}^\bullet(\bm{\lambda})$ are automorphism-weighted counts of Real covers of the marked doublet $(X,\sigma)$. 

In the present paper, we modify those counts by introducing signs. This is needed for the degeneration formulas to hold. 

\begin{definition}[Puncture class]
The \textit{puncture class} of the marked doublet $(X,B,\sigma)$ is the homology class $a$ of $X_\sigma^o$ defined as the sum of the simple loops $(\delta_i)$ around the punctures corresponding to the marked points $b_i = \{ b_i^\pm \}_\sigma$ for which the marked point $b_i^+$ belongs to a chosen connected component of $X$. 
\label{definitionPunctureClass}
\end{definition}

The Definition \ref{definitionPunctureClass} does not depend on the connected component of $X$ chosen since the sum of all the simple loops around the punctures of $X_\sigma^o$ is the trivial homology class.

\begin{definition}[Doublet Hurwitz numbers]
Let \begin{itemize}
	\item $g,d,n$ be non-negative integers,
	\item $(X,B,\sigma)$ the marked doublet of a connected genus $g$ Riemann surface with $B = \{ b_1^\pm,\ldots,b_n^\pm \}$,
	\item $\bm{\lambda} = (\bm{\lambda}_1,\ldots,\bm{\lambda}_n)$ a sequence of partitions of the integer $d$ and
	\item $a \in H_1(X_\sigma^o)$ be the puncture class of $(X,B,\sigma)$. 
\end{itemize} The \textit{connected Doublet Hurwitz number} $H_{ g,d,a}(\bm{\lambda})$ is the signed weighted-number of isomorphism classes of degree $d$ Real holomorphic maps $f : (X',\sigma') \rightarrow (X,\sigma)$ such that \begin{enumerate}[label = (\alph*)]
	\item $(X',\sigma')$ is a doublet,
	\item The ramification profile over $b_i^\pm$ is given by the partition $\bm{\lambda}_i$ and
	\item $f$ is unramified over $X \setminus B$.
\end{enumerate} The weight of the isomorphism class $[f]$ is $1 / \# \mathrm{Aut}(f)$ and its sign is $\varepsilon \rho_f(a)$, \textit{ie.} \begin{equation}
H_{g, d, a}(\bm{\lambda}) = \sum_{[f]} \frac{\varepsilon \rho_f(a)}{\# Aut(f)}.
\end{equation} The \textit{disconnected Doublet Hurwitz number} $H_{g,d,a}^\bullet(\bm{\lambda})$ is defined similarly by allowing the domain to be any Real Riemann surface in (a).
\label{definitionDoubletHurwitz}
\end{definition}

Denote by $I \subseteq \{ 1 ,\ldots, n \}$ the subset of the punctures of $X_\sigma^o$ defined by Definition \ref{definitionPunctureClass}, so that $a = \sum_{i \in I} \delta_i$. The Doublet Hurwitz numbers satisfy \begin{equation}
H_{g,d,a}(\bm{\lambda}) = \prod_{i \in I} \varepsilon(\bm{\lambda}_i) H_{g,d}(\bm{\lambda}) \text{ and } H_{g,d,a}^\bullet(\bm{\lambda}) = \prod_{i \in I} \varepsilon(\bm{\lambda}_i) H_{g,d}^\bullet(\bm{\lambda})
\end{equation} where the right-hand side of the equations are the Complex Hurwitz numbers, see Subsection \ref{subsectionClassiclaHurwitz}.

We call \textit{canonical marking} of the doublet $(X,\sigma)$ any ordering for which the puncture class vanishes. As we shall see in Subsection \ref{subsectionDegenerationFormula}, non-canonical markings are needed to study degenerations. 

	\subsection{Examples}

The connected signed Real Hurwitz numbers\footnote{With the definitions used in the present paper, connected refers to the morphisms whose source is either a connected Real Riemann surface or a doublet, see Remark \ref{remarkConventionConnected}.} can not be written as the coefficient of a product in $Z \mathbb{C} \mathfrak{S}_d$. Indeed, they correspond to morphisms $\pi_1(X_\sigma^o) \rightarrow \mathfrak{S}_d$ which satisfy the additional condition that the image acts transitively on $\{ 1,\ldots ,d \}$. Still, we can say that $\mathbb{R} H_{g,d}(\bm{\lambda})$ is the signed weighted number of tuples of permutations $\bm{\alpha},\bm{\beta},\bm{\gamma},\bm{\delta}$ with $\delta_i \in c_{\bm{\lambda}_i}$ for all $i$, such that \begin{equation}
[\alpha_1,\beta_1] \ldots [\alpha_h,\beta_h] = \gamma_0^2 \ldots \gamma_k^2 \delta_1 \ldots, \delta_n,
\end{equation} with $g = 2h + k$, and the subgroup generated by those permutations acts transitively on $\{ 1 ,\ldots, d \}$. The sign of such a tuple is $\varepsilon(\gamma_0\ldots \gamma_k)$ and its weight is $1 / d!$. 

\begin{example}
Let us compute the connected signed Real Hurwitz numbers corresponding to $X = \mathbb{P}^1_\mathbb{C}$ with the involution \begin{equation}
\theta : [x:y] \mapsto [- \overline{y} : \overline{x}]
\end{equation} and the set of branch points $B_{\mathbb{P}^1_\mathbb{C}} = \{ [1:0],[0:1] \}$. It corresponds to $g= 0$ and $n=1$. The monodromy representations correspond to the pairs of permutations \begin{equation}
\{ \sigma,\gamma \ | \ \sigma = \gamma^2, \ \gamma \text{ acts transitively}\}
\end{equation} and the sign is $\varepsilon(\gamma)$. There are two sequences of non-vanishing numbers. \begin{itemize}
	\item If the degree $d$ is odd, the transitivity assumption implies that both $\gamma$ and $\sigma$ are $d$-cycles and $\varepsilon(\gamma) = 1$. It corresponds to the ramified Real cover \begin{equation}
	[x:y] \mapsto [x^d : y^d].
\end{equation} Its automorphism group is $\{ [x :y ] \mapsto [x \zeta^k : y \zeta^k ], \ k \in \mathbb{Z} / d \mathbb{Z} \}$ where $\zeta^d = 1$. Thus, \begin{equation}
\mathbb{R}H_{0,d} ((d)) = 1 / d \text{ for } d \text{ odd}.
\end{equation}
	\item If the degree $2d$ is even, the transitivity assumption implies that $\gamma$ is a $2d$-cycle, $\sigma$ is the product of two $d$-cycles and $\varepsilon(\gamma) = -1$. It corresponds to the ramified Real cover whose source is the doublet $\mathbb{P}^1_\mathbb{C} \sqcup \overline{\mathbb{P}^1_\mathbb{C}}$, obtained by doubling the ramified cover \begin{equation}
	[x:y] \mapsto [x^d : y^d].
	\end{equation} Thus, \begin{equation}
\mathbb{R}H_{0,2d} ((d,d)) = -1 / 2d.
\end{equation} 
\end{itemize} These ramified Real covers are depicted in Figure \ref{figureCoverP1}.
\label{exampleP1}
\end{example}

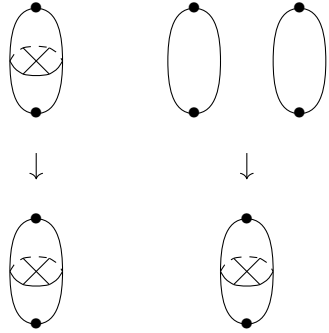
\begin{figure}[h!]
\centering
\begin{tikzpicture}[scale = 0.7]
		\node (6) at (-0.5, -1) {};
		\node (7) at (0.5, -1) {};
		\node (8) at (0, 0) {};
		\node (9) at (0, -2) {};
		\node (10) at (-0.5, 3) {};
		\node (11) at (0.5, 3) {};
		\node (12) at (0, 4) {};
		\node (13) at (0, 2) {};
		\node (14) at (3.5, -1) {};
		\node (15) at (4.5, -1) {};
		\node (16) at (4, 0) {};
		\node (17) at (4, -2) {};
		\node (18) at (2.5, 3) {};
		\node (19) at (3.5, 3) {};
		\node (20) at (3, 4) {};
		\node (21) at (3, 2) {};
		\node (22) at (4.5, 3) {};
		\node (23) at (5.5, 3) {};
		\node (24) at (5, 4) {};
		\node (25) at (5, 2) {};
		\node (32) at (-0.5, -1) {};
		\node (33) at (0.5, -1) {};
		\node (34) at (-0.25, -0.75) {};
		\node (35) at (0.25, -0.75) {};
		\node (36) at (-0.25, -1.25) {};
		\node (37) at (0.25, -1.25) {};
		\node (38) at (-0.5, 3) {};
		\node (39) at (0.5, 3) {};
		\node (40) at (-0.25, 3.25) {};
		\node (41) at (0.25, 3.25) {};
		\node (42) at (-0.25, 2.75) {};
		\node (43) at (0.25, 2.75) {};
		\node (44) at (3.5, -1) {};
		\node (45) at (4.5, -1) {};
		\node (46) at (3.75, -0.75) {};
		\node (47) at (4.25, -0.75) {};
		\node (48) at (3.75, -1.25) {};
		\node (49) at (4.25, -1.25) {};
		\node (50) at (0, 1.25) {};
		\node (52) at (0, 0.75) {};
		\node (53) at (4, 1.25) {};
		\node (54) at (4, 0.75) {};
		\node (55) at (0, 0) {$\bullet$};
		\node (56) at (0, -2) {$\bullet$};
		\node (57) at (4, 0) {$\bullet$};
		\node (58) at (4, -2) {$\bullet$};
		\node (59) at (0, 4) {$\bullet$};
		\node (60) at (0, 2) {$\bullet$};
		\node (61) at (3, 4) {$\bullet$};
		\node (62) at (3, 2) {$\bullet$};
		\node (63) at (5, 4) {$\bullet$};
		\node (64) at (5, 2) {$\bullet$};
		\draw [in=180, out=90] (6.center) to (8.center);
		\draw [in=90, out=0] (8.center) to (7.center);
		\draw [in=0, out=-90] (7.center) to (9.center);
		\draw [in=180, out=-90] (6.center) to (9.center);
		\draw [in=180, out=90] (10.center) to (12.center);
		\draw [in=90, out=0] (12.center) to (11.center);
		\draw [in=0, out=-90] (11.center) to (13.center);
		\draw [in=180, out=-90] (10.center) to (13.center);
		\draw [in=180, out=90] (14.center) to (16.center);
		\draw [in=90, out=0] (16.center) to (15.center);
		\draw [in=0, out=-90] (15.center) to (17.center);
		\draw [in=180, out=-90] (14.center) to (17.center);
		\draw [in=180, out=90] (18.center) to (20.center);
		\draw [in=90, out=0] (20.center) to (19.center);
		\draw [in=0, out=-90] (19.center) to (21.center);
		\draw [in=180, out=-90] (18.center) to (21.center);
		\draw [in=180, out=90] (22.center) to (24.center);
		\draw [in=90, out=0] (24.center) to (23.center);
		\draw [in=0, out=-90] (23.center) to (25.center);
		\draw [in=180, out=-90] (22.center) to (25.center);
		\draw [bend left=15, looseness=1.25, dashed] (32.center) to (34.center);
		\draw [bend left=15, looseness=0.75, dashed] (34.center) to (35.center);
		\draw [bend left=15, looseness=1.25, dashed] (35.center) to (33.center);
		\draw [bend left=15, looseness=1.25] (33.center) to (37.center);
		\draw [bend left=15, looseness=0.75] (37.center) to (36.center);
		\draw [bend left=15, looseness=1.25] (36.center) to (32.center);
		\draw (34.center) to (37.center);
		\draw (36.center) to (35.center);
		\draw [bend left=15, looseness=1.25, dashed] (38.center) to (40.center);
		\draw [bend left=15, looseness=0.75, dashed] (40.center) to (41.center);
		\draw [bend left=15, looseness=1.25, dashed] (41.center) to (39.center);
		\draw [bend left=15, looseness=1.25] (39.center) to (43.center);
		\draw [bend left=15, looseness=0.75] (43.center) to (42.center);
		\draw [bend left=15, looseness=1.25] (42.center) to (38.center);
		\draw (40.center) to (43.center);
		\draw (42.center) to (41.center);
		\draw [bend left=15, looseness=1.25, dashed] (44.center) to (46.center);
		\draw [bend left=15, looseness=0.75, dashed] (46.center) to (47.center);
		\draw [bend left=15, looseness=1.25, dashed] (47.center) to (45.center);
		\draw [bend left=15, looseness=1.25] (45.center) to (49.center);
		\draw [bend left=15, looseness=0.75] (49.center) to (48.center);
		\draw [bend left=15, looseness=1.25] (48.center) to (44.center);
		\draw (46.center) to (49.center);
		\draw (48.center) to (47.center);
		\draw[->] (50.center) to (52.center);
		\draw[->] (53.center) to (54.center);
\end{tikzpicture}
\caption{The two types of ramified Real covers described in Example \ref{exampleP1}. On the left-hand side, the degree $d$ is odd. On the right-hand side, the involution on the source exchanges the two connected components and the degree can be $2d$ for any $d$.}
\label{figureCoverP1}
\end{figure}

\begin{example}
The vanishing of $\mathbb{R}H_{g,d}(\bm{\lambda})$ does not imply that there is no ramified Real cover. For instance, let us describe the covers contributing to \begin{equation}
\mathbb{R}H_{0,2}((2),(2)) = 0.
\end{equation} We take the base to be the pair $(\mathbb{P}^1_\mathbb{C},\theta)$ as in Example \ref{exampleP1}, with the branch points $[1:0],[0:1],[1:w],[-\overline{w}:1]$ for a given $w \in \mathbb{C}^*$. Forgetting about the anti-holomorphic involution, there is up to isomorphism only one such ramified cover. The source $E$ is the compactification of the plane curve \begin{equation}
\{ y^2 = x (x-w)(x+1/ \overline{w} )\}
\end{equation} by a point $\infty$ at $1/y = 0$. The cover is defined by \begin{equation}
f : \left\{ \begin{array}{lll}
	(x,y) & \mapsto & [1:x] \\
	\infty & \mapsto & [0:1].
\end{array} \right.
\end{equation} It has one non-trivial automorphism $(x,y) \mapsto (x,-y)$, so that one recovers \begin{equation}
H_{0,2}((2),(2),(2),(2)) = \frac{1}{2}.
\end{equation} There are two different anti-holomorphic involutions $\sigma_\pm$ which can be given to $E$ in order to make $f$ a Real ramified cover. The set $\{(x,y) \in E \ | \ |x | = 1  \}$ is the union of two circles. The involution $\sigma_+$ preserves the two components while $\sigma_-$ switches them, see Figure \ref{figureCoverElliptic}.

Represent the admissible class by the path \begin{equation}
\gamma : t \in [0,\pi] \mapsto [1 : \exp(it)]
\end{equation} in $\mathbb{P}^1_\mathbb{C}$. Denote by $p,q$ the starting point and end point of $\gamma$ and by $p_0,p_1,q_0,q_1$ their preimages by $f$, where $p_i,q_i$ belong to the same component of $E$. The lift of $\gamma$ joins $p_i$ to $q_i$. With the involution $\sigma_+$, $q_i = \sigma_+(p_i)$, thus the permutation associated to $\gamma$ is the identity. With the involution $\sigma_-$, $q_0 = \sigma_- (p_1)$, thus the permutation associated to $\gamma$ is a transposition. Therefore, $f_+$ and $f_-$ contribute respectively with the sign $+1$ and $-1$. Since the automorphism of $f$ is compatible with $\sigma_+$ and $\sigma_-$, we recover \begin{equation}
\mathbb{R}H_{0,2}((2),(2)) = \frac{1}{2} - \frac{1}{2} = 0.
\end{equation}
\label{exampleHyperelliptic}
\end{example}	

\begin{figure}[h!]
\centering
\begin{tikzpicture}[scale = 0.7]
		\node (0) at (-0.25, 1) {};
		\node (1) at (0.25, 1) {};
		\node (2) at (-0.25, -1) {};
		\node (3) at (0.25, -1) {};
		\node (4) at (-0.5, 0) {};
		\node (5) at (0.5, 0) {};
		\node (6) at (0, 6) {};
		\node (7) at (0, 5) {};
		\node (8) at (0, 3) {};
		\node (9) at (0, 4) {};
		\node (10) at (-0.5, 4.5) {};
		\node (11) at (0.5, 4.5) {};
		\node (12) at (-1.5, 4.5) {};
		\node (13) at (1.5, 4.5) {};
		\node (14) at (3.75, 1) {};
		\node (15) at (4.25, 1) {};
		\node (16) at (3.75, -1) {};
		\node (17) at (4.25, -1) {};
		\node (18) at (3.5, 0) {};
		\node (19) at (4.5, 0) {};
		\node (20) at (4, 6) {};
		\node (21) at (4, 5) {};
		\node (22) at (4, 3) {};
		\node (23) at (4, 4) {};
		\node (24) at (3.5, 4.5) {};
		\node (25) at (4.5, 4.5) {};
		\node (26) at (2.5, 4.5) {};
		\node (27) at (5.5, 4.5) {};
		\node (28) at (-0.5, 0) {};
		\node (29) at (0.5, 0) {};
		\node (30) at (-0.25, 0.25) {};
		\node (31) at (0.25, 0.25) {};
		\node (32) at (-0.25, -0.25) {};
		\node (33) at (0.25, -0.25) {};
		\node (34) at (3.5, 0) {};
		\node (35) at (4.5, 0) {};
		\node (36) at (3.75, 0.25) {};
		\node (37) at (4.25, 0.25) {};
		\node (38) at (3.75, -0.25) {};
		\node (39) at (4.25, -0.25) {};
		\node (40) at (-1.5, 4.5) {};
		\node (41) at (-0.5, 4.5) {};
		\node (42) at (-1.25, 4.75) {};
		\node (43) at (-0.75, 4.75) {};
		\node (44) at (-1.25, 4.25) {};
		\node (45) at (-0.75, 4.25) {};
		\node (46) at (0.5, 4.5) {};
		\node (47) at (1.5, 4.5) {};
		\node (48) at (0.75, 4.75) {};
		\node (49) at (1.25, 4.75) {};
		\node (50) at (0.75, 4.25) {};
		\node (51) at (1.25, 4.25) {};
		\node (52) at (2.5, 4.5) {};
		\node (53) at (3.5, 4.5) {};
		\node (54) at (2.75, 4.75) {};
		\node (55) at (3.25, 4.75) {};
		\node (56) at (2.75, 4.25) {};
		\node (57) at (3.25, 4.25) {};
		\node (58) at (4.5, 4.5) {};
		\node (59) at (5.5, 4.5) {};
		\node (60) at (4.75, 4.75) {};
		\node (61) at (5.25, 4.75) {};
		\node (62) at (4.75, 4.25) {};
		\node (63) at (5.25, 4.25) {};
		\node (64) at (0, 6) {$\bullet$};
		\node (65) at (0, 5) {$\bullet$};
		\node (66) at (0, 4) {$\bullet$};
		\node (67) at (0, 3) {$\bullet$};
		\node (68) at (4, 6) {$\bullet$};
		\node (69) at (4, 5) {$\bullet$};
		\node (70) at (4, 4) {$\bullet$};
		\node (71) at (4, 3) {$\bullet$};
		\node (72) at (-0.25, 1) {$\bullet$};
		\node (73) at (0.25, 1) {$\bullet$};
		\node (74) at (-0.25, -1) {$\bullet$};
		\node (75) at (0.25, -1) {$\bullet$};
		\node (76) at (3.75, 1) {$\bullet$};
		\node (77) at (4.25, 1) {$\bullet$};
		\node (78) at (3.75, -1) {$\bullet$};
		\node (79) at (4.25, -1) {$\bullet$};
		\node (80) at (0, 2.25) {};
		\node (81) at (0, 1.75) {};
		\node (82) at (4, 2.25) {};
		\node (83) at (4, 1.75) {};
		\draw [bend left=15] (4.center) to (0.center);
		\draw [bend left=60, looseness=1.50] (0.center) to (1.center);
		\draw [bend left=15] (1.center) to (5.center);
		\draw [bend left=15] (5.center) to (3.center);
		\draw [bend left=60, looseness=1.50] (3.center) to (2.center);
		\draw [bend left=15] (2.center) to (4.center);
		\draw [bend left=45] (10.center) to (7.center);
		\draw [bend left=45] (7.center) to (11.center);
		\draw [bend left=45] (11.center) to (9.center);
		\draw [bend left=45] (9.center) to (10.center);
		\draw [bend left=45] (12.center) to (6.center);
		\draw [bend left=45] (6.center) to (13.center);
		\draw [bend left=45] (13.center) to (8.center);
		\draw [bend left=45] (8.center) to (12.center);
		\draw [bend left=15] (18.center) to (14.center);
		\draw [bend left=60, looseness=1.50] (14.center) to (15.center);
		\draw [bend left=15] (15.center) to (19.center);
		\draw [bend left=15] (19.center) to (17.center);
		\draw [bend left=60, looseness=1.50] (17.center) to (16.center);
		\draw [bend left=15] (16.center) to (18.center);
		\draw [bend left=45] (24.center) to (21.center);
		\draw [bend left=45] (21.center) to (25.center);
		\draw [bend left=45] (25.center) to (23.center);
		\draw [bend left=45] (23.center) to (24.center);
		\draw [bend left=45] (26.center) to (20.center);
		\draw [bend left=45] (20.center) to (27.center);
		\draw [bend left=45] (27.center) to (22.center);
		\draw [bend left=45] (22.center) to (26.center);
		\draw [bend left=15, looseness=1.25, dashed] (28.center) to (30.center);
		\draw [bend left=15, looseness=0.75, dashed] (30.center) to (31.center);
		\draw [bend left=15, looseness=1.25, dashed] (31.center) to (29.center);
		\draw [bend left=15, looseness=1.25] (29.center) to (33.center);
		\draw [bend left=15, looseness=0.75] (33.center) to (32.center);
		\draw [bend left=15, looseness=1.25] (32.center) to (28.center);
		\draw (30.center) to (33.center);
		\draw (32.center) to (31.center);
		\draw [bend left=15, looseness=1.25, dashed] (34.center) to (36.center);
		\draw [bend left=15, looseness=0.75, dashed] (36.center) to (37.center);
		\draw [bend left=15, looseness=1.25, dashed] (37.center) to (35.center);
		\draw [bend left=15, looseness=1.25] (35.center) to (39.center);
		\draw [bend left=15, looseness=0.75] (39.center) to (38.center);
		\draw [bend left=15, looseness=1.25] (38.center) to (34.center);
		\draw (36.center) to (39.center);
		\draw (38.center) to (37.center);
		\draw [bend left=15, looseness=1.25, dashed] (40.center) to (42.center);
		\draw [bend left=15, looseness=0.75, dashed] (42.center) to (43.center);
		\draw [bend left=15, looseness=1.25, dashed] (43.center) to (41.center);
		\draw [bend left=15, looseness=1.25] (41.center) to (45.center);
		\draw [bend left=15, looseness=0.75] (45.center) to (44.center);
		\draw [bend left=15, looseness=1.25] (44.center) to (40.center);
		\draw (42.center) to (45.center);
		\draw (44.center) to (43.center);
		\draw [bend left=15, looseness=1.25, dashed] (46.center) to (48.center);
		\draw [bend left=15, looseness=0.75, dashed] (48.center) to (49.center);
		\draw [bend left=15, looseness=1.25, dashed] (49.center) to (47.center);
		\draw [bend left=15, looseness=1.25] (47.center) to (51.center);
		\draw [bend left=15, looseness=0.75] (51.center) to (50.center);
		\draw [bend left=15, looseness=1.25] (50.center) to (46.center);
		\draw (48.center) to (51.center);
		\draw (50.center) to (49.center);
		\draw [bend left=15, looseness=1.25, dashed] (52.center) to (54.center);
		\draw [bend left=15, looseness=0.75, dashed] (54.center) to (55.center);
		\draw [bend left=15, looseness=1.25, dashed] (55.center) to (53.center);
		\draw [bend left=15, looseness=1.25] (53.center) to (57.center);
		\draw [bend left=15, looseness=0.75] (57.center) to (56.center);
		\draw [bend left=15, looseness=1.25] (56.center) to (52.center);
		\draw [bend left=15, looseness=1.25, dashed] (58.center) to (60.center);
		\draw [bend left=15, looseness=0.75, dashed] (60.center) to (61.center);
		\draw [bend left=15, looseness=1.25, dashed] (61.center) to (59.center);
		\draw [bend left=15, looseness=1.25] (59.center) to (63.center);
		\draw [bend left=15, looseness=0.75] (63.center) to (62.center);
		\draw [bend left=15, looseness=1.25] (62.center) to (58.center);
		\draw [->] (80.center) to (81.center);
		\draw [->] (82.center) to (83.center);
		\draw [stealth-stealth,dotted] (58.center) to (53.center);
\end{tikzpicture}
\caption{Description of the two ramified Real covers of Example \ref{exampleHyperelliptic}. The involutions $\sigma_+$ and and $\sigma_-$ correspond respectively to lefth-hand side and right-hand side.}
\label{figureCoverElliptic}
\end{figure}

	\subsection{Extension to pairs $(G,\varepsilon)$}
	\label{subsectionExtensionG}

In this subsection, we explain how to generalize the results of the Subsection \ref{subsectionDefinitionSign} to any pair $(G,\varepsilon)$ where :\begin{itemize}
	\item $G$ is a finite group whose characters are real-valued and
	\item $\varepsilon : G \rightarrow \{ \pm 1 \}$ is a non-trivial group morphism.
\end{itemize} The condition on $G$ is required in order to use Corrolary \ref{corollaryVectorL}. In this setting, the objects that we count are isomorphism classes of principal $G$-bundles over the surfaces $X_\sigma^o$. Indeed, monodromy representation for principal $G$-bundles identifies the isomorphism class $[P]$ of $P$ to a morphism \begin{equation}
\rho_P : \pi_1(X_\sigma^o,p) \rightarrow G
\end{equation} well-defined up to conjugation by $G$. 

\begin{definition}[$(G,\varepsilon)$-Real Hurwitz numbers]
Let\begin{itemize}
	\item $g,n$ be non-negative integers,
	\item $(X,B,\sigma)$ a connected genus $g$ marked Real Riemann surface with $B = \{ b_1^\pm,\ldots,b_n^\pm \}$,
	\item $\bm{c} = (\bm{c}_1,\ldots,\bm{c}_n)$ a sequence of conjugacy classes of $G$ and
	\item $a \in H_1(X_\sigma^o)$ an admissible class. 
\end{itemize} The \textit{$(G,\varepsilon)$-Real Hurwitz number} $\mathbb{R} H_{g,G,\varepsilon}^\bullet(\bm{c})$ is the signed weighted-number of isomorphism classes of principal $G$-bundles $P \rightarrow X_\sigma^o$ whose monodromy around $b_i = \{ b_i^+,b_i^- \}$ belongs to the class $\bm{c}_i$. The weight of the isomorphism class $[P]$ is $1 / \# \mathrm{Aut}(P)$ and its sign is $\varepsilon \rho_P(a)$.
\end{definition} 

The numbers defined do not depend on the admissible class $a$. Indeed, one can use monodromy representation to express $\mathbb{R} H_{g,G,\varepsilon}^\bullet(\bm{c})$ as the coefficient of the identity in the product \begin{equation}
\frac{1}{\# G} \mathfrak{K}^h \mathfrak{L}^{k+1} \bm{c}_1 \ldots \bm{c}_n
\end{equation} using the same proof as Lemma \ref{lemmaRealOperatorProduct}. Moreover, one finds the extension of Theorem \ref{theoremRealExplicit} to $(G,\varepsilon)$ by performing the same computations.

\begin{theorem}
The $(G,\varepsilon)$-Real Hurwitz numbers are given explicitly by
\begin{equation}
\mathbb{R} H_{g,G,\varepsilon}^\bullet(\bm{c}) = \sum_{\rho^T = \rho} \left( SFS_{G,\varepsilon}(\rho) \frac{dim(\rho)}{\# G} \right)^{1-g} \prod_{i=1}^n f_{\bm{c}_i}(\rho).
\end{equation}
\label{theoremRealExplicitG}
\end{theorem}

For $(G,\varepsilon) = (\mathfrak{S}_d,\varepsilon)$, one recovers the degree $d$ disconnected signed Real Hurwitz numbers. In terms of principal $\mathfrak{S}_d$-bundles, the degree $d$ connected signed Real Hurwitz numbers correspond to the count of the bundles that do not admit a reduction to a subgroup $\mathfrak{S}_{d_1} \times \mathfrak{S}_{d_2}$ with $d_1 + d_2 = d$ and $d_1,d_2 \geq 1$.\\

In the case of doublet targets, we can make use of the \textit{puncture class} (Definition \ref{definitionPunctureClass}) to define the \textit{$(G,\varepsilon)$-Doublet Hurwitz numbers} from the $G$-Complex Hurwitz numbers introduced in Subsection \ref{subsectionClassiclaHurwitz}. 

\begin{definition}[$(G,\varepsilon)$-Doublet Hurwitz numbers]
Let $(X,B,\sigma)$ be the marked doublet of a genus $g$ Riemann surface, of associated partition $I \sqcup J = \{ 1,\ldots,n\}$ and puncture class $a$, and $\bm{c} = (\bm{c}_1,\ldots,\bm{c}_n)$ be conjugacy classes of $G$. The \textit{$(G,\varepsilon)$-Complex Hurwitz numbers} $H_{g,G,a}^\bullet(\bm{c})$ are the numbers \begin{equation}
H_{g,G,\varepsilon,a}^\bullet(\bm{c}) = \prod_{i \in I} \varepsilon(\bm{c}_i) H_{g,G}^\bullet(\bm{c}).
\end{equation}
\label{definitionGEpsilonComplexHurwitz}
\end{definition}

They are expressed in terms of the characters of $G$ using (\ref{equationClosedFormulaComplexG}). 

	\subsection{Signed Real Hurwitz numbers with completed cycles}
	\label{subsectionCompletedCyclesReal}

In \cite{articleOkounkovPandharipande}, Okounkov and Pandharipande relate the relative stationnary Gromov-Witten invariants \textit{with descendents} of Riemann surfaces to the Hurwitz numbers with \textit{completed cycles}, see Subsection \ref{subsectionCompletedCycles} for the construction of the completed cycles and (\ref{equationHurwitzCompletedCycles}) for the expression of the disconnected Hurwitz numbers with completed cycles.

Comparing Theorem \ref{theoremRealExplicit} with \cite[Remark 9.15]{articleGeorgievaIonel}, the signed Real Hurwitz numbers defined in Subsection \ref{subsectionDefinitionSign} coincide with the relative Real Gromov-Witten invariants of the target. In the present Subsection, we construct the \textit{signed Real Hurwitz numbers with completed cycles}. As it will be proved in a subsequent paper, they are related to the relative stationnary Real Gromov-Witten invariants with descendents of the target.\\

We adapt the geometric description of Hurwitz numbers with completed cycles presented in \cite{articleShadrinSpitzZvonkine} to the Real context. Let $(X,\sigma,B \sqcup C)$ be a connected marked Real Riemann surface with \begin{equation}
B = \{ b_1^\pm,\ldots,b_n^\pm \} \text{ and } C = \{ c_1^\pm, \ldots , c_m^\pm \}.
\end{equation} We consider a set $\bm{\lambda} = (\bm{\lambda}_1,\ldots,\bm{\lambda}_n)$ of partitions of an integer $d$ and a set $\bm{k} = (\bm{k}_1,\ldots,\bm{k}_m)$ of positive integers. In the present subsection, we are interested in the isomorphism classes of Real holomorphic maps \begin{equation}
f : (X',\sigma') \rightarrow (X,\sigma)
\end{equation} unramified over $X \setminus (B \sqcup C)$, whose ramification profile over $b_i^\pm$ is $\bm{\lambda}_i$, together with the following decorations : for each $j$, the choice of a pair of conjugated subsets $S_j^\pm \subseteq f^{-1}(c_j^\pm)$ which contain all the points at which $f$ is ramified (it might also contain some unramified points). It defines a partition $\bm{\nu}_j$ which corresponds to the ramification order at the points of $S_j^\pm$. Notice that the set of integers $\bm{k}$ is not involved in the description of these morphisms. The pair $(\bm{k}_j,\bm{\nu}_j)$ can be used to define the weight \begin{equation}
q_{\bm{k}_j,\bm{\nu}_j},
\end{equation} which is the coefficient of $\bm{\nu}_j$ in the expression of the completed cycle $\overline{(\bm{k}_j)}$, see (\ref{equationCompletedCycles}).

\begin{definition}[Disconnected signed Real Hurwitz numbers with completed cycles]
Let $(X,\sigma,B \sqcup C)$, $\bm{\lambda}$ and $\bm{k}$ be as above. Denote by $g$ the genus of $X$ and choose an admissible class $a$. The \textit{disconnected signed Real Hurwitz number with completed cycles} associated to this data is the number \begin{equation}
\mathbb{R}H^\bullet_{g,d}(\bm{\lambda};\bm{k}) = \sum_{[f;S_1^\pm,\ldots,S_m^\pm]} \frac{\varepsilon \rho_f(a)}{\# Aut(f;S_1^\pm,\ldots,S_m^\pm)} q_{\bm{k}_1,\bm{\nu}_1} \ldots q_{\bm{k}_m,\bm{\nu}_m}.
\end{equation} The sum is over the isomorphism classes of \textit{decorated} Real holomorphic maps $(f;S_1^\pm,\ldots,S_m^\pm)$ unramified over $X \setminus (B \sqcup C)$ with ramification profile $\bm{\lambda}_i$ around $b_i^\pm$. An isomorphism between decorated ramified Real covers is a an isomorphism of the underlying ramified Real covers which preserve the subsets $S_j^\pm$ set-wise. 
\end{definition}

\begin{proposition}
The disconnected signed Real Hurwitz numbers with completed cycles are given explicitly by \begin{equation}
\mathbb{R} H_{g,d}^\bullet(\bm{\lambda};\bm{k}) = \sum_{\mu^T = \mu} \left( (-1)^{\frac{d-r(\mu)}{2}}  \frac{dim(\mu)}{d!} \right)^{1-g} \prod_{i=1}^n f_{\bm{\lambda}_i}(\mu) \prod_{j=1}^m \frac{p^*_{\bm{k}_j}(\mu)}{\bm{k}_j}.
\end{equation}
\end{proposition}

\begin{proof}
We start with the definition of $\mathbb{R} H_{g,d}^\bullet(\bm{\lambda};\bm{k})$. It can be written as a sum over isomorphism classes of ramified Real covers (without the decorations) by replacing $1 / \# Aut(f;S_1^\pm,\ldots,S_m^\pm)$ by \begin{equation}
\frac{1}{\# Aut(f)} \prod_{j=1}^m \binom{m_1(\bm{\nu}_j) + d - | \bm{\nu}_j |}{d - \vert \bm{\nu}_j \vert}.
\end{equation} The combinatorial coefficient account for the different ways to choose the subsets $S_j^\pm$ given the partitions $(\bm{\nu}_j)$. Comparing with (\ref{equationfExtended}) and (\ref{equationCompletedCycles}), $\mathbb{R} H_{g,d}^\bullet(\bm{\lambda};\bm{k})$ can be expressed as \begin{equation}
\sum_{\mu^T = \mu} \left( (-1)^{\frac{d-r(\mu)}{2}}  \frac{dim(\mu)}{d!} \right)^{1-g} \prod_{i=1}^n f_{\bm{\lambda}_i}(\mu) \prod_{j=1}^m \underbrace{\sum_\nu q_{\bm{k}_j,\nu}f_\nu(\mu)}_{\frac{p^*_{\bm{k}_j}(\mu)}{\bm{k}_j}}.
\end{equation} The proposition is proved.
\end{proof}

The decorations $S_1^\pm,\ldots,S_m^\pm$ can be understood geometrically as follows, see \cite{articleShadrinSpitzZvonkine} in the context of double Hurwitz numbers. For each $j$, a doublet is attached to the source $(X',\sigma')$ at the points $S_j^\pm$ in such a way that the involution $\sigma'$ extends to the singular Riemann surface obtained. The ramified Real cover is also defined on this singular Real Riemann surface by contracting a component on the doublet onto $c^+_j$ and the other onto $c^-_j$. One can think of the coefficients $q_{k,\nu}$ as some kind of intersection number on the moduli space of the contracted doublets.

The notion of connectedness and doublets has to be changed to reflect this geometrical point of view on the decorations.

\begin{definition}[Connected signed Real Hurwitz numbers with completed cycles]
Let $(X,\sigma,B \sqcup C)$, $\bm{\lambda}$ and $\bm{k}$ be as above. Denote by $g$ the genus of $X$ and choose an admissible class $a$. The \textit{connected signed Real Hurwitz number with completed cycles} associated to this data is the number \begin{equation}
\mathbb{R}H_{g,d}(\bm{\lambda};\bm{k}) = \sum_{[f;S_1^\pm,\ldots,S_m^\pm]} \frac{\varepsilon \rho_f(a)}{\# Aut(f;S_1^\pm,\ldots,S_m^\pm)} q_{\bm{k}_1,\bm{\nu}_1} \ldots q_{\bm{k}_m,\bm{\nu}_m}.
\end{equation} The sum is over the isomorphism classes of \textit{decorated} Real holomorphic maps $(f;S_1^\pm,\ldots,S_m^\pm)$, unramified over $X \setminus (B \sqcup C)$ with ramification profile $\bm{\lambda}_i$ around $b_i^\pm$, which are such that the singular Riemann surface obtained by indentifying the points of $S_j^+$ together, similarly for $S_j^-$, for all $j$ is either a singular connected Real Riemann surface or a singular doublet.
\label{definitionConnectedRealCompletedCycles}
\end{definition}

The generating series introduced in the proof of Lemma \ref{lemmaRealOperatorProduct} extend in a standard way to allow completed cycles, see for instance \cite{bookCavalieriMiles}. These generating series still satisfy the relation (\ref{equationConnectedDisconnected}).

\begin{proposition}
The contribution of doublet covers to $\mathbb{R}H_{g, 2d'}(\bm{\lambda};\bm{k})$ is \begin{equation}
\frac{(-1)^{d'(g-1)}}{2} \sum_{\bm{\lambda} = \bm{\lambda^+} \sqcup \bm{\lambda^-}} \prod_{i=1}^n \varepsilon \left( \bm{\lambda}_i^{s(i)} \right) \prod_{j=1}^m \left( 1 + (-1)^{\bm{k}_j - 1} \right)  H_{g,d}(\bm{\lambda^+}, \bm{\lambda^-} ; \bm{k})
\end{equation} for any function \begin{equation}
s : \{ 1,\ldots,n\} \rightarrow \{ \pm \}.
\end{equation} It vanishes if at least one of the partitions $\bm{\lambda}_i$ is odd or one of the integers $\bm{k}_j$ is even. 
\end{proposition}

\begin{proof}
Let $f : (D,\sigma') \rightarrow (X,\sigma)$ be a ramified Real cover accounting for the doublet contribution of $\mathbb{R}H_{g,2d'}(\bm{\lambda};\bm{k})$ with the decoration $S_1^\pm,\ldots,S_m^\pm$. It has the ramification profile $\bm{\lambda}$ at $B$ and $\bm{\nu} = ((\bm{\nu}_1,1,\ldots,1), \ldots, (\bm{\nu}_m,1,\ldots,1))$ at $C$.

By Definition \ref{definitionConnectedRealCompletedCycles}, the fact that $(f; S_1^\pm,\ldots,S_m^\pm)$ is part of the doublet contribution means that after identiying the points in $S_j^+$ together, those in $S_j^-$ together, for all $j$ seperately, we obtain a doublet. Thus, $(D,\sigma')$ must be a union of doublets $(D_l,\sigma_l)$. Moreover, we claim that we can select a component $X'_l$ of $D_l$ such that for any $j$ \begin{equation}
S_j^+ \subseteq \bigsqcup_l X'_l \text{ or } S_j^- \subseteq \bigsqcup_l X'_l.
\end{equation} Indeed, consider the graph $\Gamma$ whose vertices are the connected components of $D$ and such that two vertices $v,v'$ are related by an edge if there exist $j$ and $\epsilon\in \{ \pm \}$ such that $v$ and $v'$ intersect $S_j^\epsilon$. By definition, $\Gamma$ has exactly two connected components. Moreover, it carries an involution which exchanges its two connected components. Selecting the components of $D$ which belong to a given connected component of $\Gamma$ gives the required splitting.

As in Subsection \ref{subsectionDoubletContribution}, restricting $f$ to $X' = \bigsqcup X'_l$ gives rise to a splitting $\bm{\nu} = \bm{\nu^+} \sqcup \bm{\nu^-} $. It follows from the choice of $X'$ that $\bm{\nu^+}_j \sqcup \bm{\nu^-}_j$ has the shape \begin{equation}
(\bm{\nu}_j,1,\ldots,1) \sqcup (1,\ldots,1) \text{ or } (1,\ldots,1) \sqcup (\bm{\nu}_j,1,\ldots,1)
\end{equation} for all $j$. Since the coefficient $q_{\bm{\nu}_j,\bm{k}_j}$ vanishes unless $\varepsilon(\bm{\nu}_j) = (-1)^{\bm{k}_j-1}$, it is straightforward to adapt the proof of Corollary \ref{corollaryContributionDoublets} to obtain \begin{equation}
\frac{(-1)^{(g-1)d'}}{2}\sum_{\bm{\lambda^+} \sqcup \bm{\lambda^-}} \prod_{i=1}^n \varepsilon \left( \bm{\lambda}^{s(i)}_i \right) \prod_{j=1}^m \left( 1 + (-1)^{\bm{k}_j - 1} \right) H_{g,d'}(\bm{\lambda^+}, \bm{\lambda^-} ; \bm{k} )
\end{equation} as required.
\end{proof}

\section{Degeneration formulas}
\label{sectionDegenerationFormulas}

In this section, we fix a pair $(G,\varepsilon)$ as in Subsection \ref{subsectionExtensionG} and omit $(G,\varepsilon)$ from the notations. We degenerate the targets by collapsing simultaneously a pair of conjugate circles in the target curve. Such a \textit{degeneration process} is introduced in Subsection \ref{subsectionDegenerations}. In Subsection \ref{subsectionDegenerationFormula}, we relate the $(G,\varepsilon)$-Real Hurwitz numbers corresponding to a connected Real target to the $(G,\varepsilon)$-Real and Complex Hurwitz numbers of the degenerated target. It is a direct consequence of the orthogonality of the characters. In Subsection \ref{subsectionDegenerationSigns}, we consider the more subtle problem of the relation between the sign of a given covering and the sign of the degenerated covering. It enables to relate our signs with those provided in \cite{articleGeorgievaIonel} using the Real Gromov-Witten theory.

	\subsection{Degenerations}
	\label{subsectionDegenerations}

Let $(X,\sigma,B)$ be a marked Real Riemann surface. Given a pair of disjoint embedded circles in $X \setminus B$ exchanged by $\sigma$, we construct another marked Real Riemann surface $(Y,\sigma,C)$ as follows : \begin{enumerate}[label = (\arabic*)]
	\item Collapse the two circles into two nodes, conjugated by $\sigma$.
	\item Resolve each of the two nodes. Label the two marked points obtained by resolving the same node by the same label $\pm$. The involution $\sigma$ extends uniquely to the four marked points obtained. 
\end{enumerate} At the level of the complex structures, the degeneration process happens when one varries the complex structure of $(X,\sigma)$ in such a way that the length of the embbeded circles tend to $0$.\\

When $(X,\sigma)$ is connected of genus $g$, the degeneration $(Y,\sigma,C)$ can have exactly four different topological types : \begin{enumerate}[label = (\alph*)]
	\item $(Y,\sigma)$ is the union of two connected Real Riemann surfaces of genus $h,h'$ such that $g = h + h' + 1$.
	\item $(Y,\sigma)$ is the union of a connected Real Riemann surface of genus $h$ and the doublet of a genus $h'$ Riemann surface such that $g = h + 2h'$.
	\item $(Y,\sigma)$ is a connected Real Riemann surface of genus $h$ such that $g = h +2$.
	\item $(Y,\sigma)$ is the doublet of a Riemann surface of genus $h$ such that $g = 2h + 1$.
\end{enumerate} The cases (a),(b),(c),(d) are depicted in Figure \ref{figureDegene}.

\begin{lemma}
If the degeneration is of type (d), then each connected component of $Y$ carries one of the additionnal marked points labelled $+$ and one of those labelled $-$.  
\label{lemmaDegeneCaseD}
\end{lemma}

\begin{proof}
Each connected component carries two additionnal marked points. Suppose that one component carries the two marked points labelled $+$, and the other the two marked points labelled $-$. The original Riemann surface must be homeomorphic to the space obtained from $Y$ by gluing the positive marked points together, smoothing the node, and proceeding similarly for the negative marked points. The later space has two connected components, while $X$ is connected. This is a contradiction.
\end{proof}

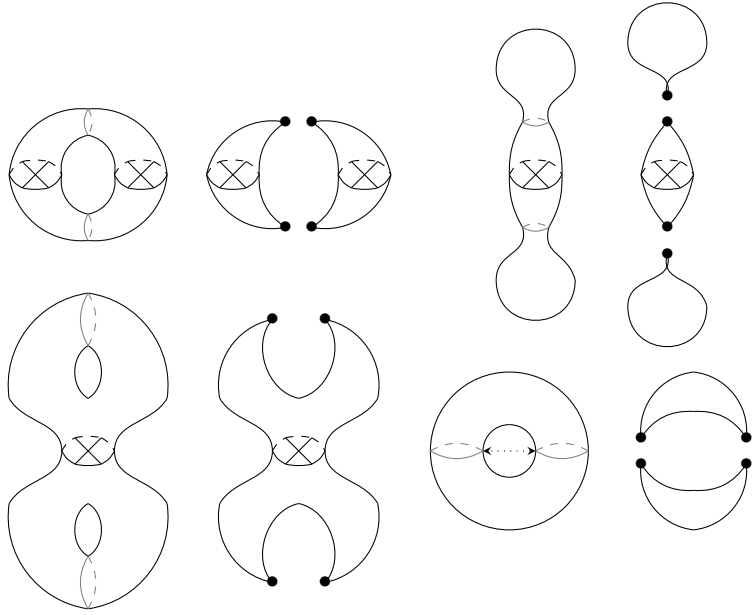
\begin{figure}[h!]
\centering
\begin{tikzpicture}[scale = 0.7]
		\node (4) at (-1.25, 1) {};
		\node (5) at (-0.75, 1) {};
		\node (6) at (-1.25, -1) {};
		\node (7) at (-0.75, -1) {};
		\node (8) at (-1.75, 2) {};
		\node (9) at (-0.25, 2) {};
		\node (10) at (-1.75, -2) {};
		\node (11) at (-0.25, -2) {};
		\node (24) at (-1.5, 0) {};
		\node (25) at (-0.5, 0) {};
		\node (26) at (1.5, 1) {};
		\node (27) at (1.5, 1) {};
		\node (28) at (1.5, -1) {};
		\node (29) at (1.5, -1) {};
		\node (34) at (1, 0) {};
		\node (35) at (2, 0) {};
		\node (36) at (1.5, 1.5) {};
		\node (37) at (1.5, 1.5) {};
		\node (38) at (0.75, 2.5) {};
		\node (39) at (2.25, 2.5) {};
		\node (40) at (1.5, -1.5) {};
		\node (41) at (1.5, -1.5) {};
		\node (42) at (0.75, -2.5) {};
		\node (43) at (2.25, -2.5) {};
		\node (44) at (-9.5, 1.25) {};
		\node (45) at (-9.5, 0.75) {};
		\node (46) at (-9.5, -1.25) {};
		\node (47) at (-9.5, -0.75) {};
		\node (48) at (-11, 0) {};
		\node (49) at (-10, 0) {};
		\node (50) at (-9, 0) {};
		\node (51) at (-8, 0) {};
		\node (52) at (-9.5, 1.25) {};
		\node (53) at (-9.5, 0.75) {};
		\node (54) at (-9.5, -0.75) {};
		\node (55) at (-9.5, -1.25) {};
		\node (56) at (-5.75, 1) {};
		\node (57) at (-5.75, 1) {};
		\node (58) at (-5.75, -1) {};
		\node (59) at (-5.75, -1) {};
		\node (60) at (-7.25, 0) {};
		\node (61) at (-6.25, 0) {};
		\node (62) at (-4.75, 0) {};
		\node (63) at (-3.75, 0) {};
		\node (64) at (-5.25, 1) {};
		\node (65) at (-5.25, 1) {};
		\node (66) at (-5.25, -1) {};
		\node (67) at (-5.25, -1) {};
		\node (68) at (-10, -5.25) {};
		\node (69) at (-9, -5.25) {};
		\node (70) at (-11, -4.25) {};
		\node (71) at (-8, -4.25) {};
		\node (72) at (-9.5, -4.25) {};
		\node (73) at (-9.5, -3.25) {};
		\node (74) at (-9.5, -3.25) {};
		\node (75) at (-9.5, -2.25) {};
		\node (76) at (-9.5, -2.25) {};
		\node (77) at (-9.5, -6.25) {};
		\node (78) at (-9.5, -7.25) {};
		\node (79) at (-9.5, -7.25) {};
		\node (80) at (-9.5, -8.25) {};
		\node (81) at (-9.5, -8.25) {};
		\node (82) at (-11, -6.25) {};
		\node (83) at (-8, -6.25) {};
		\node (84) at (-6, -5.25) {};
		\node (85) at (-5, -5.25) {};
		\node (86) at (-7, -4.25) {};
		\node (87) at (-4, -4.25) {};
		\node (88) at (-5.5, -4.25) {};
		\node (89) at (-6, -2.75) {};
		\node (90) at (-5, -2.75) {};
		\node (91) at (-6, -2.75) {};
		\node (92) at (-5, -2.75) {};
		\node (93) at (-5.5, -6.25) {};
		\node (94) at (-6, -7.75) {};
		\node (95) at (-5, -7.75) {};
		\node (96) at (-6, -7.75) {};
		\node (97) at (-5, -7.75) {};
		\node (98) at (-7, -6.25) {};
		\node (99) at (-4, -6.25) {};
		\node (100) at (-1.5, -3.75) {};
		\node (101) at (-1.5, -4.75) {};
		\node (102) at (-1.5, -6.75) {};
		\node (103) at (-1.5, -5.75) {};
		\node (104) at (-3, -5.25) {};
		\node (105) at (-2, -5.25) {};
		\node (106) at (-1, -5.25) {};
		\node (107) at (0, -5.25) {};
		\node (108) at (-1.5, -3.75) {};
		\node (109) at (-1.5, -4.75) {};
		\node (110) at (-1.5, -5.75) {};
		\node (111) at (-1.5, -6.75) {};
		\node (124) at (2, -3.75) {};
		\node (125) at (2, -4.5) {};
		\node (126) at (1, -5) {};
		\node (127) at (1, -5) {};
		\node (128) at (3, -5) {};
		\node (129) at (3, -5) {};
		\node (130) at (2, -3.75) {};
		\node (131) at (2, -4.5) {};
		\node (132) at (2, -6.75) {};
		\node (133) at (2, -6) {};
		\node (134) at (1, -5.5) {};
		\node (135) at (1, -5.5) {};
		\node (136) at (3, -5.5) {};
		\node (137) at (3, -5.5) {};
		\node (138) at (2, -6) {};
		\node (139) at (2, -6.75) {};
		\node (140) at (-10, -5.25) {};
		\node (141) at (-9, -5.25) {};
		\node (142) at (-9.75, -5) {};
		\node (143) at (-9.25, -5) {};
		\node (144) at (-9.75, -5.5) {};
		\node (145) at (-9.25, -5.5) {};
		\node (146) at (-6, -5.25) {};
		\node (147) at (-5, -5.25) {};
		\node (148) at (-5.75, -5) {};
		\node (149) at (-5.25, -5) {};
		\node (150) at (-5.75, -5.5) {};
		\node (151) at (-5.25, -5.5) {};
		\node (152) at (-1.5, 0) {};
		\node (153) at (-0.5, 0) {};
		\node (154) at (-1.25, 0.25) {};
		\node (155) at (-0.75, 0.25) {};
		\node (156) at (-1.25, -0.25) {};
		\node (157) at (-0.75, -0.25) {};
		\node (158) at (1, 0) {};
		\node (159) at (2, 0) {};
		\node (160) at (1.25, 0.25) {};
		\node (161) at (1.75, 0.25) {};
		\node (162) at (1.25, -0.25) {};
		\node (163) at (1.75, -0.25) {};
		\node (164) at (-11, 0) {};
		\node (165) at (-10, 0) {};
		\node (166) at (-10.75, 0.25) {};
		\node (167) at (-10.25, 0.25) {};
		\node (168) at (-10.75, -0.25) {};
		\node (169) at (-10.25, -0.25) {};
		\node (170) at (-9, 0) {};
		\node (171) at (-8, 0) {};
		\node (172) at (-8.75, 0.25) {};
		\node (173) at (-8.25, 0.25) {};
		\node (174) at (-8.75, -0.25) {};
		\node (175) at (-8.25, -0.25) {};
		\node (176) at (-7.25, 0) {};
		\node (177) at (-6.25, 0) {};
		\node (178) at (-7, 0.25) {};
		\node (179) at (-6.5, 0.25) {};
		\node (180) at (-7, -0.25) {};
		\node (181) at (-6.5, -0.25) {};
		\node (182) at (-4.75, 0) {};
		\node (183) at (-3.75, 0) {};
		\node (184) at (-4.5, 0.25) {};
		\node (185) at (-4, 0.25) {};
		\node (186) at (-4.5, -0.25) {};
		\node (187) at (-4, -0.25) {};
		\node (188) at (-9.5, -0.75) {};
		\node (189) at (-9.5, -1.25) {};
		\node (190) at (-9.5, -7.25) {};
		\node (191) at (-9.5, -8.25) {};
		\node (192) at (-9.5, -2.25) {};
		\node (193) at (-9.5, -3.25) {};
		\node (194) at (-9.5, 1.25) {};
		\node (195) at (-9.5, 0.75) {};
		\node (196) at (-1.25, -1) {};
		\node (197) at (-0.75, -1) {};
		\node (198) at (-1.25, 1) {};
		\node (199) at (-0.75, 1) {};
		\node (200) at (-3, -5.25) {};
		\node (201) at (-2, -5.25) {};
		\node (202) at (-1, -5.25) {};
		\node (203) at (0, -5.25) {};
		\node (204) at (-5.75, 1) {$\bullet$};
		\node (205) at (-5.25, 1) {$\bullet$};
		\node (206) at (-5.75, -1) {$\bullet$};
		\node (207) at (-5.25, -1) {$\bullet$};
		\node (208) at (1.5, 1.5) {$\bullet$};
		\node (209) at (1.5, 1) {$\bullet$};
		\node (210) at (1.5, -1) {$\bullet$};
		\node (211) at (1.5, -1.5) {$\bullet$};
		\node (212) at (-6, -2.75) {$\bullet$};
		\node (213) at (-5, -2.75) {$\bullet$};
		\node (214) at (-6, -7.75) {$\bullet$};
		\node (215) at (-5, -7.75) {$\bullet$};
		\node (216) at (1, -5) {$\bullet$};
		\node (217) at (1, -5.5) {$\bullet$};
		\node (218) at (3, -5) {$\bullet$};
		\node (219) at (3, -5.5) {$\bullet$};
		\draw [in=75, out=-90, looseness=1.25] (8.center) to (4.center);
		\draw [in=105, out=-90, looseness=1.25] (9.center) to (5.center);
		\draw [in=90, out=-75, looseness=1.25] (6.center) to (10.center);
		\draw [in=105, out=-105, looseness=1.25] (7.center) to (11.center);
		\draw [bend right=90, looseness=1.75] (10.center) to (11.center);
		\draw [bend left=90, looseness=1.75] (8.center) to (9.center);
		\draw [bend right=15] (4.center) to (24.center);
		\draw [bend right=15] (24.center) to (6.center);
		\draw [bend left=15] (5.center) to (25.center);
		\draw [bend left=15] (25.center) to (7.center);
		\draw [bend right=15] (26.center) to (34.center);
		\draw [bend right=15] (34.center) to (28.center);
		\draw [bend left=15] (27.center) to (35.center);
		\draw [bend left=15] (35.center) to (29.center);
		\draw [in=75, out=-90, looseness=1.25] (38.center) to (36.center);
		\draw [in=105, out=-90, looseness=1.25] (39.center) to (37.center);
		\draw [bend left=90, looseness=1.75] (38.center) to (39.center);
		\draw [in=90, out=-75, looseness=1.25] (40.center) to (42.center);
		\draw [in=105, out=-105, looseness=1.25] (41.center) to (43.center);
		\draw [bend right=90, looseness=1.75] (42.center) to (43.center);
		\draw [bend left=45] (49.center) to (45.center);
		\draw [bend left=45] (53.center) to (50.center);
		\draw [bend left=45] (50.center) to (54.center);
		\draw [bend right=315] (47.center) to (49.center);
		\draw [bend left=45] (48.center) to (44.center);
		\draw [bend left=45] (52.center) to (51.center);
		\draw [bend left=45] (51.center) to (55.center);
		\draw [bend right=45] (48.center) to (46.center);
		\draw [bend left] (61.center) to (57.center);
		\draw [bend left] (65.center) to (62.center);
		\draw [bend left] (62.center) to (66.center);
		\draw [bend left] (59.center) to (61.center);
		\draw [bend left=45] (60.center) to (56.center);
		\draw [bend left=45] (64.center) to (63.center);
		\draw [bend left=45] (63.center) to (67.center);
		\draw [bend right=45] (60.center) to (58.center);
		\draw [bend left=60] (78.center) to (77.center);
		\draw [bend left=60] (77.center) to (79.center);
		\draw [in=-120, out=90] (69.center) to (71.center);
		\draw [bend right=45] (71.center) to (76.center);
		\draw [bend right=45] (75.center) to (70.center);
		\draw [in=90, out=-60] (70.center) to (68.center);
		\draw [bend right=60] (73.center) to (72.center);
		\draw [bend right=60] (72.center) to (74.center);
		\draw [bend left=45] (80.center) to (82.center);
		\draw [in=270, out=60] (82.center) to (68.center);
		\draw [in=120, out=-90] (69.center) to (83.center);
		\draw [bend left=45] (83.center) to (81.center);
		\draw [bend left=60] (94.center) to (93.center);
		\draw [bend left=60] (93.center) to (95.center);
		\draw [in=-120, out=90] (85.center) to (87.center);
		\draw [bend right=45] (87.center) to (92.center);
		\draw [bend right=45] (91.center) to (86.center);
		\draw [in=90, out=-60] (86.center) to (84.center);
		\draw [bend right=60] (89.center) to (88.center);
		\draw [bend right=60] (88.center) to (90.center);
		\draw [bend left=45] (96.center) to (98.center);
		\draw [in=270, out=60] (98.center) to (84.center);
		\draw [in=120, out=-90] (85.center) to (99.center);
		\draw [bend left=45] (99.center) to (97.center);
		\draw [bend left=45] (105.center) to (101.center);
		\draw [bend left=45] (109.center) to (106.center);
		\draw [bend left=45] (106.center) to (110.center);
		\draw [bend right=315] (103.center) to (105.center);
		\draw [bend left=45] (104.center) to (100.center);
		\draw [bend left=45] (108.center) to (107.center);
		\draw [bend left=45] (107.center) to (111.center);
		\draw [bend right=45] (104.center) to (102.center);
		\draw [bend left] (127.center) to (125.center);
		\draw [bend left] (131.center) to (128.center);
		\draw [bend left=45] (126.center) to (124.center);
		\draw [bend left=45] (130.center) to (129.center);
		\draw [bend left] (136.center) to (138.center);
		\draw [bend left] (133.center) to (135.center);
		\draw [bend left=45] (137.center) to (139.center);
		\draw [bend right=45] (134.center) to (132.center);
		\draw [bend left=15, looseness=1.25, dashed] (140.center) to (142.center);
		\draw [bend left=15, looseness=0.75, dashed] (142.center) to (143.center);
		\draw [bend left=15, looseness=1.25, dashed] (143.center) to (141.center);
		\draw [bend left=15, looseness=1.25] (141.center) to (145.center);
		\draw [bend left=15, looseness=0.75] (145.center) to (144.center);
		\draw [bend left=15, looseness=1.25] (144.center) to (140.center);
		\draw (142.center) to (145.center);
		\draw (144.center) to (143.center);
		\draw [bend left=15, looseness=1.25, dashed] (146.center) to (148.center);
		\draw [bend left=15, looseness=0.75, dashed] (148.center) to (149.center);
		\draw [bend left=15, looseness=1.25, dashed] (149.center) to (147.center);
		\draw [bend left=15, looseness=1.25] (147.center) to (151.center);
		\draw [bend left=15, looseness=0.75] (151.center) to (150.center);
		\draw [bend left=15, looseness=1.25] (150.center) to (146.center);
		\draw (148.center) to (151.center);
		\draw (150.center) to (149.center);
		\draw [bend left=15, looseness=1.25, dashed] (152.center) to (154.center);
		\draw [bend left=15, looseness=0.75, dashed] (154.center) to (155.center);
		\draw [bend left=15, looseness=1.25, dashed] (155.center) to (153.center);
		\draw [bend left=15, looseness=1.25] (153.center) to (157.center);
		\draw [bend left=15, looseness=0.75] (157.center) to (156.center);
		\draw [bend left=15, looseness=1.25] (156.center) to (152.center);
		\draw (154.center) to (157.center);
		\draw (156.center) to (155.center);
		\draw [bend left=15, looseness=1.25, dashed] (158.center) to (160.center);
		\draw [bend left=15, looseness=0.75, dashed] (160.center) to (161.center);
		\draw [bend left=15, looseness=1.25, dashed] (161.center) to (159.center);
		\draw [bend left=15, looseness=1.25] (159.center) to (163.center);
		\draw [bend left=15, looseness=0.75] (163.center) to (162.center);
		\draw [bend left=15, looseness=1.25] (162.center) to (158.center);
		\draw (160.center) to (163.center);
		\draw (162.center) to (161.center);
		\draw [bend left=15, looseness=1.25, dashed] (164.center) to (166.center);
		\draw [bend left=15, looseness=0.75, dashed] (166.center) to (167.center);
		\draw [bend left=15, looseness=1.25, dashed] (167.center) to (165.center);
		\draw [bend left=15, looseness=1.25] (165.center) to (169.center);
		\draw [bend left=15, looseness=0.75] (169.center) to (168.center);
		\draw [bend left=15, looseness=1.25] (168.center) to (164.center);
		\draw (166.center) to (169.center);
		\draw (168.center) to (167.center);
		\draw [bend left=15, looseness=1.25, dashed] (170.center) to (172.center);
		\draw [bend left=15, looseness=0.75, dashed] (172.center) to (173.center);
		\draw [bend left=15, looseness=1.25, dashed] (173.center) to (171.center);
		\draw [bend left=15, looseness=1.25] (171.center) to (175.center);
		\draw [bend left=15, looseness=0.75] (175.center) to (174.center);
		\draw [bend left=15, looseness=1.25] (174.center) to (170.center);
		\draw (172.center) to (175.center);
		\draw (174.center) to (173.center);
		\draw [bend left=15, looseness=1.25, dashed] (176.center) to (178.center);
		\draw [bend left=15, looseness=0.75, dashed] (178.center) to (179.center);
		\draw [bend left=15, looseness=1.25, dashed] (179.center) to (177.center);
		\draw [bend left=15, looseness=1.25] (177.center) to (181.center);
		\draw [bend left=15, looseness=0.75] (181.center) to (180.center);
		\draw [bend left=15, looseness=1.25] (180.center) to (176.center);
		\draw (178.center) to (181.center);
		\draw (180.center) to (179.center);
		\draw [bend left=15, looseness=1.25, dashed] (182.center) to (184.center);
		\draw [bend left=15, looseness=0.75, dashed] (184.center) to (185.center);
		\draw [bend left=15, looseness=1.25, dashed] (185.center) to (183.center);
		\draw [bend left=15, looseness=1.25] (183.center) to (187.center);
		\draw [bend left=15, looseness=0.75] (187.center) to (186.center);
		\draw [bend left=15, looseness=1.25] (186.center) to (182.center);
		\draw (184.center) to (187.center);
		\draw (186.center) to (185.center);
		\draw [bend left, gray, dashed] (188.center) to (189.center);
		\draw [bend right, gray] (188.center) to (189.center);
		\draw [bend left, gray, dashed] (190.center) to (191.center);
		\draw [bend right, gray] (190.center) to (191.center);
		\draw [bend left, gray, dashed] (192.center) to (193.center);
		\draw [bend right, gray] (192.center) to (193.center);
		\draw [bend left, gray, dashed] (194.center) to (195.center);
		\draw [bend right, gray] (194.center) to (195.center);
		\draw [bend left, gray, dashed] (196.center) to (197.center);
		\draw [bend right, gray] (196.center) to (197.center);
		\draw [bend left, gray, dashed] (198.center) to (199.center);
		\draw [bend right, gray] (198.center) to (199.center);
		\draw [bend left, gray, dashed] (200.center) to (201.center);
		\draw [bend right, gray] (200.center) to (201.center);
		\draw [bend left, gray, dashed] (202.center) to (203.center);
		\draw [bend right, gray] (202.center) to (203.center);
		\draw [stealth-stealth,dotted] (202.center) to (201.center);
\end{tikzpicture}
\caption{The four topological types of degeneration of a pair of conjugated loops, drawn in gray. In the first three cases, one can add genus symmetrically and cross-caps to the picture. In the last one, only genus can be added.}
\label{figureDegene}
\end{figure}

	\subsection{Degeneration formulas}
	\label{subsectionDegenerationFormula}

The signed Real Hurwitz numbers are topological. Indeed, they do not depend on the complex structure chosen in the target. In particular, they remain constant during a degeneration process. At the limit, it implies an identity with the numbers associated to the degenerated marked Real Riemann surface. This is the content of Proposition \ref{propositionDegene}. We have chosen to provide a short representation-theoretic proof, but a combinatorial one could also have been given using Lemma \ref{lemmaFuntorialityAdm}. The formulas involve the combinatorial coefficients \begin{equation}
z_c = \frac{\# G}{\# c}
\end{equation} where $c$ is a conjugacy class of $G$.

\begin{proposition}
The following degeneration formulas hold. \begin{enumerate}[label=(\alph*)]
	\item Real-Real degeneration : \begin{equation}
	\mathbb{R}H^\bullet_{g}(\bm{c},\bm{c'}) = \sum_{c \in C(G)} z_c \mathbb{R}H^\bullet_{h}(\bm{c},c) \mathbb{R}H^\bullet_{h'}(c,\bm{c'})
	\end{equation} with $g = h + h' + 1$.
	\item Real-Complex degeneration : \begin{equation}
	\mathbb{R}H^\bullet_{g}(\bm{c},\bm{c'}) = \sum_{c \in C(G)} z_c \mathbb{R}H^\bullet_{h}(\bm{c},c) H^\bullet_{h'}(c,\bm{c'})
	\end{equation} with $g = h + 2h'$.
	\item Real degeneration : \begin{equation}
	\mathbb{R}H^\bullet_{g}(\bm{c}) = \sum_{c \in C(G)}  z_c \mathbb{R} H^\bullet_{h}(\bm{c},c,c)
	\end{equation} with $g = h+2$.
	\item Complex degeneration : \begin{equation}
	\mathbb{R}H^\bullet_{g}(\bm{c}) = \sum_{c \in C(G)}  \varepsilon(c) z_c H^\bullet_{h}(\bm{c},c,c)
	\end{equation} with $g = 2h+1$.
\end{enumerate}
\label{propositionDegene}
\end{proposition}

\begin{proof}
The characters of $G$ are real-valued, see the assumption made in Subsection \ref{subsectionExtensionG}. In this case, the orthogonality of the characters reads \begin{equation}
\sum_{c \in C(G)} z_c f_c(\rho) f_c(\rho') = \left( \frac{dim(\rho)}{\# G} \right)^{-2} \delta_{\rho,\rho'}.
\end{equation} The degeneration formulas are a direct consequence of this relation. Consider for instance the formula (b). The right-hand side involves the sum over $c \in C(G)$ and two sums over irreducible representations $\rho,\rho'$. We perform the sum over $c$ first. The orthogonality of the characters transforms the two sums over irreducible characters into a single one. Thus, the right-hand side equals \begin{equation}
\sum_{\rho^T = \rho} \left( SFS_{G,\varepsilon}(\rho) \frac{dim(\rho)}{\# G} \right)^{1-h} \prod_{i} f_{\bm{c}_i}(\rho) \left( \frac{dim(\rho)}{\# G} \right)^{-2} \left( \frac{dim(\rho)}{\# G} \right)^{2-2h'} \prod_{j} f_{\bm{c'}_j}(\rho).
\end{equation} All the representations involved in the sum satisfy $SFS_{G,\varepsilon}(\rho) = \pm 1$, so that we can add it inside the powers $-2$ and $2-2h$. The degeneration formula (b) is proved by noticing that $1 - g = (1 - h) - 2 + (2 - 2h')$ if $g = h + 2h'$.

The proof of the formulas (a) and (c) is similar. In order to prove (d), we use the following version of the orthogonality of the characters : \begin{equation}
\sum_{c \in C(G)} \varepsilon(c) z_c f_c(\rho) f_c(\rho) = \sum_{c \in C(G)} z_c f_c(\rho) f_c(\rho^T) = \left( \frac{dim(\rho)}{\# G} \right)^{-2} \delta_{\rho,\rho^T}.
\end{equation} The right-hand side of the formula (d) involves a sum over $c$ and a sum over $\rho$. Performing the sum over $c$ first restricts the sum over $\rho$ to the symmetric representations, and the right-hand side becomes \begin{equation}
\sum_{\rho^T = \rho} \left( \frac{dim(\rho)}{\# G} \right)^{2-2h} \prod_{i} f_{\bm{c}_i}(\rho) \left( \frac{dim(\rho)}{\# G} \right)^{-2}.
\end{equation} We incorporate $SFS_{G,\varepsilon}(\rho)$ as before and notice that $1 - 2g = 2 - 2h - 2$. This proves the formula (d).
\end{proof}

The degeneration formulas (b) and (d) both involve the creation of a doublet component, whose corresponding $(G,\varepsilon)$-Doublet Hurwitz numbers is given by Definition \ref{definitionGEpsilonComplexHurwitz}. 

In formula (b), the summands corresponding to an odd conjugacy class $c$ vanish, and the left-hand-side vanishes if $\bm{c}$ or $\bm{c'}$ contain an odd conjugacy. Thus, one can replace $H^\bullet_{h'}(c,\bm{c'})$ by $H^\bullet_{h',a}(c,\bm{c'})$ for any chosen puncture class $a$ in the right-hand side of (b).

However, it is necessary to keep the sign $\varepsilon(c)$ in (d). The righ-hand side involves \begin{equation}
H^\bullet_{h,a}(\bm{c},c,c) =  \varepsilon(c) H^\bullet_{h}(\bm{c},c,c)
\end{equation} where $a$ is the puncture class consiting of a simple loop around one of the marked points created. One can replace $a$ by any puncture class defined from a marking consistent with the degeneration process. The latter are constrained by Lemma \ref{lemmaDegeneCaseD}.

\begin{remark}
Using degenerations of types (a) and (b), one is able to express any $(G,\varepsilon)$-Real Hurwitz number in terms of $\mathbb{R}H^\bullet_{0,G}(c)$ with $c \in C(G)$ and the $G$-Complex Hurwitz numbers.
\end{remark}

\begin{corollary}
If $G = \mathfrak{S}_d$ is the symmetric group of order $d$, then the formula \begin{equation}
\mathbb{R}  H^\bullet_{g,d}(\bm{\lambda};\bm{k}) = \sum_{\nu_0,\ldots,\nu_g} H^\bullet_{0,d}(\bm{\lambda},\nu_0,\ldots,\nu_g ; \bm{k}) \prod_{i=0}^g z_{\nu_i} \mathbb{R}H^\bullet_{0,d}(\nu_i)
\end{equation} holds with completed cycles.
\end{corollary}

\begin{proof}
The formula holds without completed cycles due to the degeneration formulas (a) and (b). The disconnected signed Real Hurwitz numbers and disconnected Complex Hurwitz numbers with completed cycles are linear combinations of their analogs without completed cycles. Thus, the formula holds also with completed cycles.
\end{proof}

The last degeneration formula has a surprising combinatorial consequence.

\begin{corollary}
Given a finite group $G$ whose characters are real-valued and $\varepsilon : G \rightarrow \{ \pm 1 \}$ a non-trivial morphism, \begin{equation}
\# \{ \rho \in Irr(G) \ | \ \rho^T = \rho \} = \sum_{c \in C(G)} \varepsilon(c).
\end{equation} In particular, the number of symmetric Young diagrams of size $d$ is given by \begin{equation}
\sum_{| \lambda | = d} \varepsilon(c_\lambda).
\end{equation}
\end{corollary}

\begin{proof}
The number $\mathbb{R}H^\bullet_1(\emptyset)$ is given explicitly by \begin{equation}
\mathbb{R}H^\bullet_1(\emptyset) = \sum_{\rho^T = \rho} 1 = \# \{ \rho \in Irr(G) \ | \ \rho^T = \rho \}
\end{equation} according to theorem \ref{theoremRealExplicitG}. Using a degeneration of type (d), it also reads \begin{equation}
\mathbb{R}H^\bullet_1(\emptyset) = \sum_{c \in C(G)} \varepsilon(c) z_c H^\bullet_0(c,c) = \sum_{c \in C(G)} \varepsilon(c).
\end{equation}
\end{proof}

	\subsection{Extended Frobenius algebra}

In this Subsection, we use $(G,\varepsilon)$-Real Hurwitz numbers to describe an extra structure on the Frobenius algebra $(Z \mathbb{C} G,\star,e,\eta)$ of Subsection \ref{subsectionFrobenius}. This structure is called an \textit{extended Frobenius algebra} in \cite{articleTuraevTurner}. The case of the symmetric groups with the sign morphism has been obtained in \cite{articleGeorgievaIonel} using the Real Gromov-Witten invariants. Recall the vector \begin{equation}
\Delta = \sum_{c \in C(G)} z_c c \otimes \overline{c}
\end{equation} introduced in Subsection \ref{subsectionFrobenius}.

\begin{definition}[Extended Frobenius algebra, \cite{articleTuraevTurner}]
An \textit{extended Frobenius algebra} is a Frobenius algebra $(A,\star,e,\eta)$ together with \begin{itemize}
	\item an involutive automorphism $\Omega$ of $(A,\star,e,\eta)$ and
	\item a vector $\mathcal{U} \in A$ which satisfies $\Omega(x \star \mathcal{U}) = x \star \mathcal{U}$ for all $x$, and such that $\mathcal{U} \star \mathcal{U}$ equals the product of the bivector $( \Omega \otimes id_A) (\Delta)$.
\end{itemize}
\end{definition}

We define a linear map \begin{equation}
Z_{X,\sigma,B} : Z \mathbb{C} G^{\otimes m} \rightarrow Z \mathbb{C} G^{\otimes n}
\end{equation} for any marked Real Riemann surface $(X,\sigma,B)$ with $m$ pairs of conjugated marked points considered as inputs and $n$ as outputs. In what follows, the fact that the characters of $G$ are real-valued implies that the conjugacy classes $c$ and $\overline{c}$ are always equal. The map $Z_{X,\sigma,B}$ is defined in the standard basis. If $(X,\sigma)$ is a connected Real Riemann surface of genus $g$, we set the image of $\bm{c}_1 \otimes \ldots \otimes \bm{c}_m$ to be \begin{equation}
\sum_{\bm{c'}} z_{\bm{c'}_1} \ldots z_{\bm{c'}_n} \mathbb{R} H^\bullet_{g,G,\varepsilon}(\bm{{c}},\bm{c'}) \bm{c'}_1 \otimes \ldots \otimes \bm{c'}_n.
\end{equation} If $(X,\sigma)$ is the doublet of a genus $g$ Riemann surface, we set it to be \begin{equation}
\sum_{\bm{c'}} z_{\bm{c'}_1} \ldots z_{\bm{c'}_n}  H^\bullet_{g,G,\varepsilon,a}(\bm{{c}},\bm{c'}) \bm{c'}_1 \otimes \ldots \otimes \bm{c'}_n
\end{equation} where $a$ is the puncture class associated to the marking. In particular, if the doublet has a canonical marking, then $Z_{X,\sigma,B}$ coincides with the linear map associated in Subsection \ref{subsectionFrobenius} to one of the connected components of $X$.

In the general case, we define $Z_{X,\sigma,B}$ by declaring that $Z$ transforms disjoint unions into tensor product of maps.\\

Introduce a linear map $\Omega$ and a vector $\mathcal{U}$ as $Z_{X,\sigma,B}$ for the doublet of sphere with one input, one output and the non-canonical marking, and a connected sphere with one output respectively. 

\begin{lemma}
The vector $\mathcal{U}$ is the element \begin{equation}
\mathfrak{L} = \frac{1}{\# G} \sum_{\gamma \in G} \varepsilon(\gamma) \gamma^2 = \sum_{\rho^T = \rho} SFS(\rho) \frac{\# G}{dim (\rho)} v_\rho
\end{equation} of (\ref{equationVectorL}). The linear map $\Omega$ satisfies \begin{equation}
\Omega(c) = \varepsilon(c) c \text{ or equivalently } \Omega(v_\rho) = v_{\rho^T}.
\end{equation}
\label{lemmaExpressionZ}
\end{lemma}

\begin{proof}
The first expression of $\mathcal{U}$ is a consequence of \begin{equation}
\mathbb{R}H^\bullet_{0,G}(c) = \frac{1}{\# G} \sum_{\gamma^2 \in c} \varepsilon(\gamma) 
\end{equation} and (\ref{equationVectorL}). According to Theorem \ref{theoremRealExplicitG}, we also have \begin{align}
\mathcal{U} &= \sum_{c \in C(G)} z_c \sum_{\rho^T = \rho} SFS(\rho) \frac{dim (\rho)}{\#G} f_c(\rho)  c \\
	&=  \sum_{\rho^T = \rho} SFS(\rho) \frac{\# G}{dim (\rho)} \sum_{c \in C(G)}\frac{dim(\rho)}{\# G} \chi_\rho(c) c \\
	&= \sum_{\rho^T = \rho} SFS(\rho) \frac{\# G}{dim(\rho)} v_\rho.
\end{align} The computation for $\Omega$ is straightforward \begin{align}
\Omega(c) &= \sum_{c' \in C(G)} \varepsilon(c') z_{c'} H^\bullet_{0}(c,c') c' \\
	&= \varepsilon(c) c.
\end{align} In the idempotent basis, \begin{align}
\Omega(v_\rho) &= \sum_c \frac{dim(\rho)}{\# G} \chi_{\rho^T}(c) c \\
	&= \sum_c \frac{dim(\rho^T)}{\# G} \chi_{\rho^T}(c) c \\
	&= v_{\rho^T}.
\end{align}.
\end{proof}

\begin{lemma}
In the idempotent basis, the linear maps $Z_{X,\sigma,B)}$ can be described as follows :\begin{enumerate}[label = (\arabic*)]
	\item Let $(X,\sigma,B)$ be the doublet of a genus $g$ Riemann surface with a canonical marking. Then, $Z_{X,\sigma,B}$ is the linear map \begin{equation}
v_{\rho_1} \otimes \ldots \otimes v_{\rho_m} \mapsto \left\{ \begin{array}{ll}
	\left( \frac{dim(\rho)}{\# G} \right)^{2-2g-2n} v_\rho^{\otimes n} & \text{ if } \rho_1 = \ldots = \rho_n =: \rho, \\
	0 & \text{ otherwise.}
\end{array} \right.
\label{equationZcomplex}
\end{equation}
	\item Let $(X,\sigma,B)$ be the doublet of a genus $g$ Riemann surface with puncture class $a$, of associated splitting $I \sqcup J$ of the marked points. Then, $Z_{X,\sigma,B}$ is obtained from (\ref{equationZcomplex}) by applying $\Omega$ at all the inputs and outputs corresponding to the marked points in $I$.
	\item Let $(X,\sigma,B)$ be a Real Riemann surface of genus $g$. Then, $Z_{X,\sigma,B}$ is the linear map that sends $v_{\rho_1} \otimes \ldots \otimes v_{\rho_m}$ to \begin{equation}
\left\{ \begin{array}{ll}
	\left( SFS_{G,\varepsilon}(\rho) \frac{dim(\rho)}{\# G} \right)^{1-g-2n} v_\rho^{\otimes n} & \text{ if } \rho_1 = \ldots = \rho_n =: \rho = \rho^T, \\
	0 & \text{ otherwise.}
\end{array} \right.
\label{equationZreal}
\end{equation}
\end{enumerate} 
\end{lemma}

\begin{proof}
The formula (1) is a standard computation in Complex Hurwitz theory. Changing the order at the $i$-th pair of marked points multiplies the summands by the sign of the conjugacy class corresponding to this pair of marked points. This is exactly what the composition by $\Omega$ does. It proves the formula (2). The proof of the formula (3) is similar to (1).
\end{proof}

\begin{proposition}
The involution $\Omega$ and the vector $\mathcal{U}$ provide $(Z\mathbb{C}G,\star,e,\eta)$ with the structure of an extended Frobenius algebra.
\label{propositionExtendedFrobenius}
\end{proposition}

\begin{proof}
We perform all the computations in the idempotent basis. It is straightforward to check that $\Omega$ is involutive and that $Z_{X,\sigma,B}$ is invariant when $\Omega$ is composed at every input and output whenever $(X,\sigma,B)$ is a doublet. Thus, $\Omega$ is an involutive automorphism of $(Z\mathbb{C}G,\star,e,\eta)$.

The relation $\Omega(v_\rho \star \mathcal{U}) = v_\rho \star \mathcal{U}$ holds since both sides equal \begin{equation}
SFS(\rho) \frac{\# G}{dim (\rho)} v_\rho
\end{equation} if $\rho$ is symmetric and vanish otherwise. Finally, both $\mathcal{U} \star \mathcal{U}$ and the product of the bivector $( \Omega \otimes id_A) (\Delta)$ equal \begin{equation}
\sum_{\rho^T = \rho} \left( \frac{dim(\rho)}{\# G} \right)^{-2} v_\rho. 
\end{equation}
\end{proof}

Proposition \ref{propositionDegene} is equivalent to a functoriality result on the linear maps $Z_{X,\sigma,B}$. It is stated in Proposition \ref{propositionFunctorZ}. The proof is written in the idempotent basis. 

\begin{proposition}
Let $(X,\sigma,B)$ be a marked connected Real Riemann surface and $(Y,\sigma,C)$ its degeneration along a pair conjugated circles. Then, $Z_{X,\sigma,B} : A^{\otimes m} \rightarrow A^{\otimes n}$ is obtained from $Z_{Y,\sigma,C} : A^{\otimes m+1} \rightarrow A^{\otimes n+1}$ by contracting the additional input and output created by the degeneration. 
\label{propositionFunctorZ}
\end{proposition}

\begin{proof}
The four types of degenerations must be treated seperately. They are similar so that we focus on the degenerations of type (d). Let $(X,\sigma,B)$ be a connected Real Riemann surface of genus $2g+1$ with $m$ inputs and $n$ inputs. Its degeneration $(Y,\sigma,C)$ of type (d) is the doublet of a genus $g$ Riemann surface with $m+1$ inputs and $n+1$ outputs. By Lemma \ref{lemmaDegeneCaseD}, the created input and output do not belong to the same part of the partition $I \sqcup J$ induced by the marking. By Lemma \ref{lemmaExpressionZ}, the contraction of the created input and output is therefore \begin{equation}
v_{\rho_1} \otimes \ldots \otimes v_{\rho_m} \mapsto \left\{ \begin{array}{ll}
	\left( \frac{dim(\rho)}{\# G} \right)^{2-2g-2(n+1)} v_\rho^{\otimes n} & \text{ if }\rho_1 = \ldots = \rho_n =: \rho = \rho^T, \\
	0 & \text{ otherwise.}
\end{array} \right.
\end{equation} Since the power is even, one can add the factor $SFS(\rho) \in \{ \pm 1 \}$ in the parenthesis. Writting the exponent as $1-(2g+1)-2n$ identifies the contraction with $Z_{X,\sigma,B}$ as required.
\end{proof}

For instance, the identity between $\mathcal{U} \star \mathcal{U}$ and the product of the bivector $( \Omega \otimes id_A) (\Delta)$ can be interpreted as two different degenerations of a connected Real Riemann surface of genus $1$, as pictured for the cases (a) and (d) of Figure \ref{figureDegene}.

	\subsection{Degeneration of the signs}
	\label{subsectionDegenerationSigns}

The signed Real Hurwitz numbers, connected or not, have been first obtained in \cite{articleGeorgievaIonel} as relative Real Gromov-Witten invariants. In both cases, they are obtained as a signed weighted sum of ramified Real covers meeting the same conditions as in Definition \ref{definitionRealHurwitz}, with the same weight of $1 / \# \mathrm{Aut}(f)$. However, their signs do not have an explicit expression. In the present subsection, we prove that the signs introduced in \cite{articleGeorgievaIonel} coincide with those of Definition \ref{definitionRealHurwitz}.

The method is the following. The signs coincide for the Real ramified covers described in Example \ref{exampleP1}. Thus, we degenerate the target to obtain Real ramified covers over targets isomorphic to the one of Example \ref{exampleP1}. We use Lemma \ref{lemmaFuntorialityAdm} to relate the sign of the original cover to the sign of its degeneration. It leads to Corollary \ref{corollaryProductSign}. The product formula of Corollary \ref{corollaryProductSign} also holds in Real Gromov-Witten theory. Therefore, the signs coincide, see Theorem \ref{theoremComparSign}.\\

Let $(X,\sigma,B)$ be a connected marked Real Riemann surface. Contracting a pair of conjugate circles in $X \setminus B$ creates a new target $(Y,\sigma,C)$ as described in Subsection \ref{subsectionDegenerations}. The space $Y_\sigma^o$ is naturally homeomorphic to $X_\sigma^o$ with a circle removed. Thus, there is a natural inclusion $i : Y_\sigma^o \rightarrow X_\sigma^o$. Each connected component of $Y_\sigma^o$ produces an element of $H_1(Y_\sigma^o)$, which is an arbitrary admissible class for non-orientable components and a puncture class for orientable components. Denote their sum by $a$. The next lemma discusses whether $i_* a$ is admissible or not. There are four cases to consider, which correspond to the four cases of degenerations described in Subsection \ref{subsectionDegenerations}. The pictures to have in mind are represented in Figure \ref{figureDegene}.

\begin{lemma}
The homology class $i_* a \in H_1(X_\sigma^o)$ is admissible in the cases (a),(b),(d). In the case (c), there exists $x \in \mathbb{Z} / 2 \mathbb{Z}$ such that $x \delta+ i_* a$ is admissible, where $\delta$ is the homology class of the removed circle. 
\label{lemmaFuntorialityAdm}
\end{lemma}

\begin{proof}
In the cases (a),(b),(c), we choose a presentation (\ref{equationPresentation}) of the fundamental group of each connected component of $Y_\sigma^o$. In the case (d), we have to use also the presentation (\ref{equationPresentationBis}) of Remark \ref{remarkPresentationBis}. The map $i_*$ in homology can then be described explicitly and we use Lemma \ref{lemmaPoincareDual} to conclude. The presentations are consistent with Figure \ref{figureDegene}. \begin{enumerate}[label = (\alph*)]
	\item $(Y,\sigma)$ is the union of two connected Real Riemann surfaces of respective genus $h$ and $h'$ with $g = h + h' + 1$. We choose a presentation of the two fundamental groups. Their first homology is generated respectively by the classes $\bm{\gamma},\bm{\delta},\delta$ and $\bm{\gamma'},\bm{\delta'},\delta'$. Here, $\delta$ and $\delta'$ denote the homology class of a simple loop around the new puncture. We can make a consistent choice of presentation of the fundamental group of $X_\sigma^o$, such that its first homlogy group is generated by the classes $\bm{\gamma},\bm{\delta},\bm{\gamma'},\bm{\delta'}$ and $i_*$ sends the classes $\bm{\gamma},\bm{\delta},\bm{\gamma'},\bm{\delta'}$ in $H_1(Y_\sigma^o)$ to the classes $\bm{\gamma},\bm{\delta},\bm{\gamma'},\bm{\delta'}$ in $H_1(X_\sigma^o)$. By construction, the class $a$ can be expressed as \begin{equation}
	a = \sum_{i=0}^h \gamma_i + \sum_{i=0}^{h'} \gamma'_i + \sum_{i=1}^n x_i \delta_i + \sum_{i=1}^{n'} x'_i \delta'_i + x\delta + x' \delta'
	\end{equation} with $x,x',x_i,x'_i$ in $\mathbb{Z} / 2 \mathbb{Z}$. Since the relations $\delta = \sum_i \delta_i$ and $\delta' = \sum_i \delta'_i$ hold in $H_1(Y_\sigma^o)$, we can supress $\delta$ and $\delta'$ from the expression of $a$. The image $i_* a$ is as required in Lemma \ref{lemmaPoincareDual}. 
	\item $(Y,\sigma)$ is the union of a connected Real Riemann surface and a doublet. The proof is parallel to (a).
	\item $(Y,\sigma)$ is a connected Real Riemann surface of genus $h$ with $g = h+2$. In this case, we choose the presentations such that $H_1(Y_\sigma^o)$ is generated by $\bm{\gamma},\bm{\delta},\delta,\delta'$, $H_1(X_\sigma^o)$ is generated by $\alpha,\beta,\bm{\gamma},\bm{\delta}$ and $i_*$ sends the classes $\bm{\gamma},\bm{\delta}$ to themselves. The homology class $a$ reads \begin{equation}
	a = \sum_{i=0}^h \gamma_i + \sum_{i=1}^n x_i \delta_i + x\delta + x' \delta'
	\end{equation} with $x,x',x_i$ in $\mathbb{Z} / 2 \mathbb{Z}$. Now, the relation $\delta + \delta' = \sum_i \delta_i$ holds in $H_1(Y^o_\sigma)$. Thus, $i_* a$ is admissible if $x=x'$, otherwise $\delta + i_* a$ admissible.
	\item $(Y,\sigma)$ is the double of Riemann surface of genus $h$ with $g = 2h + 1$. We choose the presentations such that $H_1(Y_\sigma^o)$ is generated by $\bm{\alpha},\bm{\beta},\bm{\delta},\delta,\delta'$ and $H_1(X_\sigma^o)$ is generated by $\bm{\alpha},\bm{\beta},\bm{\delta},\zeta,\xi$. The presentation for $X_\sigma^o$ refers to the presentation (\ref{equationPresentationBis}). The map $i_*$ sends the classes $\bm{\alpha},\bm{\beta},\bm{\delta}$ to themselves and $\delta,\delta'$ to $\zeta$. According to Lemma \ref{lemmaDegeneCaseD}, the puncture class can be written as \begin{equation}
	a = \delta + \sum_{i=1}^n x_i \delta_i
\end{equation} with $x_i$ in $\mathbb{Z} / 2 \mathbb{Z}$. Using the expression for admissible classes provided in Remark \ref{remarkPresentationBis}, the class $i_* a$ is admissible. 
\end{enumerate}
\end{proof}

Take a ramified Real cover $f : (X',\sigma') \rightarrow (X,\sigma)$ unramified over $X \setminus B$. As the complex structure of the target varies, the complex structure of the source varies so that the map $f$ remains a ramified Real cover. In particular, at the limit of a degeneration process, one obtains a ramified Real cover $\tilde{f} : (Y',\sigma') \rightarrow (Y,\sigma)$ unramified over $Y \setminus C$. In Corollary \ref{corollaryProductSign}, we use Lemma \ref{lemmaFuntorialityAdm} to express the sign of $f$ in terms of the signs of the ramified Real covers obtained after a particular sequence of degenerations of the target. This sequence of degenerations is depicted in Figure \ref{figureDegeneSign}.

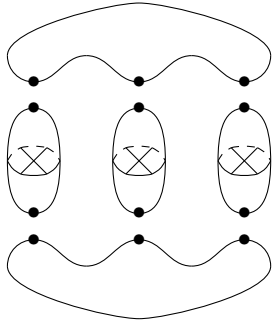
\begin{figure}[h!]
\centering
\begin{tikzpicture}[scale = 0.7]
		\node (18) at (-2.5, 0) {};
		\node (19) at (-1.5, 0) {};
		\node (20) at (-2, 1) {};
		\node (21) at (-2, -1) {};
		\node (22) at (-0.5, 0) {};
		\node (23) at (0.5, 0) {};
		\node (24) at (0, 1) {};
		\node (25) at (0, -1) {};
		\node (70) at (-2, 1) {};
		\node (71) at (-2, -1) {};
		\node (72) at (0, 1) {};
		\node (73) at (0, -1) {};
		\node (74) at (1.5, 0) {};
		\node (75) at (2.5, 0) {};
		\node (76) at (2, 1) {};
		\node (77) at (2, -1) {};
		\node (78) at (2, 1) {};
		\node (79) at (2, -1) {};
		\node (80) at (-2, 1.5) {};
		\node (81) at (0, 1.5) {};
		\node (82) at (2, 1.5) {};
		\node (83) at (-1, 2) {};
		\node (84) at (1, 2) {};
		\node (85) at (-2.5, 2) {};
		\node (86) at (2.5, 2) {};
		\node (87) at (0, 3) {};
		\node (88) at (0, -1.5) {};
		\node (89) at (2, -1.5) {};
		\node (90) at (-2, -1.5) {};
		\node (91) at (-2.5, -2) {};
		\node (92) at (2.5, -2) {};
		\node (93) at (-1, -2) {};
		\node (94) at (1, -2) {};
		\node (95) at (0, -3) {};
		\node (96) at (-2.5, 0) {};
		\node (97) at (-1.5, 0) {};
		\node (98) at (-2.25, 0.25) {};
		\node (99) at (-1.75, 0.25) {};
		\node (100) at (-2.25, -0.25) {};
		\node (101) at (-1.75, -0.25) {};
		\node (102) at (-0.5, 0) {};
		\node (103) at (0.5, 0) {};
		\node (104) at (-0.25, 0.25) {};
		\node (105) at (0.25, 0.25) {};
		\node (106) at (-0.25, -0.25) {};
		\node (107) at (0.25, -0.25) {};
		\node (108) at (1.5, 0) {};
		\node (109) at (2.5, 0) {};
		\node (110) at (1.75, 0.25) {};
		\node (111) at (2.25, 0.25) {};
		\node (112) at (1.75, -0.25) {};
		\node (113) at (2.25, -0.25) {};
		\node (114) at (-2, 1.5) {$\bullet$};
		\node (115) at (-2, 1) {$\bullet$};
		\node (116) at (0, 1.5) {$\bullet$};
		\node (117) at (0, 1) {$\bullet$};
		\node (118) at (2, 1.5) {$\bullet$};
		\node (119) at (2, 1) {$\bullet$};
		\node (120) at (-2, -1) {$\bullet$};
		\node (121) at (-2, -1.5) {$\bullet$};
		\node (122) at (0, -1) {$\bullet$};
		\node (123) at (0, -1.5) {$\bullet$};
		\node (124) at (2, -1) {$\bullet$};
		\node (125) at (2, -1.5) {$\bullet$};
		\draw [in=180, out=90] (18.center) to (20.center);
		\draw [in=90, out=0] (20.center) to (19.center);
		\draw [in=0, out=-90] (19.center) to (21.center);
		\draw [in=180, out=-90] (18.center) to (21.center);
		\draw [in=180, out=90] (22.center) to (24.center);
		\draw [in=90, out=0] (24.center) to (23.center);
		\draw [in=0, out=-90] (23.center) to (25.center);
		\draw [in=180, out=-90] (22.center) to (25.center);
		\draw [in=180, out=90] (74.center) to (76.center);
		\draw [in=90, out=0] (76.center) to (75.center);
		\draw [in=0, out=-90] (75.center) to (77.center);
		\draw [in=180, out=-90] (74.center) to (77.center);
		\draw [in=-180, out=-90] (85.center) to (80.center);
		\draw [in=180, out=0] (80.center) to (83.center);
		\draw [in=-180, out=0] (83.center) to (81.center);
		\draw [in=-180, out=0] (81.center) to (84.center);
		\draw [in=180, out=0] (84.center) to (82.center);
		\draw [in=-90, out=0] (82.center) to (86.center);
		\draw [bend left=45] (91.center) to (90.center);
		\draw [in=180, out=0] (90.center) to (93.center);
		\draw [in=180, out=0] (93.center) to (88.center);
		\draw [in=180, out=0] (88.center) to (94.center);
		\draw [in=-180, out=0] (94.center) to (89.center);
		\draw [bend left=45] (89.center) to (92.center);
		\draw [in=0, out=-90, looseness=0.50] (92.center) to (95.center);
		\draw [in=270, out=-180, looseness=0.50] (95.center) to (91.center);
		\draw [in=180, out=90, looseness=0.50] (85.center) to (87.center);
		\draw [in=90, out=0, looseness=0.50] (87.center) to (86.center);
		\draw [bend left=15, looseness=1.25, dashed] (96.center) to (98.center);
		\draw [bend left=15, looseness=0.75, dashed] (98.center) to (99.center);
		\draw [bend left=15, looseness=1.25, dashed] (99.center) to (97.center);
		\draw [bend left=15, looseness=1.25] (97.center) to (101.center);
		\draw [bend left=15, looseness=0.75] (101.center) to (100.center);
		\draw [bend left=15, looseness=1.25] (100.center) to (96.center);
		\draw (98.center) to (101.center);
		\draw (100.center) to (99.center);
		\draw [bend left=15, looseness=1.25, dashed] (102.center) to (104.center);
		\draw [bend left=15, looseness=0.75, dashed] (104.center) to (105.center);
		\draw [bend left=15, looseness=1.25, dashed] (105.center) to (103.center);
		\draw [bend left=15, looseness=1.25] (103.center) to (107.center);
		\draw [bend left=15, looseness=0.75] (107.center) to (106.center);
		\draw [bend left=15, looseness=1.25] (106.center) to (102.center);
		\draw (104.center) to (107.center);
		\draw (106.center) to (105.center);
		\draw [bend left=15, looseness=1.25, dashed] (108.center) to (110.center);
		\draw [bend left=15, looseness=0.75, dashed] (110.center) to (111.center);
		\draw [bend left=15, looseness=1.25, dashed] (111.center) to (109.center);
		\draw [bend left=15, looseness=1.25] (109.center) to (113.center);
		\draw [bend left=15, looseness=0.75] (113.center) to (112.center);
		\draw [bend left=15, looseness=1.25] (112.center) to (108.center);
		\draw (110.center) to (113.center);
		\draw (112.center) to (111.center);
\end{tikzpicture}
\caption{Degeneration of a connected Real Riemann surface used in Corollary \ref{corollaryProductSign}. The sign of a ramified Real cover with target the original surface is the product of the signs assigned to the degenerated covers with target the Real Riemann spheres.}
\label{figureDegeneSign}
\end{figure}

\begin{corollary}
Let $(X,\sigma,B)$ be a marked connected Real Riemann surface of genus $g$. Consider a sequence of degenerations of $(X,\sigma,B)$ into $(Y,\sigma,C)$  which is the union of : \begin{itemize}
	\item $g+1$ connected Real Riemann surfaces $(X_0,\sigma_0,B_0),\ldots,(X_g,\sigma_g,B_g)$ isomorphic to $(\mathbb{P}^1_\mathbb{C},\theta,\{[1:0],[0:1]\})$ as in Example \ref{exampleP1} and
	\item the doublet $(D,\sigma_D,C_D)$ of a sphere with $g+n+1$ marked points. 
\end{itemize} Let $f : (X',\sigma') \rightarrow (X,\sigma)$ a ramified Real cover, unramified over $X \setminus B$. Suppose that its ramification profile $\bm{\lambda}_i$ around $b_i^\pm$ is even for all $i$. Then \begin{equation}
\varepsilon \rho_f(a) = \varepsilon \rho_{f_0}(a_0) \ldots \varepsilon \rho_{f_g}(a_g)
\end{equation} where $a,a_0,\ldots,a_g$ are any admissible classes and $f_i$ is the degeneration of $f$ over $(X_i,\sigma_i,B_i)$.
\label{corollaryProductSign}
\end{corollary}

\begin{proof}
Denote also by $f_D$ the degeneration of $f$ over the doublet $(D,\sigma_D,C_D)$. The diagram \begin{equation}
	\begin{tikzcd}
		H_1(Y_\sigma^o) \ar[rr,"i_*"] \ar[dr] & & H_1(X_\sigma^o) \ar[dl] \\
		& \{ \pm 1 \} &
	\end{tikzcd}
	\end{equation} commutes, where the diagonal arrows are $\varepsilon \rho_{f_D}\varepsilon \rho_{f_0} \ldots \varepsilon \rho_{f_g}$ and $\varepsilon \rho_f$. Choosing admissible classes as in the statement and denoting by $a'$ the puncture class of $(Y',\sigma,C')$, we can use Lemma \ref{lemmaFuntorialityAdm} to obtain 	\begin{equation}
	\varepsilon \rho_f(a) = \varepsilon \rho_{f_D}(a')\varepsilon \rho_{f_0}(a_0) \ldots \varepsilon \rho_{f_g}(a_g).
	\end{equation} Indeed, the degenerations involved are of type (a),(b) and the left-hand side does not depend on the admissible class $a$. 
	
	We claim that the ramification profiles of $f_D$ are all even. For those that correspond to the points in $B$, it is part of the assumptions of the corollary. The others correspond to collapsed circles. The monodromy around them is the same as the monodromy around the only pair of marked points of one of the $(X_i,\sigma_i,B_i)$. Those have been classified in Example \ref{exampleP1}, and they are all even. 
	
	As a consequence, the function $\varepsilon \rho_{f_D}$ is constant equal to $1$. This proves the statement.
\end{proof}

\begin{theorem}
Let $(X,\sigma,B)$ be a marked connected Real Riemann surface and $f : (X',\sigma') \rightarrow (X,\sigma)$ a ramified Real cover, unramified over $X \setminus B$. Suppose that the ramification $\bm{\lambda}_i$ around $b_i^\pm$ is even for all $i$. Then the sign associated to $f$ in \cite{articleGeorgievaIonel} is $\varepsilon \rho_f (a)$ for any admissible class $a$. 
\label{theoremComparSign}
\end{theorem}

\begin{proof}
Through the degeneration of Corollary \ref{corollaryProductSign}, the signs of \cite{articleGeorgievaIonel} enjoy the same product formula, see \cite{articleGeorgievaIonel2}. Therefore, it is enough to compare them over the connected Real Riemann surface $(\mathbb{P}^1_\mathbb{C},\sigma,B)$ of Example \ref{exampleP1}. Both signs transform disjoint unions of covers into product of signs, so that it is enough to check that the signs match if the source is either a doublet or a connected Real Riemann surface. Those are classified in Example \ref{exampleP1}. With the signs of \cite{articleGeorgievaIonel}, the cover by a doublet (corresponding to the partition $(d,d)$ of $2d$) is counted negatively while the cover by a $(\mathbb{P}^1_\mathbb{C},\sigma)$ (corresponding to the partition $(d)$ of $d$ with odd $d$) is counted positively. This is exactly what we have computed in Example \ref{exampleP1}.    
\end{proof}

\section{Extension to non-empty fixed locus}	
\label{sectionExtensionNonEmpty}

In \cite{articleGeorgievaIonel}, the signed Real Hurwitz numbers are defined from Real Gromov-Witten theory using a target which may have a non-empty fixed locus. It turns out that the signed count of ramified Real covers of a connected Real Riemann surface does not depend on the structure of its fixed locus. In Subsection \ref{subsectionDefinitionSign}, we have provided a topological definition of the sign when the fixed locus of the target is empty. The purpose of the present section is to define signs when the fixed locus of the target is not empty, with the constraint that the signed weighted sum over ramified Real covers is still the signed Real Hurwitz numbers of Section \ref{subsectionDefinitionSign}. As in \cite{articleGeorgievaIonel}, we restrict to the case of covers unramified over the fixed locus, which means that we consider marked connected Real Riemann surfaces $(X,\sigma,B)$ with $X^\sigma \cap B = \emptyset$. A case with Real branch points is studied in \cite{articleItenbergZvonkine}.
	
	\subsection{$\mathbb{P}^1_\mathbb{C}$ with a pair of conjugate points}
	\label{subsectionFixedLocusP1}

In this subsection, we focus on the target $\mathbb{P}^1_\mathbb{C}$ with the involution \begin{equation}
\tilde{\theta} : [x:y] \mapsto [\overline{y}: \overline{x} ]
\end{equation} and the marked points $B_{\mathbb{P}^1_\mathbb{C}} = \{[0:1],[1:0]\}$. The fixed locus is the circle \begin{equation}
\{ [x:y ] \ | \ |x| = |y | \}.
\end{equation} Classification of ramified Real covers $f : (X',\sigma') \rightarrow (\mathbb{P}^{1}_\mathbb{C},\tilde{\theta})$ unramified over $\mathbb{P}^1_\mathbb{C} \setminus B_{\mathbb{P}^1_\mathbb{C}}$, with $(X',\sigma')$ either a connected Real Riemann surface or a doublet, is described in Example \ref{exampleP1Fixed}.

\begin{example} \begin{itemize}
	\item The cover obtained by doubling $[x:y] \mapsto [x^d:y^d]$ accounts for the partition $\lambda = (d,d)$ of $2d$. Its automorphism group has cardinality $2d$. It is counted with a sign $-1$ in \cite{articleGeorgievaIonel}.
	\item The cover \begin{equation}
	f : \left\{ \begin{array}{lll}
		(\mathbb{P}^1_\mathbb{C},\tilde{\theta}) & \rightarrow & (\mathbb{P}^1_\mathbb{C},\tilde{\theta})\\ \relax
		[x:y] & \mapsto & [x^d:y^d]
	\end{array} \right.
	\end{equation} accounts for the partition $(d)$ of $d$. Its automorphism group has cardinality $d$. It is counted with a sign $+1$ in \cite{articleGeorgievaIonel}.
	\item The cover \begin{equation}
	f : \left\{ \begin{array}{lll}
		(\mathbb{P}^1_\mathbb{C},{\theta}) & \rightarrow & (\mathbb{P}^1_\mathbb{C},\tilde{\theta})\\ \relax
		[x:y] & \mapsto & [x^d:y^d]
	\end{array} \right.
	\end{equation} for $d$ even accounts for the partition $(d)$ of $d$, where the involution $\theta$ is as in Example \ref{exampleP1}. Its automorphism group has cardinality $d$. It is counted with a sign $-1$ in \cite{articleGeorgievaIonel}.
\end{itemize} With this choice of signs, we do recover that \begin{equation}
\mathbb{R}H_{0,d}(\lambda) = \left\{ \begin{array}{ll}
	1 / d & \text{ if } \lambda =  (d) \text{ with } d \text{ odd,} \\
	-1/2d' & \text{ if } \lambda =  (d',d') \text{ with } d = 2d',\\
	0 & \text{ otherwise.}
\end{array} \right.
\end{equation}
\label{exampleP1Fixed}
\end{example}

\begin{figure}[h!]
\centering
\begin{tikzpicture}[scale = 0.7]
		\node (6) at (-0.5, 0) {};
		\node (7) at (0.5, 0) {};
		\node (8) at (0, 1) {$\bullet$};
		\node (9) at (0, -1) {$\bullet$};
		\node (10) at (2, 0) {};
		\node (11) at (3, 0) {};
		\node (12) at (2.5, 1) {$\bullet$};
		\node (13) at (2.5, -1) {$\bullet$};
		\node (14) at (2, 4) {};
		\node (15) at (3, 4) {};
		\node (16) at (2.5, 5) {};
		\node (17) at (2.5, 3) {};
		\node (18) at (-5, 4) {};
		\node (19) at (-4, 4) {};
		\node (20) at (-4.5, 5) {};
		\node (21) at (-4.5, 3) {};
		\node (22) at (-3, 4) {};
		\node (23) at (-2, 4) {};
		\node (24) at (-2.5, 5) {};
		\node (25) at (-2.5, 3) {};
		\node (26) at (-4, 0) {};
		\node (27) at (-3, 0) {};
		\node (28) at (-3.5, 1) {$\bullet$};
		\node (29) at (-3.5, -1) {$\bullet$};
		\node (30) at (-0.5, 4) {};
		\node (31) at (0.5, 4) {};
		\node (32) at (0, 5) {};
		\node (33) at (0, 3) {};
		\node (34) at (-0.5, 0) {};
		\node (35) at (0.5, 0) {};
		\node (36) at (-0.25, 0.25) {};
		\node (37) at (0.25, 0.25) {};
		\node (38) at (-0.25, -0.25) {};
		\node (39) at (0.25, -0.25) {};
		\node (40) at (2, 0) {};
		\node (41) at (3, 0) {};
		\node (42) at (2.25, 0.25) {};
		\node (43) at (2.75, 0.25) {};
		\node (44) at (2.25, -0.25) {};
		\node (45) at (2.75, -0.25) {};
		\node (46) at (-4, 0) {};
		\node (47) at (-3, 0) {};
		\node (48) at (-3.75, 0.25) {};
		\node (49) at (-3.25, 0.25) {};
		\node (50) at (-3.75, -0.25) {};
		\node (51) at (-3.25, -0.25) {};
		\node (52) at (-0.5, 4) {};
		\node (53) at (0.5, 4) {};
		\node (54) at (-0.25, 4.25) {};
		\node (55) at (0.25, 4.25) {};
		\node (56) at (-0.25, 3.75) {};
		\node (57) at (0.25, 3.75) {};
		\node (58) at (2, 4) {};
		\node (59) at (3, 4) {};
		\node (60) at (2.25, 4.25) {};
		\node (61) at (2.75, 4.25) {};
		\node (62) at (2.25, 3.75) {};
		\node (63) at (2.75, 3.75) {};
		\node (64) at (-3.5, 2.25) {};
		\node (65) at (-3.5, 1.75) {};
		\node (66) at (2.5, 2.25) {};
		\node (67) at (2.5, 1.75) {};
		\node (68) at (0, 2.25) {};
		\node (69) at (0, 1.75) {};
		\node (70) at (-4.5, 5) {$\bullet$};
		\node (71) at (-4.5, 3) {$\bullet$};
		\node (72) at (-2.5, 5) {$\bullet$};
		\node (73) at (-2.5, 3) {$\bullet$};
		\node (74) at (0, 5) {$\bullet$};
		\node (75) at (0, 3) {$\bullet$};
		\node (76) at (2.5, 5) {$\bullet$};
		\node (77) at (2.5, 3) {$\bullet$};
		\node (78) at (-3.5, 1) {};
		\node (79) at (-3.5, -1) {};
		\node (80) at (0, 1) {};
		\node (81) at (0, -1) {};
		\node (82) at (2.5, 1) {};
		\node (83) at (2.5, -1) {};
		\draw [in=180, out=90] (6.center) to (8.center);
		\draw [in=90, out=0] (8.center) to (7.center);
		\draw [in=0, out=-90] (7.center) to (9.center);
		\draw [in=180, out=-90] (6.center) to (9.center);
		\draw [in=180, out=90] (10.center) to (12.center);
		\draw [in=90, out=0] (12.center) to (11.center);
		\draw [in=0, out=-90] (11.center) to (13.center);
		\draw [in=180, out=-90] (10.center) to (13.center);
		\draw [in=180, out=90] (14.center) to (16.center);
		\draw [in=90, out=0] (16.center) to (15.center);
		\draw [in=0, out=-90] (15.center) to (17.center);
		\draw [in=180, out=-90] (14.center) to (17.center);
		\draw [in=180, out=90] (18.center) to (20.center);
		\draw [in=90, out=0] (20.center) to (19.center);
		\draw [in=0, out=-90] (19.center) to (21.center);
		\draw [in=180, out=-90] (18.center) to (21.center);
		\draw [in=180, out=90] (22.center) to (24.center);
		\draw [in=90, out=0] (24.center) to (23.center);
		\draw [in=0, out=-90] (23.center) to (25.center);
		\draw [in=180, out=-90] (22.center) to (25.center);
		\draw [in=180, out=90] (26.center) to (28.center);
		\draw [in=90, out=0] (28.center) to (27.center);
		\draw [in=0, out=-90] (27.center) to (29.center);
		\draw [in=180, out=-90] (26.center) to (29.center);
		\draw [in=180, out=90] (30.center) to (32.center);
		\draw [in=90, out=0] (32.center) to (31.center);
		\draw [in=0, out=-90] (31.center) to (33.center);
		\draw [in=180, out=-90] (30.center) to (33.center);
		\draw [bend left=15, looseness=1.25, dashed, very thick] (34.center) to (36.center);
		\draw [bend left=15, looseness=0.75, dashed, very thick] (36.center) to (37.center);
		\draw [bend left=15, looseness=1.25, dashed, very thick] (37.center) to (35.center);
		\draw [bend left=15, looseness=1.25, very thick] (35.center) to (39.center);
		\draw [bend left=15, looseness=0.75, very thick] (39.center) to (38.center);
		\draw [bend left=15, looseness=1.25, very thick] (38.center) to (34.center);
		\draw [bend left=15, looseness=1.25, dashed, very thick] (40.center) to (42.center);
		\draw [bend left=15, looseness=0.75, dashed, very thick] (42.center) to (43.center);
		\draw [bend left=15, looseness=1.25, dashed, very thick] (43.center) to (41.center);
		\draw [bend left=15, looseness=1.25, very thick] (41.center) to (45.center);
		\draw [bend left=15, looseness=0.75, very thick] (45.center) to (44.center);
		\draw [bend left=15, looseness=1.25, very thick] (44.center) to (40.center);
		\draw [bend left=15, looseness=1.25, dashed, very thick] (46.center) to (48.center);
		\draw [bend left=15, looseness=0.75, dashed, very thick] (48.center) to (49.center);
		\draw [bend left=15, looseness=1.25, dashed, very thick] (49.center) to (47.center);
		\draw [bend left=15, looseness=1.25, very thick] (47.center) to (51.center);
		\draw [bend left=15, looseness=0.75, very thick] (51.center) to (50.center);
		\draw [bend left=15, looseness=1.25, very thick] (50.center) to (46.center);
		\draw [bend left=15, looseness=1.25, dashed, very thick] (52.center) to (54.center);
		\draw [bend left=15, looseness=0.75, dashed, very thick] (54.center) to (55.center);
		\draw [bend left=15, looseness=1.25, dashed, very thick] (55.center) to (53.center);
		\draw [bend left=15, looseness=1.25, very thick] (53.center) to (57.center);
		\draw [bend left=15, looseness=0.75, very thick] (57.center) to (56.center);
		\draw [bend left=15, looseness=1.25, very thick] (56.center) to (52.center);
		\draw [bend left=15, looseness=1.25, dashed] (58.center) to (60.center);
		\draw [bend left=15, looseness=0.75, dashed] (60.center) to (61.center);
		\draw [bend left=15, looseness=1.25, dashed] (61.center) to (59.center);
		\draw [bend left=15, looseness=1.25] (59.center) to (63.center);
		\draw [bend left=15, looseness=0.75] (63.center) to (62.center);
		\draw [bend left=15, looseness=1.25] (62.center) to (58.center);
		\draw (60.center) to (63.center);
		\draw (62.center) to (61.center);
		\draw [->] (64.center) to (65.center);
		\draw [->] (66.center) to (67.center);
		\draw [->] (68.center) to (69.center);
\end{tikzpicture}
\caption{}
\label{figureCoverP1Fixed}
\end{figure}

We take the signs of \cite{articleGeorgievaIonel} as a definition in those cases. By declaring that the sign of a disjoint union of covers is the product of each of their signs, it fixes the sign of any ramified Real cover $f : (X',\sigma') \rightarrow (\mathbb{P}^{1}_\mathbb{C},\tilde{\theta})$ unramified over $\mathbb{P}^1_\mathbb{C} \setminus B_{\mathbb{P}^1_\mathbb{C}}$.

We shall describe those signs in terms of of the monodromy data of the ramified Real covers. Choose a base point $p$ \textit{in the fixed locus}. Given a ramified Real cover $f : (X',\sigma') \rightarrow (\mathbb{P}^{1}_\mathbb{C},\tilde{\theta})$ unramified over $\mathbb{P}^1_\mathbb{C} \setminus B_{\mathbb{P}^1_\mathbb{C}}$, we have a permutation $\epsilon$ of the fiber, of cycle type $\lambda$, by considering a loop around the fixed locus, and an involution $\tau$ of the fiber by considering the restriction of $\sigma'$. They satisfy \begin{equation}
\epsilon \tau = \tau \epsilon.
\end{equation} To such a pair, we associate the following numbers : \begin{itemize}
	\item Since $\epsilon$ and $\tau$ commute, $\tau$ preserves by conjugation the set $\mathcal{C}_i$ of $i$-cycles in $\epsilon$. Denote by $\tau_i$ the involution of $\mathcal{C}_i$ obtained and $m_i = \# \mathcal{C}_i$.
	\item $\tau_i$ being an involution, it has cycle type $(2,\ldots,2,1,\ldots,1)$. Denote by $k_i$ the number of $2$ appearing, so that the conjugation by $\tau$ preserves $m_i - 2k_i$ $i$-cycles of $\epsilon$. 
	\item For each cycle $c$ fixed by $\tau$, we can consider the restriction of $\tau$ on it. If $i$ is odd, it must  be the identity, but if $i= 2j$, it might be the identity or $c^j$. Denote by $a_i$ the number of those $i$-cycles on which $\tau$ is the identity, and $b_i = m_i - 2k_i - a_i$.
\end{itemize} As the proof of Lemma \ref{lemmaFormulaSignRealLocus} shows, $\sum_i k_i$ and $\sum_j b_{2j}$ correspond respectively to the number of doublets and of $(\mathbb{P}^{1}_\mathbb{C},\theta)$ in the source.

\begin{lemma}
The sign of $f$ is \begin{equation}
s(f) = (-1)^{\sum_{i=1}^\infty k_i + \sum_{j=1}^\infty b_{2j}}.
\label{equationFormulaSignRealLocus}
\end{equation}
\label{lemmaFormulaSignRealLocus}
\end{lemma}

\begin{proof}
The formula (\ref{equationFormulaSignRealLocus}) behaves well under disjoint unions of covers. Indeed, we have \begin{equation}
s(f \sqcup f') = s(f) s(f').
\end{equation} Since we have defined the signs to satisfy this formula, it is enough to prove that $s(f)$ is the required sign in the three cases of Example \ref{exampleP1Fixed}. \begin{itemize}
	\item If the source is a doublet, we can label the fiber as $\{ 1_+,1_-,\ldots,d'_+,d'_- \}$ such that \begin{equation}
	\tau = (1_+ \ \ 1_-) \ldots (d'_+ \ \ d'_-) \text{ and } \epsilon = (1_+ \ \ \ldots \ \ d'_+)(1_- \ \ \ldots \ \ d'_-).
\end{equation} The only non-vanishing number is $k_{d'} = 1$, thus $s(f) = -1$. 
	\item If the source is $(\mathbb{P}^{1}_\mathbb{C},\tilde{\theta})$, we label the fiber so that \begin{equation}
	\tau = id \text{ and } \epsilon = (1 \ \ \ldots \ \ d ).  
	\end{equation} The only non-vanishing number is $a_d = 1$, and $s(f) = 1$.
	\item Finally, if the source is $(\mathbb{P}^{1}_\mathbb{C},{\theta})$ and the degree $d = 2d'$ is even, we label the fiber so that \begin{equation}
	\tau = (1 \ \ d'+1) \ldots (d',2d) \text{ and } \epsilon = (1 \ \ \ldots \ \ d ).
	\end{equation} The only non-vanishing number is $b_d = 1$, and $s(f) = -1$.
\end{itemize} The lemma is proved.
\end{proof}

Note that even though there exists ramified Real covers corresponding to the partition $(d)$ with $d$ even, the signed Real Hurwitz number $\mathbb{R}H_{0,d}((d))$ with $d$ even vanishes. 

Not all the ramified Real covers described in Example \ref{exampleP1Fixed} truly contribute to the signed Real Hurwitz numbers. Indeed, the number \begin{equation}
\mathbb{R}H_{0,d}((d))
\end{equation} vanishes for $d$ even although there are two isomorphism classes of ramified Real covers of $(\mathbb{P}^1_\mathbb{C},\tilde{\theta})$ corresponding to the partition $(d)$. However, the signs are such that these isomorphism classes cancel each other. Thus, we say that a ramified Real cover of $(\mathbb{P}^1_\mathbb{C},\tilde{\theta},B_{\mathbb{P}^1_\mathbb{C}})$ is said to be \textit{contributing} if \begin{itemize}
	\item it is isomorphic to the cover described in Example \ref{exampleP1Fixed} for the partition $(d)$ with $d$ odd or $(d,d)$, 
	\item or it is a disjoint union of contributing ramified Real covers.
\end{itemize} One can restrict the counts to the contributing ramified Real covers of $(\mathbb{P}^1_\mathbb{C},\tilde{\theta},B_{\mathbb{P}^1_\mathbb{C}})$ without changing the value of the signed sums. The partitions of $d$ corresponding to the ramification profile of a ramified Real cover are necessarily even. The notion of contributing cover is used in Subsection \ref{subsectionFixedLocusGeneralCase}.

	\subsection{General case}
	\label{subsectionFixedLocusGeneralCase}

Let $(X,\sigma,B)$ be a marked connected Real Riemann surface, of genus $g$, whose real locus $X^\sigma$ is non-empty. It admits a \textit{canonical degeneration} obtained as follows. For each fixed circle $S$, choose a pair of conjugated circles homotopic to $S$ in $X^o$. Degenerating these pairs of conjugated circles leads to the union of copies of $(\mathbb{P}^1_\mathbb{C},\tilde{\theta},B_{\mathbb{P}^1_\mathbb{C}})$ as in Subsection \ref{subsectionFixedLocusP1} and a marked Real Riemann surface $(X',\sigma',B')$ with empty real locus. The latter can be either a doublet or a connected Real Riemann surface.

During the degeneration process, pairs of conjugated marked points are created and need to be ordered. For the rest of the subsection, we choose an ordering of those points without placing any constraint on it. If $(X',\sigma')$ is a doublet, it leads to a puncture class $a$. If $(X',\sigma')$ is a connected Real Riemann surface, we choose an admissible class $a$ for the rest of the subsection. The definitions will essentially not depend on the ordering of the marked point created during the degeneration or on the admissible class.\\

Let $f$ be a ramified Real cover of $(X,\sigma)$ unramified over $X \setminus B$. It degenerates canonically to give a ramified Real cover of $(X',\sigma',B')$ unramified over $X' \setminus B'$ and for each fixed circle $S$ in $X$ a ramified Real cover $f_S$ of $(\mathbb{P}^1_\mathbb{C},\tilde{\theta},B_{\mathbb{P}^1_\mathbb{C}})$ as in Subsection \ref{subsectionFixedLocusP1}. We define the sign of $f$ to be \begin{equation}
s(f) = \varepsilon \rho_{f'}(a) \prod_{S} s(f_S).
\label{equationSignGeneral}
\end{equation}

\begin{proposition}
Let $(X,\sigma,B)$ be as above and $\bm{\lambda} = (\bm{\lambda}_1,\ldots,\bm{\lambda}_n)$ a sequence of partitions of an integer $d$. The sum \begin{equation}
\sum_{[f]} \frac{s(f)}{\# Aut(f)}
\end{equation} over the isomorphism classes of ramified Real covers of $(X,\sigma)$ unramified over $X \setminus B$ with ramification profile $\bm{\lambda}$ equals \begin{equation}
\mathbb{R}H^\bullet_{g,d} (\bm{\lambda})
\end{equation}
\label{propositionSumFixedLocus} defined by replacing $\sigma$ by a fixed-point free involution as in Definition \ref{definitionRealHurwitz}.
\end{proposition}

\begin{proof}
We transform the sum as a sum over $f',(f_S)$. The automorphisms are related by \begin{equation}
\# Aut(f) = \# Aut(f') \prod_{S} \# Aut(f_S)
\end{equation} and a given set $[f'],([f_S])$ of isomorphism classes has \begin{equation}
\prod_{S} z_{\bm{\eta}_S}
\end{equation} pre-images $[f]$, where $\bm{\eta}_S$ is the partition of $d$ corresponding to the monodromy around $S$. For each $S$, the relation \begin{equation}
\sum_{[f_S]} \frac{s(f_S)}{\# Aut(f_S)} = \mathbb{R}H^\bullet_{0,d}(\bm{\eta}_S)
\end{equation} holds by Subsection \ref{subsectionFixedLocusP1}, where the sum is over the ramified Real covers whose ramification profile at a pair of conjugated points is $\eta_S$ and unramified elsewhere. Thus, depending on whether $(X',\sigma')$ is the doublet of a genus $h$ Riemann surface or a connected Real Riemann surface of genus $h$, we obtain \begin{equation}
\sum_{[f]} \frac{s(f)}{\# Aut(f)} = \sum_{\bm{\eta}} H^\bullet_{h,d,a}(\bm{\lambda},\bm{\eta}) \prod_{S} z_{\bm{\eta}_S} \mathbb{R}H^\bullet_{0,d}(\bm{\eta}_S)
\end{equation} or \begin{equation}
\sum_{[f]} \frac{s(f)}{\# Aut(f)} = \sum_{\bm{\eta}} \mathbb{R}H^\bullet_{h,d}(\bm{\lambda},\bm{\eta}) \prod_{S} z_{\bm{\eta}_S} \mathbb{R}H^\bullet_{0,d}(\bm{\eta}_S).
\end{equation} Using degeneration formulas of types (a) and (b), see Proposition \ref{propositionDegene}, the right-hand side is $\mathbb{R}H^\bullet_{g,d}(\bm{\lambda})$ in both cases.
\end{proof}

By Proposition \ref{propositionSumFixedLocus}, the signs introduced in Subsections \ref{subsectionDefinitionSign} and \ref{subsectionFixedLocusGeneralCase} are such that the signed sum of ramified Real covers does not depend on the involution $\sigma$ on the target $X$.\\

We now explain to what extent the sign $s(f)$ defined in (\ref{equationSignGeneral}) is independent of the choices made. Recall that they consist of the ordering of the marked points created during the canonical degeneration and, if $(X',\sigma')$ is connected, of an admissible class. 

Due to the cancellations occuring in $\mathbb{R}H_{0,d}((d))$ for $d$ even, we can restrict the discussion to the Real ramified covers $f$ of $(X,\sigma)$ for which every $f_S$ is the union of covers contributing to $\mathbb{R}H_{0,d}((d))$ for $d$ odd or $\mathbb{R}H_{0,2d}((d,d))$. Under this assumption, the partition at a marked point created during the degeneration must be even. Therefore, the choice of the ordering at these marked points does not modify the signs. If one of the partitions in $\bm{\lambda}$ is odd, the sum vanishes by Proposition \ref{propositionSumFixedLocus}. Otherwise, the signs are well-defined under the assumption above.

\begin{example}
Similarly to Example \ref{exampleHyperelliptic}, consider the compactification $E$ of the plane curve \begin{equation}
\{ y^2 = x (x-w)(x-1/\overline{w}) \}
\end{equation} with $|w| \neq 1$. The ramified cover $f : (x,y) \mapsto [1:x]$ has branch points $[1:0],[0:1],[1:w],[\overline{w}:1]$. If $\mathbb{P}^1_\mathbb{C}$ is endowed with the involution $\tilde{\theta}$ instead of $\theta$, there are still two anti-holomorphic involutions $\sigma_\pm$ on $E$ for which $f$ is a ramified Real cover. On the subset $\{(x,y) \in E \ | \ |x | = 1  \}$, which is the union of two circles, $\sigma_+$ induces the identity while $\sigma_-$ switches the two circles, see Figure \ref{figureCoverEllipticFixed}.

Let us compute the signs $s(f_+)$ and $s(f_-)$. They depend on the marking of the target. We choose the ordering in the two pairs of conjugated marked points in $E$ so that the positive marked points belong to the same half of $\mathbb{P}^1_\mathbb{C}$. The ramified Real covers $f_+,f_-$ are both contributing since $f_{S,+}$ is the union of twice \begin{equation}
id : (\mathbb{P}^1_\mathbb{C},\tilde{\theta}) \rightarrow (\mathbb{P}^1_\mathbb{C},\tilde{\theta})
\end{equation} and $f_{S,-}$ is the doublet of \begin{equation}
id : \mathbb{P}^1_\mathbb{C} \rightarrow \mathbb{P}^1_\mathbb{C}.
\end{equation} The Real Riemann surface $(X',\sigma')$ being a doublet in this case, the signs $s(f_+)$ and $s(f_-)$ are well-defined. By (\ref{equationSignGeneral}) and Example \ref{exampleP1Fixed}, they are \begin{equation}
s(f_\pm ) = \pm 1.
\end{equation} The hyperelliptic involution $(x,y) \mapsto (x,-y)$ is the only non-trivial automorphism of $f_+$ and $f_-$ so that we recover \begin{equation}
\mathbb{R}H_{0,2}((2),(2)) = \frac{1}{2} - \frac{1}{2} =  0,
\end{equation} see Example \ref{exampleHyperelliptic}.
\label{exampleHyperellipticFixed}
\end{example}

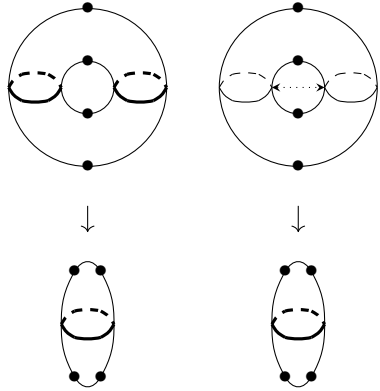
\begin{figure}[h!]
\centering
\begin{tikzpicture}[scale = 0.7]
		\node (0) at (-0.25, 1) {};
		\node (1) at (0.25, 1) {};
		\node (2) at (-0.25, -1) {};
		\node (3) at (0.25, -1) {};
		\node (4) at (-0.5, 0) {};
		\node (5) at (0.5, 0) {};
		\node (6) at (0, 6) {};
		\node (7) at (0, 5) {};
		\node (8) at (0, 3) {};
		\node (9) at (0, 4) {};
		\node (10) at (-0.5, 4.5) {};
		\node (11) at (0.5, 4.5) {};
		\node (12) at (-1.5, 4.5) {};
		\node (13) at (1.5, 4.5) {};
		\node (14) at (3.75, 1) {};
		\node (15) at (4.25, 1) {};
		\node (16) at (3.75, -1) {};
		\node (17) at (4.25, -1) {};
		\node (18) at (3.5, 0) {};
		\node (19) at (4.5, 0) {};
		\node (20) at (4, 6) {};
		\node (21) at (4, 5) {};
		\node (22) at (4, 3) {};
		\node (23) at (4, 4) {};
		\node (24) at (3.5, 4.5) {};
		\node (25) at (4.5, 4.5) {};
		\node (26) at (2.5, 4.5) {};
		\node (27) at (5.5, 4.5) {};
		\node (28) at (-0.5, 0) {};
		\node (29) at (0.5, 0) {};
		\node (30) at (-0.25, 0.25) {};
		\node (31) at (0.25, 0.25) {};
		\node (32) at (-0.25, -0.25) {};
		\node (33) at (0.25, -0.25) {};
		\node (34) at (3.5, 0) {};
		\node (35) at (4.5, 0) {};
		\node (36) at (3.75, 0.25) {};
		\node (37) at (4.25, 0.25) {};
		\node (38) at (3.75, -0.25) {};
		\node (39) at (4.25, -0.25) {};
		\node (40) at (-1.5, 4.5) {};
		\node (41) at (-0.5, 4.5) {};
		\node (42) at (-1.25, 4.75) {};
		\node (43) at (-0.75, 4.75) {};
		\node (44) at (-1.25, 4.25) {};
		\node (45) at (-0.75, 4.25) {};
		\node (46) at (0.5, 4.5) {};
		\node (47) at (1.5, 4.5) {};
		\node (48) at (0.75, 4.75) {};
		\node (49) at (1.25, 4.75) {};
		\node (50) at (0.75, 4.25) {};
		\node (51) at (1.25, 4.25) {};
		\node (52) at (2.5, 4.5) {};
		\node (53) at (3.5, 4.5) {};
		\node (54) at (2.75, 4.75) {};
		\node (55) at (3.25, 4.75) {};
		\node (56) at (2.75, 4.25) {};
		\node (57) at (3.25, 4.25) {};
		\node (58) at (4.5, 4.5) {};
		\node (59) at (5.5, 4.5) {};
		\node (60) at (4.75, 4.75) {};
		\node (61) at (5.25, 4.75) {};
		\node (62) at (4.75, 4.25) {};
		\node (63) at (5.25, 4.25) {};
		\node (64) at (0, 6) {$\bullet$};
		\node (65) at (0, 5) {$\bullet$};
		\node (66) at (0, 4) {$\bullet$};
		\node (67) at (0, 3) {$\bullet$};
		\node (68) at (4, 6) {$\bullet$};
		\node (69) at (4, 5) {$\bullet$};
		\node (70) at (4, 4) {$\bullet$};
		\node (71) at (4, 3) {$\bullet$};
		\node (72) at (-0.25, 1) {$\bullet$};
		\node (73) at (0.25, 1) {$\bullet$};
		\node (74) at (-0.25, -1) {$\bullet$};
		\node (75) at (0.25, -1) {$\bullet$};
		\node (76) at (3.75, 1) {$\bullet$};
		\node (77) at (4.25, 1) {$\bullet$};
		\node (78) at (3.75, -1) {$\bullet$};
		\node (79) at (4.25, -1) {$\bullet$};
		\node (80) at (0, 2.25) {};
		\node (81) at (0, 1.75) {};
		\node (82) at (4, 2.25) {};
		\node (83) at (4, 1.75) {};
		\draw [bend left=15] (4.center) to (0.center);
		\draw [bend left=60, looseness=1.50] (0.center) to (1.center);
		\draw [bend left=15] (1.center) to (5.center);
		\draw [bend left=15] (5.center) to (3.center);
		\draw [bend left=60, looseness=1.50] (3.center) to (2.center);
		\draw [bend left=15] (2.center) to (4.center);
		\draw [bend left=45] (10.center) to (7.center);
		\draw [bend left=45] (7.center) to (11.center);
		\draw [bend left=45] (11.center) to (9.center);
		\draw [bend left=45] (9.center) to (10.center);
		\draw [bend left=45] (12.center) to (6.center);
		\draw [bend left=45] (6.center) to (13.center);
		\draw [bend left=45] (13.center) to (8.center);
		\draw [bend left=45] (8.center) to (12.center);
		\draw [bend left=15] (18.center) to (14.center);
		\draw [bend left=60, looseness=1.50] (14.center) to (15.center);
		\draw [bend left=15] (15.center) to (19.center);
		\draw [bend left=15] (19.center) to (17.center);
		\draw [bend left=60, looseness=1.50] (17.center) to (16.center);
		\draw [bend left=15] (16.center) to (18.center);
		\draw [bend left=45] (24.center) to (21.center);
		\draw [bend left=45] (21.center) to (25.center);
		\draw [bend left=45] (25.center) to (23.center);
		\draw [bend left=45] (23.center) to (24.center);
		\draw [bend left=45] (26.center) to (20.center);
		\draw [bend left=45] (20.center) to (27.center);
		\draw [bend left=45] (27.center) to (22.center);
		\draw [bend left=45] (22.center) to (26.center);
		\draw [bend left=15, looseness=1.25, dashed, very thick] (28.center) to (30.center);
		\draw [bend left=15, looseness=0.75, dashed, very thick] (30.center) to (31.center);
		\draw [bend left=15, looseness=1.25, dashed, very thick] (31.center) to (29.center);
		\draw [bend left=15, looseness=1.25, very thick] (29.center) to (33.center);
		\draw [bend left=15, looseness=0.75, very thick] (33.center) to (32.center);
		\draw [bend left=15, looseness=1.25, very thick] (32.center) to (28.center);
		\draw [bend left=15, looseness=1.25, dashed, very thick] (34.center) to (36.center);
		\draw [bend left=15, looseness=0.75, dashed, very thick] (36.center) to (37.center);
		\draw [bend left=15, looseness=1.25, dashed, very thick] (37.center) to (35.center);
		\draw [bend left=15, looseness=1.25, very thick] (35.center) to (39.center);
		\draw [bend left=15, looseness=0.75, very thick] (39.center) to (38.center);
		\draw [bend left=15, looseness=1.25, very thick] (38.center) to (34.center);
		\draw [bend left=15, looseness=1.25, dashed, very thick] (40.center) to (42.center);
		\draw [bend left=15, looseness=0.75, dashed, very thick] (42.center) to (43.center);
		\draw [bend left=15, looseness=1.25, dashed, very thick] (43.center) to (41.center);
		\draw [bend left=15, looseness=1.25, very thick] (41.center) to (45.center);
		\draw [bend left=15, looseness=0.75, very thick] (45.center) to (44.center);
		\draw [bend left=15, looseness=1.25, very thick] (44.center) to (40.center);
		\draw [bend left=15, looseness=1.25, dashed, very thick] (46.center) to (48.center);
		\draw [bend left=15, looseness=0.75, dashed, very thick] (48.center) to (49.center);
		\draw [bend left=15, looseness=1.25, dashed, very thick] (49.center) to (47.center);
		\draw [bend left=15, looseness=1.25, very thick] (47.center) to (51.center);
		\draw [bend left=15, looseness=0.75, very thick] (51.center) to (50.center);
		\draw [bend left=15, looseness=1.25, very thick] (50.center) to (46.center);
		\draw [bend left=15, looseness=1.25, dashed] (52.center) to (54.center);
		\draw [bend left=15, looseness=0.75, dashed] (54.center) to (55.center);
		\draw [bend left=15, looseness=1.25, dashed] (55.center) to (53.center);
		\draw [bend left=15, looseness=1.25] (53.center) to (57.center);
		\draw [bend left=15, looseness=0.75] (57.center) to (56.center);
		\draw [bend left=15, looseness=1.25] (56.center) to (52.center);
		\draw [bend left=15, looseness=1.25, dashed] (58.center) to (60.center);
		\draw [bend left=15, looseness=0.75, dashed] (60.center) to (61.center);
		\draw [bend left=15, looseness=1.25, dashed] (61.center) to (59.center);
		\draw [bend left=15, looseness=1.25] (59.center) to (63.center);
		\draw [bend left=15, looseness=0.75] (63.center) to (62.center);
		\draw [bend left=15, looseness=1.25] (62.center) to (58.center);
		\draw [stealth-stealth, dotted] (53.center) to (58.center);
		\draw [->] (80.center) to (81.center);
		\draw [->] (82.center) to (83.center);
\end{tikzpicture}
\caption{Description of the two ramified Real covers of Example \ref{exampleHyperellipticFixed}. The involutions $\sigma_+$ and and $\sigma_-$ correspond respectively to left-hand side and right-hand side.}
\label{figureCoverEllipticFixed}
\end{figure}
	
\printbibliography

\end{document}